\documentclass[a4paper,11pt]{article}

\input{styleart.sty}
\usepackage{xcolor}
\usepackage{leftindex}
\def\<#1>{\mathinner{\langle#1\rangle}}

\title{Fields interpretable in the free group}
\date{}
\author{Rizos Sklinos\thanks{This material is based upon work partially supported by the National Science Foundation under Grant No. 1953784 and by the National Science Foundation of China under Grant No. 12350610234. \\ Mathematics Subject Classification (2010): 03C60, 20F67}}

\begin{document}

\maketitle

\begin{abstract} 
We prove that no infinite field is interpretable in the first-order theory of nonabelian free groups. We also obtain a characterization of Abelian groups interpretable in this theory. 
\end{abstract}

\section{Introduction} 
Tarski in 1946 asked whether the first-order theory of nonabelian free groups is complete. In seminal works Sela \cite{MR2238945} and Kharlampovich-Myasnikov \cite{MR2293770} 
answered the question positively. Moreover, Sela proved that this theory is stable \cite{MR3034289}. In one of the first papers on the subject, after the solution of Tarski's problem, 
Pillay conjectured that no infinite field is interpretable in this theory and he coined it the name "the first-order theory of the free group" \cite{MR2400726}. In this paper we confirm the conjecture. 

\begin{thmIntro}\label{Major1}
The first-order theory of the free group does not interpret an infinite field.
\end{thmIntro}

Since the first-order theory of the free group does not have the finite cover property \cite{MR3846335}, it is enough to prove that no infinite field is interpretable in a particular model. More precisely we prove:

\begin{thmIntro}\label{Major2}
Let $\F$ be a nonabelian free group and $X$ be a set interpretable in $\F$. Then either $X$ is internal to a family of centralizers of nontrivial elements or it cannot be given definably the structure of an Abelian group. 
\end{thmIntro}

Theorem \ref{Major2} above implies that no infinite field is interpretable in a nonabelian free group. Indeed, by \cite{MR2039343}, a set internal to a family of $1$-based sets is $1$-based itself and an 
infinite $1$-based set cannot be given definably the structure of a field \cite{MR1782132}. 
In \cite{MR4030181}, on the course of proving that no infinite field is definable in the first-order theory of the free group, we proved that centralizers of nontrivial elements are $1$-based, hence 
no set internal to a family of centralizers can define an infinite field. 

Theorem \ref{Major1} now follows easily, since all axioms, for a triple of formulas $(\phi, \odot, \oplus)$, that define a field are first-order expressible and the first-order theory of the free group eliminates the "there exists infinitely many" quantifier in all imaginary sorts. Indeed, if an infinite field were interpretable in some model of the first-order theory of the free group, we could transfer this interpretation to a standard model, i.e. a nonabelian free group, contradicting Theorem \ref{Major2}. In fact, Theorem \ref{Major2} provides a characterization of Abelian groups interpretable in the free group.  

\begin{thmIntro}
Let $(X, \oplus)$ be an Abelian group interpretable in the free group. Then $X$ is internal to a family of centralizers of nontrivial elements. 
\end{thmIntro}

We note that the centralizer $(C_\mathbb{M}(b),\cdot)$ of a nontrivial element $b$ of the monster model $\mathbb{M}$ is elementarily equivalent to $(\Z, +)$.

As one might expect a crucial tool for passing from the nondefinability to the noninterpretability of infinite fields is some form of elimination of imaginaries. The first-order theory of the free group (weakly) eliminates imaginaries after adding some sorts \cite{SelaIm}. 

Interestingly the first-order theory of the free group is $n$-ample for all $n$ \cite{AmpleOK} \cite{AmpleRizos}, hence our result yields the first example of an ample theory that interprets an infinite group but not an infinite field. An example of an ample theory that does not interpret an infinite field 
had been constructed in \cite{AmpleDividing}. The latter theory does not interpret an infinite group, but still forking independence is nontrivial.   
\ \\ \\
{\bf Acknowledgements:} I would like to express my heartfelt gratitude to the referee for their careful reading and insightful comments on this work. Their attention to detail and thoughtful suggestions have greatly enhanced the quality of this paper. 

I would also like to thank Zlil Sela for patiently explaining several constructions   intertwined throughout his work on the Tarski problem.  


\section{Overview}

In the first-order theory of the free group any first-order formula is equivalent to a boolean combination of $\forall\exists$-formulas. Despite 
the quantifier elimination it is hard to understand these basic definable sets. Still some special classes of formulas are easy to understand and 
even enlightening. Here is an argument that allows us to gain some intuition. A key result that makes the proof work is the following refinement of the elementary equivalence of nonabelian free groups. 

\begin{fact}[Sela/Kharlampovich-Myasnikov]
The following chain of free groups under the natural embeddings is elementary:
$$\F_2<\F_3<\ldots<\F_n<\ldots$$
\end{fact} 

\begin{propositionIntro}[The noncommutativity argument]\label{NonCommut}
Let $X:=\phi(\bar{x})$ be a definable set over $\F_n$. Suppose $\phi(\F_n)\neq\phi(\F_{n+1})$. 
Then $X$ cannot be given definably the structure of an Abelian group. 
\end{propositionIntro}
\begin{proof}
Suppose, for the sake of contradiction, that $(X,\odot)$ (where $\odot:=\psi(\bar{x},\bar{y},\bar{z})$ 
restricts to a group operation on $X$) is an Abelian group. We work in $\F_{\omega}=\langle e_1, e_2, \ldots, e_n, \ldots\rangle$. Let $\bar{a}\in X$ be a tuple 
that $X$ gains in $\F_{n+1}$, then $\bar{a}\in \F_{n+1}\setminus \F_n$. We fix a tuple of words 
$\bar{a}(x_1,\ldots,x_{n+1})$ in variables $x_1,\ldots,x_{n+1}$, such that 
$\bar{a}(e_1,\ldots,e_{n+1})=\bar{a}$. It is easy to see that $\bar{a}':=\bar{a}(e_1,\ldots,e_n,e_{n+2})$ 
belongs to $X$ and $\bar{a}'\in \F_{n+2}\setminus \F_{n+1}$. We then have that  
$\bar{a}\odot\bar{a}'$ is a tuple in $\F_{n+2}\setminus \F_{n+1}$, if not then $\bar{a}'$ is a 
tuple in $\F_{n+1}$, a contradiction. So, let $\bar{a}\odot\bar{a}'=\bar{w}(e_1,\ldots,e_{n+1},e_{n+2})$. We 
consider the automorphism, $f\in Aut(\F_{n+2}/\F_n)$, exchanging $e_{n+1}$ with $e_{n+2}$, it is clear that 
$f(\bar{a})=\bar{a}'$, $f(\bar{a}')=\bar{a}$ and $f$ fixes $\odot$. Thus, we have $f(\bar{a})\odot f(\bar{a}')=
\bar{a}'\odot\bar{a}$ and as $(X,\odot)$ is Abelian we have $\bar{a}'\odot\bar{a}=\bar{w}(e_1,\ldots,e_{n+1},e_{n+2})$. 
But, $f(\bar{a})\odot f(\bar{a}')=f(\bar{w}(e_1,\ldots,e_{n+1},e_{n+2}))=\bar{w}(e_1,\ldots,e_{n+2},e_{n+1})$ and 
as $\bar{w}(e_1,\ldots,e_{n+1},e_{n+2})$ is in $\F_{n+2}\setminus \F_{n+1}$ we have that 
$\bar{w}(e_1,\ldots,e_{n+2},e_{n+1})\neq \bar{w}(e_1,\ldots,$ $e_{n+1},e_{n+2})$, a contradiction. 
\end{proof}

Of course, when we talk about interpretable sets we also have to deal with imaginary sorts. Sela \cite{SelaIm} proved a (weak) elimination of imaginaries up to adding some sorts. Essentially these sorts are: A sort 
for cosets of centralizers with conjugation, a sort for double cosets of centralizers with conjugation and refinements of those (see Section \ref{MT}). 

There are two major predicaments in generalizing the above argument. First, there is a class of formulas that define Abelian groups, namely centralizers of nontrivial elements. For this class of 
formulas we need to find a different argument. Second, even for formulas that gain an element, like in the above proposition, the noncommutativity argument fails for conjugacy classes. Consider, 
for a simple counterexample, that the sum is given by the word $e_1e_2$, then exchanging $e_1$ with $e_2$ gives $e_2e_1$ which is in the same conjugacy class as $e_1e_2$. In this latter case 
we can fix the problem by taking more than two summands. Permuting the primitive elements that define the different summands will increase the number of conjugacy classes of the resulting sum and thus obtain a contradiction (see \cite{MR3426232}). 

The way we deal with centralizers is by proving that they are pure groups, i.e. any definable subset of a cartesian power $C_{\F}(a)\times C_{\F}(a)\times\ldots\times C_{\F}(a)$ is definable by multiplication alone (see \cite[Corollary 6.28]{MR4030181}).  
By a theorem of Wagner \cite{MR2039343}, the latter result takes care of all interpretable sets internal to a family of centralizers. 

\paragraph{Strategy:}

We first replace the interpretable set by its {\em Diophantine envelope}. The Diophantine envelope is a finite family of {\em towers}. For every such tower there is a Diophantine definable set naturally corresponding to it. 

The Diophantine envelope has the following properties. The union of sets that correspond to the 
Diophantine envelope contain the original interpretable set and moreover for particular sequences of elements, called {\em test sequences}, one may decide whether they eventually belong to the set or not. At this point the proof splits. 
If the only used part of the towers is their {\em Abelian pouch}, then the set is internal to a set of centralizers of nontrivial elements hence it cannot be the domain of a field. If not, then we can proceed applying a tweaked version of the noncommutativity argument that we describe next.

From a tower in the Diophantine envelope, 
we construct an {\em $N$-multiplet of towers} for which we can decide which test sequences belong to the $N$ cartesian product of the original set. The $N$-multiplet of towers still has the structure of a tower and it can be 
seen as a {\em star of groups} with $N$ rays.

We next apply the {\em implicit function theorem} on the $N$-multiplet of towers and the 
formula that defines the $N$-summation. The implicit function theorem gives an element in the $N$-multiplet of towers that its image under certain test sequences is the value of the $N$-summation. 
Permuting the rays of the $N$-multiplet of towers is the analogue of permuting the primitive elements in the noncommutativity argument. Hence choosing $N$ large enough we get enough different 
values for the $N$-summation and this finally shows that the interpretable set cannot be given an Abelian group structure. 

\paragraph{Paper structure:}

In Section \ref{MT} we explain some model theoretic concepts, namely internality, definable closure, imaginaries and their relative elimination in the first-order theory of the free group. 

In Section \ref{StarofGroups} we define a special type of graph of groups that we call a star of groups because the underlying graph is a star. We prove theorems about the basic imaginary sorts in stars of groups. In particular we show how the equivalence classes change as we permute the rays.  

Section \ref{Towers} contains the construction of towers. Towers, their multiplets and closures are the main tools to partially understand definable sets. 

In Section \ref{TestSequences} we define the notion of a test sequence. We state Sela's Diophantine envelope theorem, the three levels of towers theorem and generalizations of the implicit function theorem. 

Finally, in the last section we prove the main theorem of the paper as well as some special cases that are free of certain technicalities.  


\section{Some model theory}\label{MT}

\subsection{Imaginaries}\label{Imaginariess}

We fix a first-order structure $\mathcal{M}$ and we are interested in the collection of definable sets in $\mathcal{M}$, i.e. all subsets 
of some cartesian power of $\mathcal{M}$ which are the solution sets of first-order formulas. We motivate the definition 
of imaginaries with the following question: Suppose $X$ is a definable set in $\mathcal{M}$, is there a canonical way to 
define $X$, i.e. is there a tuple $\bar{b}$ and a formula $\psi(\bar{x},\bar{y})$ such that $\psi(\mathcal{M},\bar{b})=X$ 
but for any other $\bar{b}'\neq \bar{b}$, $\psi(\mathcal{M},\bar{b}')\neq X$? 

To give a positive answer to the above mentioned question one has to move to a multi-sorted expansion of 
$\mathcal{M}$ called $\mathcal{M}^{eq}$. This expansion is constructed from $\mathcal{M}$ 
by adding a new sort for each $\emptyset$-definable equivalence relation, $E(\bar{x},\bar{y})$, together 
with a class function $f_E:M^n\rightarrow M_E$, where $M_E$ (the domain of the new sort corresponding to $E$) 
is the set of all $E$-equivalence classes. The elements in these new sorts are called {\em imaginaries}. 
In $\mathcal{M}^{eq}$, one can assign to each definable set a canonical parameter in the sense discussed above. 
Moreover, every formula in the multi-sorted language corresponds to a formula in the original language in the following sense (see \cite[Chapter 1, Section 1]{MR1429864} for more details). 

\begin{fact}
Let $\phi(x_1, x_2, \ldots, x_k)$ be a formula in the multi-sorted language $\mathcal{L}^{eq}$, where $x_i$ is a variable of the sort $S_{E_i}$. 
Then there exists an $\mathcal{L}$-formula $\psi(\bar{y}_1, \bar{y}_2, \ldots, \bar{y}_k)$ such that:
$$\mathcal{M}^{eq}\models \forall \bar{y}_1, \bar{y}_2, \ldots, \bar{y}_k(\phi(f_{E_1}(\bar{y}_1), f_{E_2}(\bar{y}_2), \ldots, f_{E_k}(\bar{y}_k))\leftrightarrow \psi(\bar{y}_1, \bar{y}_2, \ldots, \bar{y}_k)) $$ 
\end{fact}

We will call the formula $\psi(\bar{y})$ in the above fact, the {\em real formula corresponding to $\phi(\bar{x})$}.

A set is {\em interpretable} in $\mathcal{M}$ if it is definable up to a definable equivalence relation. Actually, understanding interpretable sets in $\mathcal{M}$ is equivalent to understanding definable sets in $\mathcal{M}^{eq}$.

\begin{fact}
A set is interpretable in $\mathcal{M}$ if and only if it is definable in $\mathcal{M}^{eq}$.
\end{fact}

We now specialize to the first-order theory of the free group. Following Sela \cite{SelaIm} we define some families of equivalence relations that will be used to eliminate the rest. 

\begin{definition}\label{Imaginaries}
Let $\F$ be a nonabelian free group. The following equivalence relations in $\F$ are called elementary.
\begin{itemize}
 \item[$(i)$] (Cosets with conjugation) $ _cE_2((a_1,b_1),(a_2,b_2))$ if either $b_1=b_2=1$ or $b_1\neq 1$ and there exists $g\in \F$ such that 
$$(a_1\cdot C_{\F}(b_1))^g=a_2\cdot C_{\F}(b_2) $$ 
 \item[$(ii)$] (Double cosets with conjugation) $_cE_3((a_1,b_1,c_1),(a_2,b_2,c_2))$ if either 
$a_1=a_2=1$ or $c_1=c_2=1$ or $a_1,c_1\neq 1$ and there exists $g\in\F$ such that 
  $$(C_{\F}(a_1) \cdot b_1 \cdot C_{\F}(c_1))^g= C_{\F}(a_1) \cdot b_2 \cdot C_{\F}(c_1)$$
\end{itemize}

\end{definition}

We denote classes that correspond to $_cE_2$ by $_c[(x,y)]_2$ and classes that correspond to $_cE_3$ by $_c[(x,y,z)]_3$. Note that in each elementary equivalence relation we have a class that has different nature from the rest. We will call these classes trivial. The trivial classes are $_c[(1,1)]_2=\{(a,1) \ | \ a\in \F\}$ in $_cE_2$ and $_c[(1,1,1)]_3=\{(1,a,b)\ | \ a,b \in \F\}\cup \{(a,b,1)\ | \ a,b \in \F\} $ in $_cE_3$. 


Finer equivalence relations may also be definable (and not eliminable by the elementary ones). We give some examples. If $A$ is an Abelian group and $m$ a natural number, we denote by $A^m$ the subgroup $\{a^m \ | \ a\in A\}$. In addition, if $H$ is a subgroup of $G$ and $g\in G$, then $H^g$ denotes $gHg^{-1}$. 

\begin{example}\label{BasicExamples}\
\begin{itemize}
\item {\em Commuting conjugation}, denoted $E_1(a,b)$, is defined as follows: either $a=b=1$ or $a,b\neq 1$ and there is $g\in \F$ such that $a^g$ commutes with $b$. Commuting conjugation is finer than coset conjugation in the sense that if $E_1(a,b)$, then $_cE_2((1,a),(1,b))$. 
\item {\em Conjugation}, denoted by $E_1^1$, is finer than commuting conjugation. 
\item The family of {\em cosets of powers with conjugation}, denoted $_cE_2^m((a_1,b_1), (a_2,b_2))$, where $m$ is an integer, is defined as follows: either $b_1=b_2=1$ or $b_1,b_2\neq 1$ and there exists $g\in\F$ such that 
$(a_1\cdot C_{\F}(b_1)^m)^g=a_2\cdot C_{\F}(b_2)^m$. If $_cE_{2}^m((a_1,b_1),(a_2,b_2))$, then $_cE_2((a_1,b_1),(a_2,b_2))$.
\item The family of {\em cosets of powers}, denoted $E_2^m((a_1,b_1),(a_2,b_2))$, where $m$ is an integer, is defined as follows: either $b_1=b_2=1$ or $b_1,b_2\neq 1$ and 
$a_1\cdot C_{\F}(b_1)^m=a_2\cdot C_{\F}(b_2)^m$. If $E_{2}^m((a_1,b_1),(a_2,b_2))$, then $_cE_2^m(a_1,b_1),(a_2,b_2))$. It is a finer equivalence relation than that of cosets of powers with conjugation (with the same $m$).
\item Similarly, {\em double cosets of powers with conjugation}, and {\em double cosets of powers} denoted $^m _cE_3^n((a_1,b_1,c_1),(a_2,b_2,c_2))$ and $^mE_3^n((a_1,b_1,c_1),(a_2,b_2,c_2))$ respectively, are finer than double cosets with conjugation.
\item The family of {\em bilateral translates with conjugation}, denoted $^m _cE_4^n((a_1,b_1),(a_2,b_2))$, where $m,n$ are integers, is defined as follows: either $b_1=b_2=1$ or $b_1,b_2\neq 1$ and there is $g\in\F$ such that $[b_1^g,b_2]=1$ and there are  $\gamma\in C_{\F}(b_1)$ and $\delta\in C_{\F}(b_2)$ such that $(\gamma^na_1\gamma^m)^g=\delta^na_2\delta^m$. If $^m _cE_4^n((a_1,b_1),(a_2,b_2))$, then $^m _cE_3^n((a_1,b_1,a_1),(a_2,b_2,a_2))$. It is a finer equivalence relation than that of double cosets of powers with conjugation (with the same corresponding $m$ and $n$).   
\item The family of {\em bilateral translates}, denoted $^m E_4^n((a_1,b_1),(a_2,b_2))$, where $m,n$ are integers, is defined as follows: either $b_1=b_2=1$ or $b_1,b_2\neq 1$ and $[b_1,b_2]=1$, and there is $\gamma\in C_{\F}(b_1)$ such that $\gamma^na_1\gamma^m=a_2$. If $^m E_4^n((a_1,b_1),(a_2,b_2))$, then $^mE_3^n((a_1,b_1,a_1),(a_2,b_2,a_2))$.
\item We can generalize the equivalence relations that involve conjugation, i.e. $E_1,\ E_1^1, \ ^m _cE_2^n,$ $ \ ^m _cE_3^n,$ $^m _cE_4^n$ by considering finite tuples and partitioning them in singletons, couples, and triplets in a way that each such subtuple corresponds to some relation from this list. In addition, we require that the conjugating element is identical for all relations. We call each relation in this family of relations generalized equivalence relations and denote them by $E_{\{g\}}$ without any other qualitative index. These equivalence relations are finer than the corresponding product of relations that take part in their definitions where conjugation may use different elements. 

We give an example. Suppose $\bar x=x_1 x_2x_3$ and $\bar x$ splits as follows: $x_1$ corresponds to $E_1$, and $(x_2, x_3)$ corresponds to $_cE_2$. We define $E_{\{g\}}(\bar x, \bar y)$ if there exists $g$ such that [either $x_1=y_1=1$ or $x_1,y_1\neq 1$ and $[x_1^g,y_1]=1$] and [either $x_3=y_3=1$ or $x_3,y_3\neq 1$ and $(x_2 C(x_3))^g=y_2C(y_3)$]. If $E_{\{g\}}((a_1, a_2, a_3), (b_1, b_2, b_3))$, then $E_1\times _cE_2((a_1, a_2, a_3), (b_1, b_2, b_3))$. Also note that the latter relation can be eliminated by its factors, i.e. $E_1, \ _cE_2$.  

\end{itemize}

\end{example}

A first-order formula, $\phi(\bar x)$, in the language of groups is {\em positive existential} if it has the following form $\exists \bar y\bigl(\bigvee_{i=1}^n (w_1^i(\bar y, \bar x)=1\land\ldots\land w_{k_i}^i(\bar y, \bar x)=1)\bigr)$. Note that all the above equivalence relations can be expressed by a positive existential formula after removing, in each case, the trivial class. For example for the equivalence relation of cosets, $E_2$, we have in $\F^2\setminus \{(a,1) \ | \ a\in \F\}$,  $E_2(x_1, x_2, y_1, y_2):=\exists y \bigl([x_2, y_2]=1\land [y,x_2]=1 \land x_1y=x_2\bigr)$. 

\begin{definition}\label{BasicER}
Let $\F$ be a nonabelian free group. An equivalence relation is called basic if it is one of the following: equality (on tuples of elements), $E_2^m,\ ^mE_3^n,\ ^mE_4^n$ for integers $m,n$, and $E_{\{g\}}$ for any generalized equivalence relation.
\end{definition}

All basic equivalence relations make sense in any commutative transitive group. In particular, they make sense in limit groups. 

\begin{remark}
It is not hard to see that $(a_1, b_1)$, $(a_2,b_2)$ are $_c E_2^m$-equivalent in a free group $\F$, if and only if there exists $g\in \F$ such that $[b_1^g,b_2]=1$ and $a_1^g\cdot C_{\F}(b_2)^m=a_2\cdot C_{\F}(b_2)^m$. Similarly, $(a_1,b_1,c_1)$, $(a_2,b_2,c_2)$ are $_c^m E_3^n$-equivalent if and only if there exists $g\in \F$ such that $[a_1^g,a_2]=1$ and $[c_1^g,c_2]=1$ and $C_{\F}(a_2)^m\cdot b_1^g\cdot C_{\F}(c_2)^n=C_{\F}(a_2)^m\cdot b_2\cdot C_{\F}(c_2)^n$. The same holds for limit groups. 
\end{remark}

Sela proved the following theorem concerning imaginaries in nonabelian free groups (see \cite[Theorem 4.4]{SelaIm}).

\begin{theorem}\label{Elim}
Let $\F$ be a nonabelian free group. Let $E(\bar{x},\bar{y})$ be a definable equivalence relation in $\F$, with $\abs{\bar{x}}=m$.
Then there exist $k,\ell<\omega$ and a definable relation:
$$R_E \subseteq \F^m \times S_1(\F) \times \ldots \times S_{\ell}(\F)$$
such that:
\begin{itemize}
 \item[(i)] each $S_i(\F)$ is one of the basic sorts;
 \item[(ii)] for each $\bar{a}\in \F^m$ , $\abs{\{\bar z\ | \ R_E(\bar{a},\bar{z})\}}$ is uniformly bounded (i.e. the bound does not depend on $\bar{a}$) and non-empty;
 \item[(iii)] $\forall\bar{z}(R_E(\bar{a},\bar{z})\leftrightarrow R_E(\bar{b},\bar{z}))$ if and only if $E(\bar{a},\bar{b})$.
\end{itemize}
\end{theorem}

\begin{remark}
Originally Sela did not include the family of generalized equivalence relations (except standard conjugation) and the family of bilateral translates. The omission of these equivalence relations do not cause any serious problem, as they can be treated similarly to the rest. Nevertheless, we would like to thank the referee for bringing this gap to our attention. The above corrected statement has been communicated to the author by Sela. 

In the light of these new sorts the results in \cite{MR3846335}, \cite{MR3426232} and \cite{AmpleRizos} are also affected. The first paper is fixed by Lemma \ref{OrbitsCosets}, Corollary \ref{CorrectNFCP} and Proposition \ref{CorrectNFCP2}. The results of the second paper are implied by the results of the present paper. Finally, in the third paper we give arbitrarily large number of images of an equivalence class by using free group automorphisms. Since every equivalence relation is finer than $_cE_2$ or $_cE_3$, it is enough to prove the result for these two relations. In this case Theorem 2.13 in the same paper, yields already the more general result we need.
\end{remark}

We will occasionally, for convenience, denote by $R_E(x, \bar{z})$ the same relation as in the theorem above, but with $x$ a variable in the sort $S_E$. 
Formally the relation with the imaginary variable is defined by $\exists \bar y (R_E(\bar y, \bar{z})\land x=f_E(\bar{y}))$. 
Note that for $\bar a\in \F^m$, the solution sets of $R_E(\bar{a},\bar z)$ and $R_E([\bar a]_E, \bar z)$ in $\F^{eq}$ are the same. We will be calling $R_E(x,\bar{z})$ the {\em elimination relation}.


\subsection{Internality \& One-Basedness}
For this subsection we fix a stable (complete) first-order theory $T$ and a monster model $\mathbb{M}$ of $T$ (see \cite[Section 6.1]{MR2908005} for the definition of a monster model).   

\begin{definition}
Let $B\subset \mathbb{M}$. Then $\bar{a}$ is in the definable closure of $B$, denoted $dcl(B)$, if it is fixed by any automorphism $f\in Aut(\mathbb{M}/B)$, 
i.e. any automorphism of $\mathbb{M}$ that fixes $B$ pointwise.
\end{definition}

For example, for any two elements $a,b$ in a free group $\F$, any word $w(a,b)$ is in $dcl(a,b)$, in particular the subgroup generated by $a,b$,
$\langle a, b\rangle_{\F}$, is contained in $dcl(a,b)$. A more abstract example is the following:

\begin{lemma}\label{GeometricElimination}
Let $E$ be a $\emptyset$-definable equivalence relation in a nonabelian free group $\F$. Let $a$ be an element in the sort $S_E(\F)$. Let $\bar{b}_1, \ldots,\bar{b}_l$ be the solution set of $R_E(a, \bar{y})$ (also denoted $R_E(a, \F^{eq})$). Then 
$a$ is in $dcl^{eq}(\bar{b}_1, \ldots, \bar{b}_l)$.
\end{lemma}
\begin{proof}
Indeed, if an automorphism $f$ of $\mathbb{M}^{eq}$, where $\mathbb{M}$ is the monster model of the first-order theory of the free group, fixes $\bar{b}_1, \ldots, \bar{b}_l$, then it fixes $a$, as, by Theorem \ref{Elim}(iii), it is the unique solution of the formula $R_E(x,\bar{b}_1)\land\ldots\land R_E(x,\bar{b}_l)\land \forall \bar y \bigl(\bigwedge\limits_{i=1}^{l} \bar y\neq \bar b_i \rightarrow \lnot R_E(x,\bar y)\bigr)$.   
\end{proof}

Let $\Sigma$ be a family of partial types (over various small subsets of $\mathbb{M}$). Suppose $A$ is a subset of $\mathbb{M}$. We will say that $\Sigma$ is $A$-invariant if for any $f\in Aut(\mathbb{M}/A)$ and any $p\in\Sigma$ we have $f(p)\in \Sigma$.  

\begin{definition}
Let $\pi(\bar{x})$ be a partial type over $A\subset\mathbb{M}$ and $\Sigma$ be a $\emptyset$-invariant family of partial types. 
Then $\pi(\bar{x})$ is $\Sigma$-internal if for every realization $\bar a$ of $\pi(\bar x)$ there exist $B\subset\mathbb{M}$, for which $tp(\bar a / B)$ does not fork over $A$, and $\bar b$ realizing types in $\Sigma$ based on $B$, such that $\bar{a}\in dcl(B,\bar{b})$. 
\end{definition}

Straight from  the definition of forking independence we get (see \cite[Definition 2.19]{MR1429864}).  

\begin{fact}\label{Forking}
Let $\bar a$ be a tuple in $\mathbb M$. Then $tp(\bar a / B)$ does not fork over $B$, for any $B\subset \mathbb{M}$.
\end{fact}

Since, in the sequel, we will only use internality with $A=B$ we will not explain the notion of forking. In any case, the interested reader is referred to \cite[Chapter 1, Section 2]{MR1429864}. 

Intuitively, internality can be thought of as an abstract coordination notion. The (solution set of the) partial type $\pi$ is coordinated by (solution sets of partial types in) $\Sigma$.    

Finally, the following fact will be important for proving the main result of this paper (see \cite[Corollary 12]{MR2039343}). 

\begin{fact}[Wagner]
Let $\pi(\bar{x})$ be a $\Sigma$-internal (partial) type. If every type in $\Sigma$ is $1$-based, then $\pi(\bar{x})$ is $1$-based.    
\end{fact}

On the other hand Pillay proved (see \cite[Proposition 3.13]{MR1782132}). 

\begin{fact}[Pillay]
If there exists an infinite field interpretable in a first-order theory $T$, then $T$ is $n$-ample for all $n<\omega$. In particular, a $1$-based first-order theory cannot interpret an infinite field.  
\end{fact}

Since we will not deal in a direct manner with the notions of $1$-basedness and ampleness we refer the interested reader to \cite[Chapter 4]{MR1429864} and \cite{AmpleDividing} respectively.


\section{Stars of groups}\label{StarofGroups}

In this section we recount results about a particular type of graph of groups, i.e. graph of groups in which the underlying graph is a star (see Figure \ref{Star}). 
Equivalently, we amalgamate a (finite) family of groups $\{G_i\}_{i\in I}$ over a common subgroup $A$, and denote it by $G=*_A\{G_i\}_{i\in I}$. 
We call the resulting group a {\em star of groups} and each $G_i$, $i\in I$, a {\em factor subgroup}. Moreover, we call the cardinality of the index set $I$, the {\em number of rays} of the star of groups.  
Finally, for the sake of clarity we will abuse notation and 
identify $A$ with its images in the $G_i$'s under the defining embeddings $f_i:A\rightarrow G_i$.

\begin{figure}[ht!]
\centering
\includegraphics[width=.4\textwidth]{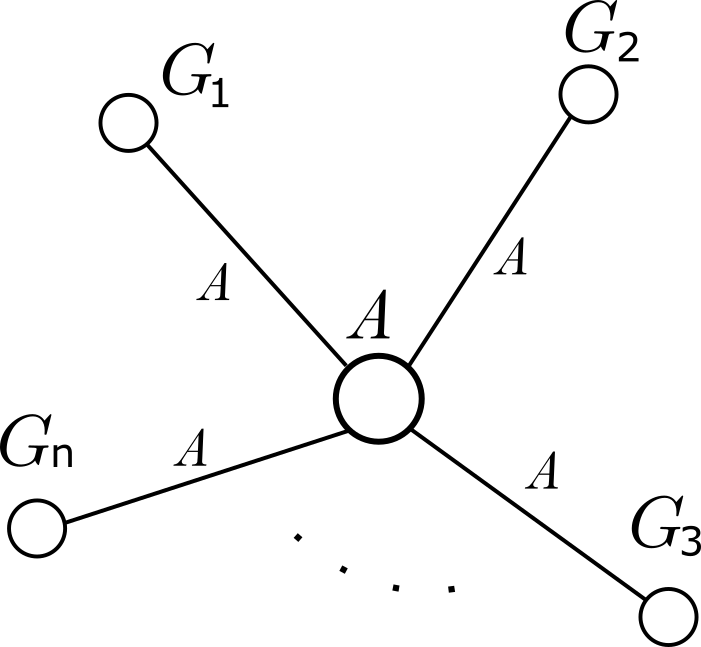}
\caption{A star of groups}
\label{Star}
\end{figure}

\subsection{Reduced Forms}\label{Canonical Forms}

We define the notion of a word in reduced form in a star of groups. Everything in this subsection is an easy generalization of results about amalgamated free products (see \cite{MR1954121}, \cite[Section 4.2]{MR0422434}, and \cite[Chapter IV (2)]{MR1812024}).  

\begin{definition}[reduced forms (cf Chapter I, Section 1.2 in \cite{MR1954121})]
Let $G=*_A\{G_i\}_{i\in I}$ be a star of groups. A word $ag_1g_2\ldots g_n$ is reduced if
\begin{itemize}
\item $a\in A$;
\item $g_j\in \bigcup \{G_i\setminus A \}_{i\in I}$, for each $j\leq n$;
\item  no consecutive $g_i$'s are in the same factor subgroup.
\end{itemize} 

To each reduced word $ag_1g_2\ldots g_n$ we can assign its length $L(ag_1g_2\ldots g_n)=n$.
\end{definition}

\begin{fact}[See Theorem 1 p. 3 in \cite{MR1954121}]
Every element $g$ of a star of groups, $G=*_A\{G_i\}_{i\in I}$, can be represented by a reduced word, i.e. $g=a\cdot g_1\cdot g_2\cdot\ldots\cdot g_n$. 

Furthermore, 
the representation is not unique, but the length and the sequence of factor subgroups for every representation is.
\end{fact}

\begin{remark}
The representation can be made unique if we fix a transversal for each $G_i$ with respect to (the image) of $A$. For the purpose of this paper it will be enough that every reduced word representative of an element of $G$ has identical sequence of factor subgroups, i.e. if $g=ag_1g_2\ldots g_n=bh_1h_2\ldots h_n$, then 
for each $i\leq n$ $h_i$ and $g_i$ belong to the same factor subgroup.
\end{remark}

\begin{lemma}[cf. Theorem 4.5(ii) in \cite{MR0422434}]\label{Commutation}
Let $G=*_A\{G_i\}_{i\in I}$ be a star of groups and $g,h\in G$ such that $g$ commutes with $h$. If neither $g$ nor $h$ is in a conjugate of the common subgroup $A$, 
and $g$ is in a conjugate of a factor subgroup $G_i$, for some $i\in I$, then $h$ is in the same conjugate of $G_i$. 
\end{lemma} 
\begin{proof}
The proof is by induction on the number of rays, $n=|I|$, of the star of groups. The basis of the induction is given by Theorem 4.5(ii) in \cite{MR0422434}. Now suppose the result holds for all 
stars of groups with at most $n$ number of rays, we will show it holds for $n+1$. Let $G=*_A\{G_i\}_{i\in I}$ for some $I$ with $|I|=n+1$. Then we may see $G$ as the amalgamated free product $(*_A\{G_i\}_{i\in J})*_AG_j$, 
where $J=I\setminus j$. Let $g, h$ be elements of $G$ as in the hypothesis. If $g$ is in a conjugate of some factor subgroup $G_i$ with $i\in J$, then it is conjugate in $(*_A\{G_i\}_{i\in J})$, thus by applying Theorem 4.5(ii) in \cite{MR0422434} for the amalgamated free product we get that $h$ belongs to the same conjugate and the result follows by the induction hypothesis. If $g$ is in a conjugate of $G_j$, then, as before, $h$ is in the same conjugate.
\end{proof}

When dealing with conjugacy classes it will be convenient to have a similar form to work with. 

\begin{definition}[Cyclically reduced forms]
A reduced word $ag_1g_2\ldots g_n$ in a star of groups is cyclically reduced, if $g_1, g_n$ are in different factor subgroups unless $n=1$.
\end{definition}

\begin{lemma}[Proposition 2 p. 5 in \cite{MR1954121} \& cf. Theorem 4.6 in \cite{MR0422434}]\label{ConjugacyNormalForm}
Every element of a star of groups, $G=*_A\{G_i\}_{i\in I}$, is conjugate to (the representative of) a cyclically reduced word. 

Furthermore, suppose $b$ and $c$ are cyclically reduced and conjugate in $G$. Then: 
\begin{itemize}
\item If $b$ belongs to $A$, then $c$ belongs to some factor subgroup $G_i$, $i\in I$, and there is a sequence $b=a_0, a_1, a_2, \ldots, a_m, a_{m+1}=c$, 
where each $a_i$, for $0<i\leq m$, is in $A$ and $a_{i+1}$ is conjugate to $a_i$, for $i<m$, in some factor subgroup. 
\item if $b$ belongs to some factor subgroup but cannot be conjugated in $A$, then $c$ belongs to the same factor and 
they are conjugates in this factor. 
\item if $b=ab'_1b_2\ldots b_n=b_1b_2\ldots b_n$ with $L(ab_1b_2\ldots b_n)>1$, then $c$ can be obtained by cyclically permuting $b_1, b_2, \ldots, b_n$ and then 
conjugating by an element of $A$.
\end{itemize} 
\end{lemma}
\begin{proof}
That every element of a star of groups is conjugate to a cyclically reduced word is proved by Proposition 2 in \cite{MR1954121}. The additional claims are proved by induction on the number of rays of the star of groups and the proofs are similar to the proof of Lemma \ref{Commutation}. We only do the last case, i.e. the case where $b$ has length strictly greater than $1$. 
The basis of the induction is given by Theorem 4.6 in \cite{MR0422434}. Suppose the result holds for all stars of groups with at most $n$ rays, we will show it holds for $n+1$. Let $G=*_A\{G_i\}_{i\in I}$ for some $I$ with $|I|=n+1$. Then we may see $G$ as the amalgamated free product $(*_A\{G_i\}_{i\in J})*_AG_j$, where $J=I\setminus j$. Let $b, c$ be elements of $G$ as in the hypothesis and $b$ has length strictly greater than $1$. If $b$ belongs to 
$(*_A\{G_i\}_{i\in J})$, i.e. if each $b_i$ belongs to some factor subgroup $G_i$ for $i\in J$, then by applying Theorem 4.6 in \cite{MR0422434} once more, we get that $c$ belongs as well to $(*_A\{G_i\}_{i\in J})$ and they are conjugates there. Hence, in this case the result follows from the induction hypothesis. 

If $b$ does not belong to $(*_A\{G_i\}_{i\in J})$, then it has length strictly greater than one with respect to the amalgamated free product $(*_A\{G_i\}_{i\in J})*_AG_j$, say $b=\gamma_1\gamma_2\ldots\gamma_m$. We may assume that $\gamma_1\ldots\gamma_m$ is cyclically reduced (if not $\gamma_1'\gamma_2\ldots\gamma_{m-1}$, where $\gamma_1'=\gamma_m\gamma_1$ is, because $\gamma_1'$ does not belong to $A$, since  $b_1b_2\ldots b_n$ is cyclically reduced). Hence, maybe after permuting, we may assume that $\gamma_{2k+2}$ belong to $G_j\setminus A$ and each $\gamma_{2k+1}$ admits a reduced form with respect to $(*_A\{G_i\}_{i\in J})$. Now, $c$ is obtained by a cyclic permutation of $\gamma_1\gamma_2\ldots\gamma_m$ followed by a conjugation by an element of $A$, i.e. $c=a\gamma_{\sigma(1)}\gamma_{\sigma(2)}\ldots\gamma_{\sigma(m)}a^{-1}$. Unfolding the $\gamma_{\sigma(2k+1)}$ to their reduced forms, we get that $c$ is a cyclic permutation of $b$ followed by conjugation by $a$ in $G=*_A\{G_i\}_{i\in I}$.   
\end{proof}

\subsection{Permuting the rays}

For this subsection we fix a star of groups $G=*_A\{G_i\}_{i\leq n}$ of infinite finitely generated groups, where the image of $A$, the common subgroup, is a proper subgroup of each $G_i$. For any pair $G_i, G_j$ of factor subgroups we consider an  isomorphism of sets $f_{ij}:G_i\rightarrow G_j$ that extends the isomorphism between the images of $A$. In addition, $f_{ii}=Id$ and $f_{ik}=f_{jk}\circ f_{ij}$ for all $i,j,k\leq n$. 

We will first study the orbits of cyclically reduced words under the action induced by permuting the rays of a star of groups. Let $X$ be the set of cyclically reduced words of length $\geq 1$ of $G$. Then the group of permutations $S_n$ acts on $X$ as follows. We consider the natural action of $S_n$ on the 
set of $n$ rays, each element $\sigma \in S_n$ induces a permutation of the factor groups $\{G_i\}_{i\leq n}$ and sends a cyclically reduced word $ag_1g_2\ldots g_m$ to a cyclically reduced word $ah_1h_2\ldots h_m$ 
in the following way: if $g_i$ belongs to the factor $G_j$, then $h_i=f_{j\sigma(j)}(g_i)$ and, thus,  belongs to the factor $G_{\sigma(j)}$. Lemmas \ref{OrbitsConjugacy} and \ref{OrbitsCosets} do not depend on the the particular choices of $f_{ij}$. It 
is immediate, by properties of permutation groups, that the word $ah_1h_2\ldots h_m$ is cyclically reduced (and of length $m\geq 1$). The action is free on the subset $X_f$ of $X$ that consists of cyclically 
reduced words that every of the $n$ factor subgroups has a representative in them. As a matter of fact a stronger property holds for $g\in X_f$, if $\sigma(g)=aga^{-1}$ for some $a\in A$, then $\sigma$ 
must be the trivial permutation.     

\begin{lemma}\label{OrbitsConjugacy}
Let $g=ag'_1g_2\ldots g_m=g_1g_2\ldots g_m$ be a cyclically reduced word of length $\geq 1$. Then the orbit $S_n.g$ 
contains at least $\lceil\frac{\lfloor n/2\rfloor}{2}\rceil$ conjugacy classes.   
\end{lemma}   
\begin{proof}
We may assume that $n$ is large enough.

We first prove the result for cyclically reduced words in which all factor subgroups have a representative in them.
We abstract the element $g=g_1g_2\ldots g_m$ by only remembering the sequence of its factor subgroups, i.e. if $g_j\in G_{i_j}$, then we consider the surjective map $u:m\rightarrow n$, where $u(j)=i_j$, for $j\leq m$. The group generated by the cyclic permutation $\tau=(1 2 \ldots m)$ acts on the set consisting of these maps $u$ by pre-composition and the symmetric group $S_n$ acts on the same set by post-composition. In addition, the two actions are compatible. We observe that the orbit of $u$ under the action of $\langle \tau\rangle$ depends only on the conjugacy class of $g$. Furthermore, there exists a smallest natural number $d>1$ such that $u.\tau^d=u$. We want to understand the permutations of $S_n$ that fix the orbit of $u$ by $\langle \tau\rangle$, that is the permutations that fix the conjugacy class of $g$. We will show that the stabilizer of the $u.\langle \tau\rangle$ in $S_n$ embeds in $\mathbb Z/d\mathbb Z$, hence it is cyclic. 
Indeed, let $\mu\in Stab_{S_n}(u.\langle \tau\rangle)$. Then $\mu.u=u.\tau^k$, where $k$ is well defined $mod \ d$. Thus, there exists a morphism $h$ from  $Stab_{S_n}(u.\langle \tau\rangle)$ to $\mathbb Z/d\mathbb Z$. Now, if 
$\mu.u=u$ we must have, since $g\in X_f$, that $\mu$ is the identity permutation, hence $h$ is injective. 

Finally, we consider the product $\Pi_2^{\lfloor n/2\rfloor}:=(\mathbb Z/2\mathbb Z)\times\ldots\times (\mathbb Z/2\mathbb Z)$ generated by the permutations $\{(1,2), (3,4), \ldots, (n-1,n)\}$. It is a subgroup of $S_n$ all of whose elements have order $2$. 
Since a cyclic group can have at most one element of order $2$ we get that the intersection of $Stab_{S_n}(u.\langle \tau\rangle)$ with $\Pi_2^{\lfloor n/2\rfloor}$ contains at most two elements. Now an easy counting argument gives  the result.  

When a factor subgroup is not represented in the cyclically reduced word $g_1g_2\ldots g_m$ we can do better. Assume, that $k$ out of $n$ factor subgroups are not represented. 
Without loss of generality, we may assume that $G_1, \ldots, G_k$ are not represented. Then, we may consider the $(n-k)\cdot k$-many transpositions, $\{ (i, k+j) \ | \ 1\leq i \leq k, \ 1\leq j \leq n-k\}$. 
Any cyclically reduced word of the above type has as many images (up to conjugation) under those transpositions and the result follows.   
\end{proof} 

We next prove similar results for basic equivalence relations. We will assume some familiarity with the theory of actions by isometries on trees. In particular, we will freely use results from \cite[Chapter 3]{MR1851337}.    


\begin{lemma}\label{OrbitsCosets}
Let $h=dh'_1h_2\ldots h_m=h_1h_2\ldots h_m\neq 1$ be a reduced word of length $\geq 1$. Suppose for each nontrivial element, $a\in A$, its centralizer, $C_{G}(a)$, is contained in a factor subgroup $G_i$, for some $i\leq n$.  
 
\begin{itemize}
    \item the orbit $S_n.h$ contains at least $\lfloor n/2\rfloor$ non-commuting elements.
    \item the orbit $S_n.h$ contains at least $\lceil\frac{\lfloor (n-1)/2\rfloor}{2}\rceil$ non-commuting conjugation elements.
    \item assume $G$ is a limit group, then for any nontrivial $a\in A$, the set $\{(\sigma.h,a) \ | \ \sigma \in S_n \}$ contains at least $n-1$ equivalence classes, $_cE_2$, of cosets with conjugation.
    \item assume $G$ is a limit group, then for any nontrivial $a,b\in A$, the set $\{(b,\sigma.h, a) \ | \ \sigma \in S_n \}$ contains at least $n-2$ equivalence classes, $_cE_3$, of double cosets with conjugation.
\end{itemize}
\end{lemma}
\begin{proof}
We consider the action on the corresponding Bass-Serre tree $X$. Let $*$ be the (unique) vertex stabilised by $A$, and $x_i$, for $i\leq n$, the (unique) vertices stabilised by $G_i$ respectively. A fundamental domain of the action is the star consisting of the convex hull of $\{*, x_1, \ldots, x_n\}$, with $*$ in the center.  

In order to prove the first point we consider the characteristic set, $A_h$, of $h$. For any permutation $\sigma\in S_n$, if $h$ commutes with $\sigma.h$, then $A_h=(\sigma.h).A_h$.  If $h_1$ belongs to the factor subgroup $G_{i_1}$, then $h.*$ belongs to the connected component of $X\setminus\{*\}$ that contains $x_{i_1}$. Moreover, $A_h$ either belongs to the same component (in the case $*$ does not belong to $A_h$ - see figure \ref{CosetOrbit}) or is a line that contains the segment $[x_{i_1},*, x_j]$, for some $i_1\neq j\leq n$ and $h$ translates from $*$ towards $x_{i_1}$. In both cases, we can choose $\lfloor n/2 \rfloor$ transpositions moving $i_1$ that will give us pairwise noncommuting elements. 

\begin{figure}[ht!]
\centering
\includegraphics[width=.7\textwidth]{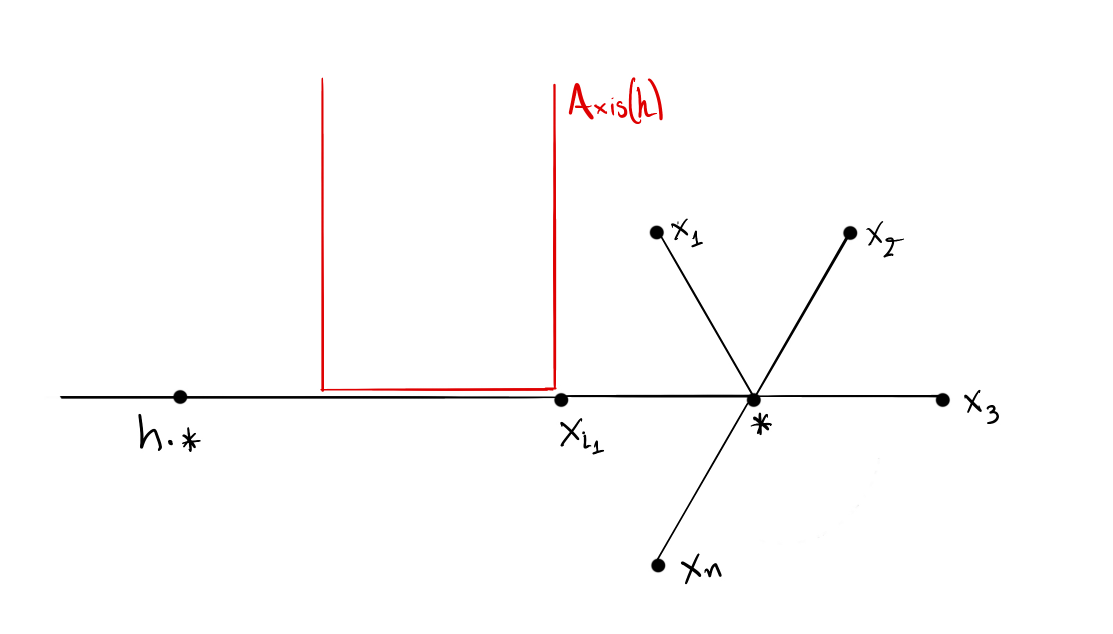}
\caption{The characteristic set of $h$ in the case $h$ is hyperbolic it does not contain $*$.}
\label{CosetOrbit}
\end{figure}


For the second point we split the proof in cases according to whether $h$ is elliptic or not. 
\begin{itemize}
\item Suppose $h$ is elliptic. Without loss of generality it fixes a translate of $x_1$. Then for any transposition, $\sigma$, of the $n-1$ transpositions from the set $\{(1, j) \ : \ 2\leq j \leq n\}$, the element $\sigma.h$ cannot be conjugated to the same vertex group as $h$. Hence, by Lemma \ref{Commutation} we have $n$ pairwise non-commuting conjugation elements   
\item Suppose $h$ is hyperbolic. We may assume that $h_1h_2\ldots h_m$ is cyclically reduced and $m>1$. We first tackle the case where all factor subgroups are represented in $h$. Let $u:m\rightarrow n$ be the function that assigns the factor subgroup to the $j$-th element of $h_1h_2\ldots h_m$, i.e. if $h_j\in G_{i_j}$, then $u(j)=i_j$, for $j\leq m$. Then the characteristic set, $A_h$, of $h$, that we will call the axis of $h$, is a line that contains the following segment $I:=[x_{u(m)}, *, x_{u(1)}, h_1.*, h_1.x_{u(2)}, h_1h_2.*, \ldots, h_1h_2\ldots h_m.*]$. It is surprisingly easy to find that $x_{(u(m)}$ belongs to the axis  since an easy calculation shows that $d(x_{u(m)}, h^2.x_{u(m)})=2d(x_{u(m)},h.x_{u(m)})$, where $d$ is the path metric on the Bass-Serre tree. Mroeover, translates of $I$ by powers of $h$ cover the whole axis. If $\sigma$ is a permutation in $S_n$, then $\sigma.h$ (that by abusing notation we will decompose as $\sigma.h_1\sigma.h_2\ldots \sigma.h_m$) is still hyperbolic (and cyclically reduced) moreover its axis contains the segment $[x_{\sigma.u(m)}, *, x_{\sigma.u(1)}, (\sigma.h_1).*,$ $ (\sigma.h_1).x_{\sigma.u(2)}, (\sigma.h_1\sigma.h_2).*, \ldots, (\sigma.h).*]$. A conjugate of a hyperbolic element is still hyperbolic and the axis of $\gamma h\gamma^{-1}$ is a translate of $A_h$ by $\gamma$, i.e. $A_{\gamma h\gamma^{-1}}=\gamma.A_h$. In addition, if two hyperbolic elements commute then their axes coincide.     

Without loss of generality we assume that $u(m)=n$ and we consider the subgroup, $S_{n-1}$, of permutations that fix $n$. Now, if $\sigma\in S_{n-1}$ and a conjugate, $\gamma h \gamma^{-1}$, of $h$ commutes with $\sigma.h$, we must have, since $\gamma.A_h=A_{\sigma.h}$, that $\gamma.*=(\sigma.h)^k\sigma.h_1\sigma.h_2\ldots\sigma.h_i.*$ for some $k\in\Z$ and $0\leq i<m$. Hence $\gamma$ is equal to $(\sigma.h)^k\sigma.h_1\sigma.h_2\ldots\sigma.h_i\cdot a$ for some $a\in Stab(*)$. Since $\sigma.h_i\cdot a$ belongs to the same factor subgroup as $\sigma.h_i$ we will take, for convenience of notation, $\gamma$ to be equal to $(\sigma.h)^k\sigma.h_1\sigma.h_2\ldots\sigma.h_i$. Moreover, we must have $$\gamma.(h_1.*)= (\sigma.h)^k\sigma.h_1\sigma.h_2\ldots\sigma.h_i\sigma.h_{i+1}.*$$ 
$$\gamma.(h_1h_2.*)= (\sigma.h)^k\sigma.h_1\sigma.h_2\ldots\sigma.h_i\sigma.h_{i+1}\sigma.h_{i+2}.*$$
$$\vdots$$
$$\gamma.(h_1h_2\ldots h_m.*)=(\sigma.h)^{k+1}\sigma.h_1\sigma.h_2\ldots\sigma.h_i.*$$
Indeed, we can eliminate the alternative that, for example, $\gamma.(h_1.*)=(\sigma.h)^k\sigma.h_1\sigma.h_2\ldots$ $\sigma.h_{i-1}.*$ (i.e. $\gamma$ inverts the direction), because the segment $[x_n, *]$ belongs to both axes $A_h$ and $A_{\sigma.h}$ (this is a consequence of $\sigma$ fixing $n$). In particular, $\gamma$ moves the segment $[x_n, *]$ along the axis of $\sigma.h$ and must move towards the same direction the segment $[*, x_{u(1)}, h_1.*, h_1.x_{u(2)}, h_1h_2.*, \ldots,$ $h_1h_2\ldots h_m.*]$. Finally, substituting $\gamma$ by 
$(\sigma.h)^k\sigma.h_1\sigma.h_2\ldots\sigma.h_i$ we get the following equalities:
$$(\sigma.h)^k\sigma.h_1\sigma.h_2\ldots\sigma.h_i.(h_1.*)= (\sigma.h)^k\sigma.h_1\sigma.h_2\ldots\sigma.h_i\sigma.h_{i+1}.*$$ 
$$(\sigma.h)^k\sigma.h_1\sigma.h_2\ldots\sigma.h_i.(h_1h_2.*)= (\sigma.h)^k\sigma.h_1\sigma.h_2\ldots\sigma.h_i\sigma.h_{i+1}\sigma.h_{i+2}.*$$
$$\vdots$$
$$(\sigma.h)^k\sigma.h_1\sigma.h_2\ldots\sigma.h_i.(h_1h_2\ldots h_m.*)=(\sigma.h)^{k+1}\sigma.h_1\sigma.h_2\ldots\sigma.h_i.*$$

The latter equalities imply that $\sigma.u=u.\tau^k$ for $\tau=(1 \ 2 \ \ldots \ m)$. Thus, we can use the same argument as in Lemma \ref{OrbitsConjugacy} to conclude that we have $\lceil\frac{\lfloor (n-1)/2\rfloor}{2}\rceil$ pairwise non-commuting conjugation elements in the orbit of $h$ by $\Pi_2^{\lfloor (n-1)/2\rfloor}:=\{(1, 2), (3, 4), \ldots,$ $(n-2, n-1)\}$. 

The easier case where at least one factor group is not represented in $h$ is left to the reader. 
\end{itemize}

For the third point, observe that $C_{G}(a)$ is assumed to be contained in a factor subgroup and without loss of generality it is contained in $G_1$. If $(\sigma.h C_{G}(a))^g=\tau.h  C_{G}(a)$, for some $g\in G$, then $[a^g,a]=1$ and in particular, since limit groups are CSA, we get that $g\in C_{G}(a)\subset G_1$. Consequently, we get $g(\sigma.h)g^{-1}\gamma=\tau.h$, for some $\gamma\in G_1$. If $h_m$ belongs to $G_{i_m}$, the transpositions $(i_m,j)$, for $1<j\leq n$, give $n-1$ distinct $_cE_2$-classes.   

Finally, for the fourth point, assume without loss of generality that $C_G(a)\subseteq G_1$ and $C_G(b)\subseteq G_2$. If $(C_{G}(b) \sigma.h C_{G}(a))^g= C_{G}(b) \tau.h C_{G}(a)$, for some $g\in G$, then $[b^g,b]=1$ and $[a^g,a]=1$, in particular, since limit groups are CSA we get that $g\in G_1\cap G_2$. Consequently, we get  $\gamma g \sigma.h g^{-1}\delta = \tau.h$, for some $\gamma\in G_2$, $\delta\in G_1$ and $g\in G_1\cap G_2$. Thus, if $h_m$ belongs to $G_{i_m}$, the transpositions $(i_m,j)$, for $2<j\leq n$, give $n-2$ distinct $_cE_3$-classes.

\end{proof}

%


\section{Towers}\label{Towers}

A {\em limit group} $L$ is a finitely generated group for which there exists a sequence of morphisms, $(h_n)_{n<\omega}:L\rightarrow \F$, such that for every nontrivial $g\in L$, $h_n(g)\neq 1$ for all but finitely many $n$. 
We call a sequence such as $(h_n)_{n<\omega}$, a stably injective sequence. We give more details on the definition and connections with the class of $\omega$-residually free groups in subsection \ref{RealTrees}.

In this section we define a special subclass of limit groups namely groups that have the structure of a {\em tower}. Towers played an important role in the proof of the elementary 
equivalence of nonabelian free groups. Notably they have been used in order to generalize Merzlyakov's  
theorem \cite{MR1972179} \cite{MR2154989}, which is the conceptual basis of the proof of the quantifier elimination (down to boolean combinations of $\forall\exists$-formulas).

A tower is built recursively by adding {\em floors} to a given ground floor, that consists of a nonabelian free group. There are two types of floors, {\em surface floors} and {\em Abelian floors}.  
The corresponding notion in the work of Kharlampovich-Myasnikov is the notion of an {\em NTQ group}, 
i.e. the coordinate group of a nondegenerate triangular quasiquadratic system of equations (see \cite[Definition 9]{MR1610660}).   

Towers can be thought of as groups equipped with construction instructions. The instructions consist of the nonabelian free group of the ground floor,  the additional floors, and finally the way and order each floor is added to the already constructed tower. 
Towers also admit closures. The closure of a tower is obtained by augmenting the Abelian floors of the original tower in a way that the original floor sits as a finite index subgroup in the augmented floor. 
It is still a tower with the same number, type and order of floors,  
moreover it contains the original tower as a subgroup. When no Abelian floors take part in the construction of a tower, then we call it hyperbolic. As we will see in the sequel, hyperbolic towers are easier to handle exactly because they coincide with any of their closures.   

In addition, we will define multiplets of a tower, namely we will add identical floors on the same basis multiple times. The end product is still a tower and in addition it can be seen as a star of groups.

Finally, we will "symmetrize" the closure of a multiplet of a tower in a way the original Abelian floors embed identically as finite index subgroups in all copies of the multiplet.

\subsection{The construction of a tower}

We assume some familiarity with Bass-Serre theory \cite{MR1954121}. We start by defining the notion of a surface floor. 

\begin{definition} [Surface-type vertex]\label{stype}
A vertex $v$ of a graph of groups $\Gamma$  
is called a \emph{surface-type vertex} if the following conditions hold:
\begin{itemize}
 \item the group $G_v$ carried by $v$ is the fundamental group of a compact surface $\Sigma$ (usually with boundary), with Euler characteristic $\chi(\Sigma)<0$;
 
  \item incident edge groups are maximal  boundary subgroups of $\pi_1(\Sigma)$, and this induces a bijection
 between the set of incident edges and the set of boundary components of $\Sigma$.
\end{itemize}
\end{definition}

\begin{definition}[Exceptional surfaces] \label{except}
Four  hyperbolic surfaces with $\chi(\Sigma)=-1$ are considered exceptional because their mapping class group is ``too small'' (they do not carry pseudo-Anosov diffeomorphisms):  the thrice-punctured sphere, 
the twice-punctured projective plane, the once-punctured Klein bottle,  and  the closed non-orientable surface of genus 3.
\end{definition}

\begin{definition}[Centered splitting] \label{centspl}
A centered splitting of $G$ is a  graph of groups decomposition $G=\pi_1(\Gamma)$ such that the vertices of $\Gamma$
are $v,v_1,\dots,v_m$, with $m\ge1$, where $v$ is surface-type and every edge joins $v$ to some $v_i$ (see Figure \ref{HypFloor}).

The vertex $v$ is called the \emph{central vertex} of $\Gamma$. The vertices $v_1,\dots,v_n$ are the \emph{bottom vertices}, and we denote by $H_i$ 
 the {\em bottom group} carried by $v_i$. The base of $\Gamma$ is the abstract free product  $H=H_1*\dots* H_m$.

The centered splitting $\Gamma$ is simple if $i=1$, and non-exceptional if the surface $\Sigma$ is non-exceptional. 

 \begin{figure}[ht!] 
\centering
\includegraphics[width=.6\textwidth]{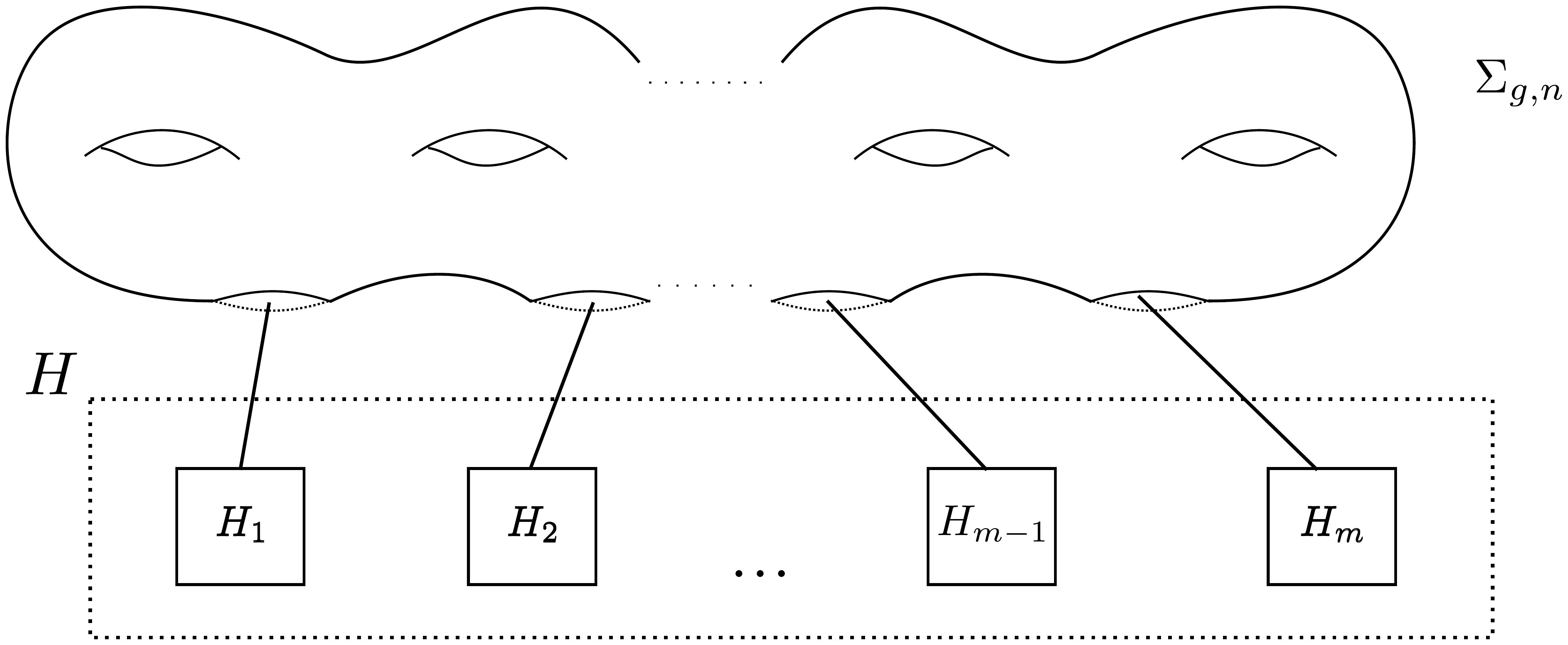}
\caption{A centered splitting}
\label{HypFloor}
\end{figure}

\end{definition}


\begin{definition}[Surface floor]\label{SurfaceFloor}
Let $G$ be a group and $H$ be a nonabelian subgroup of $G$. Then $G$ has the structure of a surface floor over $H$,  
if the following conditions hold:  
\begin{itemize}
\item the group $G$ admits a simple non-exceptional centered splitting with bottom group $H$; and 
\item there exists a retraction $r:G\to H$ that sends the group carried by the central vertex of $\Gamma$ to a non Abelian image.
\end{itemize}

\end{definition}

An Abelian floor is defined in a similar way. 

\begin{definition}
Let $G$ be a group and $H$ be a subgroup of $G$. Then $G$ has the structure of an Abelian floor over $H$,  
if $G$ admits a splitting as an amalgamated free product $H*_E(E\oplus \Z^m)$, where (the image of) $E$ is a maximal Abelian subgroup of $H$ and $\Z^m$ is a free Abelian group of rank $m$ (see Figure \ref{AbFloor}). 

We call the image of $E$ in $H$ the peg of the Abelian floor.
\end{definition}

\begin{figure}[ht!]
\centering
\includegraphics[width=.3\textwidth]{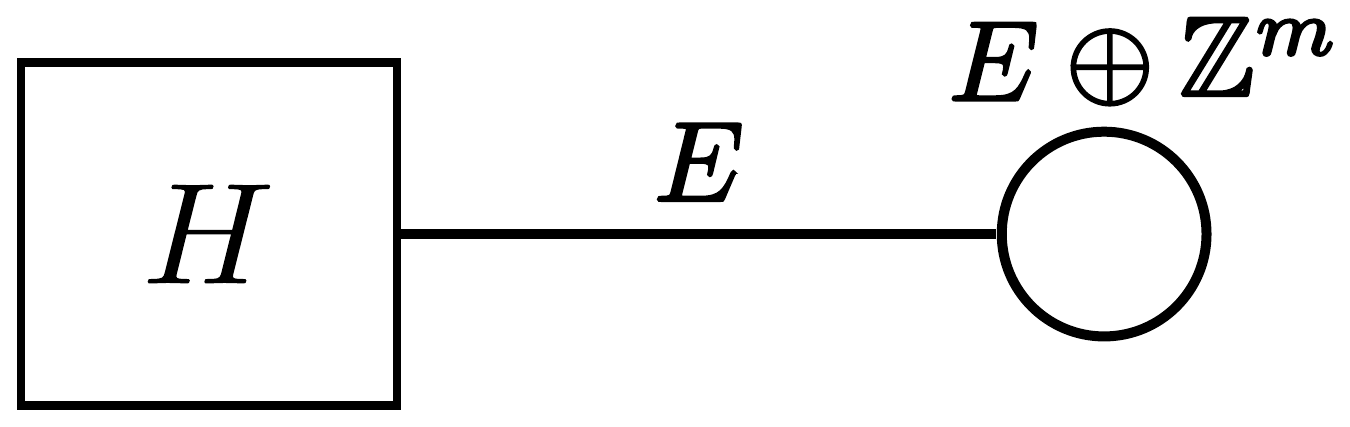}
\caption{An Abelian floor}
\label{AbFloor}
\end{figure}

We can now define towers.

\begin{definition}\label{Tower}
A group $G$ has the structure of a tower (of height $m$) over a nonabelian subgroup $\F$ if there exists a (possibly trivial) free group $\F_n$ and a sequence $G=G^m>G^{m-1}>\ldots>G^0=\F*\F_n$ such that for each $i$, $0\leq i<m$, 
either $G^{i+1}$ is a surface floor or an Abelian floor over $G^i$. 
\end{definition}

 \begin{figure}[ht!] 
\centering
\includegraphics[width=.3\textwidth]{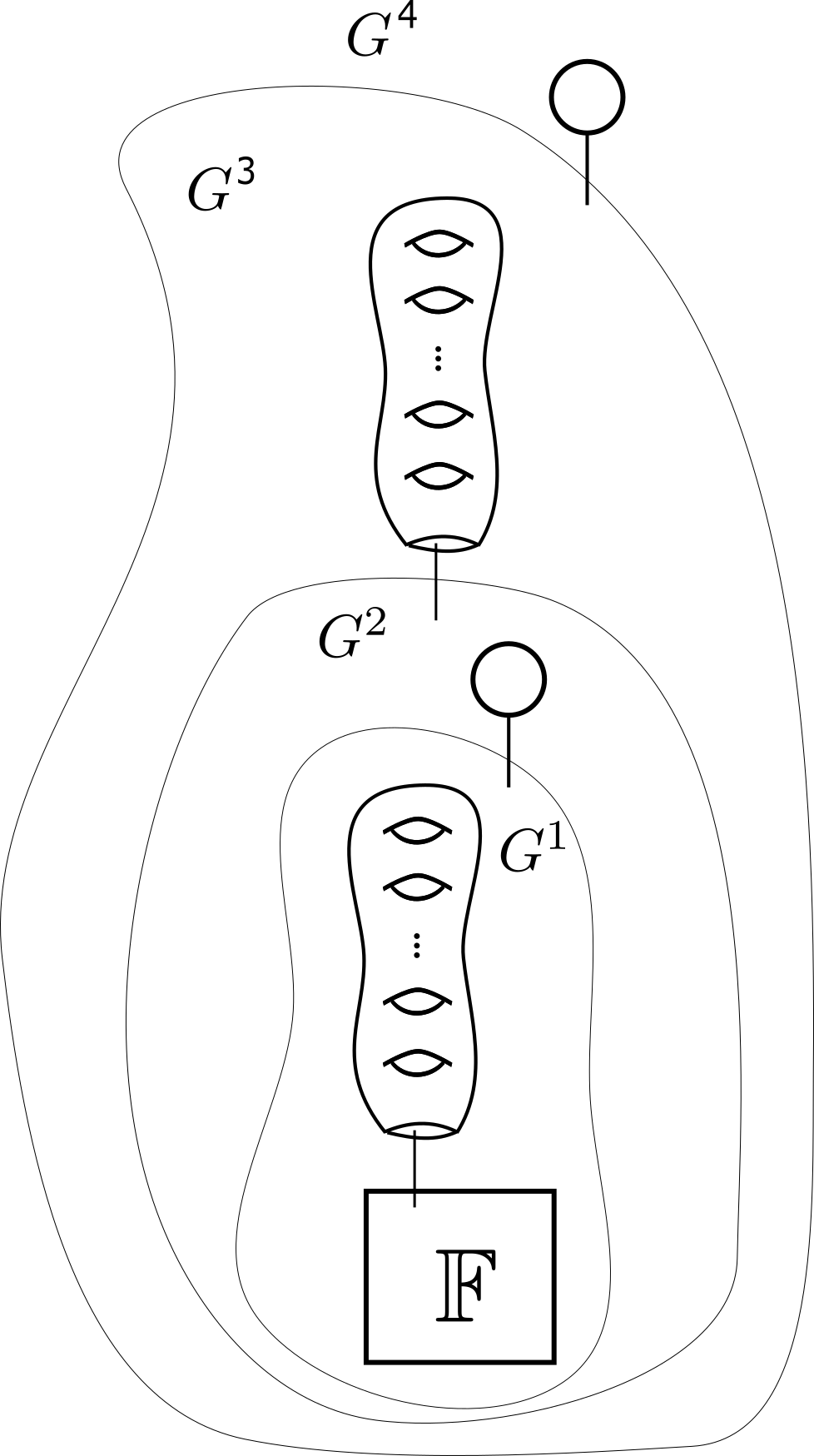}
\caption{A tower of height 4.}
\label{TowerPic}
\end{figure}

\begin{remark}\label{SimplevsExtended}
In this paper we will use simple towers instead of "extended" ones, i.e. towers for which the base of surface floors may be disconnected (for a thorough discussion and connections between the different types of towers we refer the reader to Section 4.8 of \cite{Towers}). 

We note that Definition \ref{SurfaceFloor} works for a surface floor with disconnected base after choosing a particular subgroup of $G$ which is the free product of vertex stabilizers of (vertices in the orbits of) the distinct bottom vertices in the corresponding Bass-Serre tree. 

This is without loss of generality since, by \cite[Proposition 5.2]{Towers}, if $G$ is an extended surface floor over the subgroup $H:=H_1*\ldots *H_m$, and $r:G\rightarrow H$ is the retraction witnessing it, then $G*\F_{m-1}$, with $\F_{m-1}:=\langle t_2,\ldots, t_m\rangle$, is a simple surface floor (i.e. with a single base vertex) over $\tilde H:=H_1*(H_2)^{t_2}*\ldots *(H_m)^{t_m}$. The particular construction will be useful in the sequel. 

We consider the map $f:H\rightarrow G*\F_{m-1}$, which is the identity on $H_1$ and conjugation by $t_i$ on $H_i$, for $2<i\leq m$. The image of $f$ in $G*\F_r$ is $\tilde H$. The group $G*\F_{m-1}$ has a structure of a (simple) surface floor with the same surface type vertex and amalgamation embeddings as the original (disconnected) centered splitting only conjugated by the corresponding $t_i$'s on the side of the base vertex group. The retraction from $G*\F_{m-1}$ onto $\tilde H$ is given by the map that agrees with $f\circ r$ on $G$ and kills $\F_{m-1}$.

We will further justify our choice of simple towers over extended ones after defining test sequences and the Diophantine envelope of a first-order formula in Section \ref{TestSequences}.  
\end{remark}

In \cite{MR2605174} (or rather in \cite{MR2605175}) it was proved that the class of limit groups coincides with the class of constructive limit groups \cite[Definition 1.14]{MR2605174}. Thus, one can easily deduce the following fact. 

\begin{fact}
If $G$ has the structure of a tower, then $G$ is a limit group.
\end{fact}

It will be useful to collect the information witnessing that a group $G$ has the structure of a tower. Thus we define:

\begin{definition}
Suppose $G$ has the structure of a tower (of height $m$). Then the tower $\mathcal{T}(G,\F)$ is the group $G=G^m$ together with the following collection of data:
$$((\mathcal{G}(G^1,G^0),r_1),(\mathcal{G}(G^2,G^1),r_2),\ldots,(\mathcal{G}(G^m,G^{m-1}),r_m))$$
where the splitting $\mathcal{G}(G^{i+1},G^i)$ is the splitting that witnesses that $G^{i+1}$ has the structure of a surface or Abelian floor over $G^i$ and $r_i$ is the corresponding retraction. 
In addition, for each floor $\mathcal{G}(G^{i+1},G^{i})$ the data includes the choice of a maximal subtree of the graph of groups, that contains the vertices stabilized by $G^{i+1}$ and $G^i$.

Finally, if $\mathcal{T}$ is a tower (of height $m$), then we denote by $\mathcal{T}_i$, for $i\leq m$, the tower that consists of the first $i$ floors of $\mathcal{T}$.
\end{definition}

\begin{remark}\label{AbelRetract}
When a floor $(\mathcal{G}(G^i,G^{i-1}),r_i)$ is Abelian with free Abelian vertex group $\Z^r\oplus\Z^k:=\langle a_1, \ldots a_r, z_1, \ldots, z_k\rangle$, the retraction $r_i:G^i\rightarrow G^{i-1}$, is the map induced by fixing $a_1, \ldots, a_r$ and sending each $z_i$ to $a_1$.  
\end{remark}

\begin{definition}
Let $\mathcal{T}(G,\F)$ be a tower of height $m$ and $G^m>G^{m-1}>\ldots>G^1>G^0:=\F*\F_n$ the sequence of the groups of its floors. We define the height of $g\in G$ to be the least $i$ 
such that $g\in G^i\setminus G^{i-1}$. 

In addition, for each $g\in G$ we define a nested analysis of $g$ with respect to $\mathcal{T}$ recursively as follows:
\begin{itemize}
\item if $g$ has height $0$, then $g$ admits a canonical form $g_1g_2\ldots g_k$ with respect to $\F*\F_n$ and the elements $\{g_1, g_2, \ldots, g_k\}$ is the nested analysis of $g$;
\item suppose the nested analysis has been defined for all elements $g$ of height $i$; 
\item if $g$ has height $i+1$, then it has a reduced form, $g_1g_2\ldots g_k$ with respect to the splitting $\mathcal{G}(G^{i+1}, G^i)$ and the (given) maximal subtree. The nested analysis of $g$ is the union 
of the $g_i$'s in $G^{i+1}\setminus G^i$ together with nested analyses of each $g_{i_1}, \ldots, g_{i_\ell}$ for the elements $g_{i_j}$, for $j\leq \ell$ that belong to $G^i$.  
\end{itemize}

\end{definition}

Note that the nested analysis $An_{\mathcal{T}}(g)$ of an element $g$ contains elements only from vertex groups of the splittings $\mathcal{G}(G^{i+1}, G^i)$, for $i\geq 1$, Bass-Serre elements,  and also elements from $\F$ and $\F_n$. Furthermore, we assume that the nested analysis of an element remembers the information to reconstruct the element from its components. The latter will be useful when we may want to change some elements in the analysis of $g$ (respecting the rules of reduced forms) in order to obtain some element $g'$ that will differ from $g$ only in the particular places we changed.  

It will be convenient to collect the elements of a nested analysis divided in the various levels of the tower that they belong to. 

\begin{definition}
Let $\mathcal{T}(G,\F)$ be a tower of height $m$ and $G^m>G^{m-1}>\ldots>G^1>G^0:=\F*\F_n$ the sequence of the groups of its floors. Let $g\in G$ and $An_{\mathcal{T}}(g)$ be a nested analysis of $g$ with respect to $\mathcal{T}$. For each $-2\leq i< m$, we define: 

\[ 
An^i_\mathcal{T}(g)= \left\{
\begin{array}{ll}
      \textrm{the elements in $An_{\mathcal{T}}(g)$ that belong to $G^{i+1}\setminus G^i$}, & \ \textrm {if} \ 0\leq i<m  \\
      \textrm{the elements in $An_{\mathcal{T}}(g)$ that belong to $\F$,} & \ \textrm{if}  \ i=-1\\
      \textrm{the elements in $An_{\mathcal{T}}(g)$ that belong to $\F_n$,} & \ \textrm{if}  \ i=-2\\
\end{array} 
\right. 
\]
\end{definition}
 
Since reduced forms with respect to $\mathcal{G}(G^{i+1}, G^i)$ are not unique, an element  $g\in G$ does not have a unique nested analysis with respect to $\mathcal{T}(G, \F)$. This fact will be supressed in the sequel since we will always work with a particular reduced form.

A tower in which no Abelian floor occurs in its construction is called a {\em hyperbolic tower} (or {\em regular NTQ group} in the terminology of Kharlampovich-Myasnikov) and it is, by an application of \cite{MR1152226}, a hyperbolic group. 

For a tower $\mathcal{T}(G,\F)$, it will be convenient to assume some further properties on Abelian floors. These assumptions will make proofs in later sections technically less involved (see subsection \ref{TowersEquiv}), but are not essential for the results. We would like $\mathcal{T}$ to satisfy the following: 
\begin{itemize}
\item  if a peg can be conjugated into $\F$, then it is a subgroup of $\F$;
\item no peg embeds to a conjugate of a peg of a higher floor;
\item the Abelian floors that correspond to pegs that belong to $\F$ always come before any other floors in the construction of the tower.
\end{itemize}

Not all towers satisfy these properties, but we can always change a tower to satisfy them. Indeed, changing a peg of an Abelian floor to a conjugate of itself is equivalent to changing the fundamental domain of the Bass-Serre action on the tree corresponding to the splitting of the particular floor. Hence, we can always consider pegs up to conjugation. Moreover, if a (conjugate of a) peg belongs to a lower floor, we can change the order of the floors by moving the Abelian floor that contains this peg lower. Finally, if we have two consecutive Abelian floors, where the peg of the higher floor is (a conjugate of) the maximal Abelian group that contains the peg of the lower floor, we can "merge" the two floors into one. More concretely: 

\begin{lemma}\label{ChangeFloorOrder}
Let $\mathcal{T}(G, \F)$ be a tower and $(\mathcal{G}(G^{i+1}, G^i), r_{i+1})$ its $i$-th floor. Suppose $(\mathcal{G}(G^{i+2},$ $G^{i+1}), r_{i+2})$ is an 
Abelian floor and its peg, $A<G^{i+1}$, is a subgroup of $G^{i}$. Then we can exchange the $i$ and $i+1$ floors to obtain a new tower $\hat{\mathcal{T}}(G,\F)$. 
\end{lemma}
\begin{proof}
We consider the tower $\mathcal{T}_{i-1}$, i.e. the first $i-1$ floors of $\mathcal{T}$. Suppose $E_1\oplus \Z^m$ is the free Abelian group of $(\mathcal{G}(G^{i+2}, G^{i+1}), r_{i+2})$. Since $A$ (the image of $E_1$ in $G^{i+1}$ under, $f$, the amalgamation embedding) is a subgroup of $G^i$ and maximal Abelian in $G^{i+1}$ (hence in $G^i$), we can add the $i+1$-th floor of $\mathcal{T}$ to $\mathcal{T}_{i-1}$. We call this 
tower $\hat{\mathcal{T}}_i$. Its corresponding group is $\hat{G}^{i+1}=\langle G^i, E_1\oplus \Z^m \ | \ e=f(e), \ e\in E_1 \rangle$ and its last floor is $\hat{\mathcal{G}}(\hat{G}^{i+1}, G^i)$. We next add the $i+1$-th floor of $\mathcal{T}$ to $\hat{\mathcal{T}}_i$. We take cases according to whether $(\mathcal{G}(G^{i+1}, G^i), r_{i+1})$ is Abelian or surface type. 
\begin{itemize}
\item Suppose $(\mathcal{G}(G^{i+1}, G^i), r_{i+1})$ is of surface type. Suppose the surface type vertex carries the fundamental group of an (orientable) surface of genus $g$ with $b$ boundary components, i.e. $\langle a_1, \ldots, a_b, u_1, v_1, \ldots, u_g, v_g \ | \ a_1\cdots a_b=[u_1,v_1]\cdots[u_g,v_g] \rangle$ (the case of a nonorientable surface is similar). Then $G^{i+1}$ has the following presentation:
$$G^{i+1}=\left\langle
    \begin{matrix} G^i \\ 
    a_1,\dots,a_b\\
       t_1,\dots,t_b\\
      u_1,v_1,\dots,u_g,v_g
    \end{matrix}\ \middle|\ 
       \begin{matrix}
      a_1\cdots a_b=[u_1,v_1]\cdots[u_g,v_g]\\
      t_j h_j t_j^{-1}=a_j,\ h_j\in G^i \ \textrm{for}  \ j\leq b \\
      t_1=1 
  \end{matrix}
\right \rangle$$

We define $(\hat{\mathcal{G}}(\hat{G}^{i+2}, \hat{G}^{i+1}), \hat{r}_{i+2})$ to be a splitting of surface type. Its surface type vertex carries the same fundamental group as the fundamental group of the surface type vertex of $(\mathcal{G}(G^{i+1}, G^i), r_{i+1})$ and has the following presentation:

$$\hat{G}^{i+2}=\left\langle
    \begin{matrix} \hat{G}^{i+1} \\ 
    a_1,\dots,a_b\\
       t_1,\dots,t_b\\
      u_1,v_1,\dots,u_g,v_g
    \end{matrix}\ \middle|\ 
       \begin{matrix}
      a_1\cdots a_b=[u_1,v_1]\cdots[u_g,v_g]\\
      t_j h_j t_j^{-1}=a_j,\ h_j\in \hat{G}^{i+1} \ \textrm{for}  \ j\leq b \\
      t_1=1 
  \end{matrix}
\right \rangle$$

Note that $G^i$ is a subgroup of $\hat{G}^{i+1}$, hence it is possible to "glue" the surface type vertex along its boundary components on the same elements and in addition we can retract $\hat{G}^{i+2}$ to $\hat{G}^{i+1}$ by the morphism $\hat{r}_{i+1}:\hat{G}^{i+2}\rightarrow \hat{G}^{i+1}$ defined as the identity on $\hat{G}^{i+1}$ and equal to $r_i$ on $\langle a_1, \ldots, a_b, t_1, \ldots, t_b, u_1, v_1, \ldots, u_g, v_g\rangle$. The image of $\langle a_1, \ldots, a_b, u_1, v_1, \ldots, u_g, v_g\rangle$ under $\hat{r}_{i+1}$ is contained in $G^{i}$ and it is nonabelian, hence it is nonabelian in $\hat{G}^{i+1}$.
 
Finally, by comparing presentations, we get $\hat{G}^{i+2}=G^{i+2}$.

\item Suppose $(\mathcal{G}(G^{i+1}, G^i), r_{i+1})$ is of Abelian type. Let $E_2\oplus\Z^n$ be the free Abelian vertex group of the splitting and $B<G^i$ the peg. We first observe that since no nontrivial element of $B$ can be conjugated into $A$ by an element of $G^i$, no nontrivial element of $B$ commutes with a conjugate in $\hat{G}^{i+1}$ of a nontrivial element of $A$. Indeed, suppose, for a contradiction, that $\gamma\in \hat{G}^{i+1}$, is such that $b$ commutes with $a^\gamma$, for some $b\in B\setminus\{1\}$ and $a\in A\setminus\{1\}$. Then, since $A\oplus\Z^m$ is maximal Abelian in $\hat{G}^{i+1}$ we have that $b\in  (A\oplus\Z^m)^\gamma$. Now we have that $b$ fixes the path between $x$ and $\gamma.y$, where $x$ is the vertex stabilized by $G^i$ and $y$ the vertex stabilized by $E\oplus\Z^m$ in the tree that corresponds to the splitting $\hat{\mathcal{G}}(\hat{G}^{i+1}, G^i)$. But any such path contains an edge stabilized by $A^g$, for some $g\in G^i$. In particular, $b$ belongs to $A^g$, hence $b$ can be conjugated to $A$ in $G^i$, a contradiction. 
We next show that $B$ is maximal Abelian in $\hat{G}^{i+1}$. Suppose, for a contradiction, not. Let $g\in \hat{G}^{i+1}\setminus G^i$ be an element that commutes with $B$. Then, by our previous argument, neither $g$ nor any nontrivial element of $B$ can be conjugated in $A$ (in $\hat{G}^{i+1}$). Hence, by Lemma \ref{Commutation}, we get that $g$ belongs to $G^i$, the same factor subgroup as $B$, a contradiction.        

In this case, we define $(\hat{\mathcal{G}}(\hat{G}^{i+2}, \hat{G}^{i+1}), \hat{r}_{i+2})$ to be a splitting of Abelian type. Its free Abelian vertex is $E_2\oplus\Z^n$ and the amalgamation embedding, $g$, maps $E_2$ isomorphically onto $B$ in the same way $E_2$ is mapped into $G^i$ in the splitting $(\mathcal{G}(G^{i+1}, G^i), r_{i+1})$. The group $\hat{G}^{i+2}$ has the following presentation:
$$ 
\hat{G}^{i+2}=\left\langle
    \begin{matrix} \hat{G}^{i+1} \\ 
    E_2, z_1, \ldots, z_n \\
    \end{matrix}\ \middle|\ 
       \begin{matrix}
      [e_2,z_i]=1, \ i\leq n, \ e_2\in E_2 \\
      [z_i,z_j]=1, \ i\neq j \\
      e_2=g(e_2), \ e_2\in E_2 
  \end{matrix}
\right \rangle$$
Again, by comparing presentations we conclude that $\hat{G}^{i+2}=G^{i+2}$. 
\end{itemize}
Finally, to construct $\hat{\mathcal{T}}(\hat{G}, \F)$, since in each of the above cases the group $\hat{G}^{i+2}$ coincides with the group $G^{i+2}$, 
we add the floors $(\mathcal{G}(G_{j+1}, G_j), r_{j+1})$, for $i+1<j<m$ on top of $\hat{G}^{i+2}$. It is immediate that $\hat{G}$ is $G$.
\end{proof}

\begin{lemma}\label{MergeAbelianFloors}
Let $\mathcal{T}(\mathcal{G}, \F)$ be a tower of height $m$ and $(\mathcal{G}(G^{i+1}, G^i), r_{i+1})$, $(\mathcal{G}(G^{i+2},$ $G^{i+1}), r_{i+2})$ be consecutive Abelian floors for some $i<m-1$. If the peg of the $i$-th floor 
can be conjugated into the peg of the $i+1$-th floor, then there exists a tower $\hat{\mathcal{T}}(G, \F)$ of height $m-1$, which is obtained from $\mathcal{T}(G, \F)$ by merging the $i$-th and $i+1$-th floors together, while  
keeping the rest of the floors identical.    
\end{lemma}
\begin{proof}
Suppose $E\oplus \Z^m:=\langle e_1, \ldots, e_k, z_1, \ldots, z_m\rangle$,  is the free Abelian vertex group in the $i$-th floor and $D\oplus Y^n:=\langle d_1, \ldots, d_\ell, y_1, \ldots, y_n\rangle$, is the free Abelian vertex in the $i+1$-th floor. Let  $A<G^i$ be the peg of the $i$-th floor, i.e. the image of $E$ in $G^{i}$ by the amalgamation embedding $f$ and $B<G^{i+1}$ the peg of the $i+1$-th floor, i.e. the image of $D$ in $G^{i+1}$ by the amalgamation embedding $g$. 
By the hypothesis of the lemma there is $\gamma\in G^{i+1}$ such that $A^\gamma<B$. Since $B$ is maximal Abelian in $G^{i+1}$ we must have that $B$ is $(A\oplus \Z^m)^\gamma$. In particular, $\ell=k+m$, and $g$ sends $d_i$ to $e_i^\gamma$ for $i\leq k$ and to $z_{k-i}^\gamma$ for $k<i\leq \ell$. Hence, $G^{i+2}$ admits the following presentation: 

$$ 
G^{i+2}=\left \langle
    \begin{matrix} G^{i} \\ 
    e_1, \ldots, e_k, z_1, \ldots, z_m \\
    d_1, \ldots, d_\ell, y_1, \ldots, y_n\\
    \end{matrix}\ \middle| \ 
       \begin{matrix}
       [e_i,e_j]=1, \ \textrm{for} \ i, j\leq k, \ [z_i,z_j]=1, \ \textrm{for} \ i\neq j\leq m \  \\ 
      [e_i,z_j]=1, \ \textrm{for} \ i\leq k, \ j\leq m, 
      e_i=f(e_i), \textrm{for} \ \ i\leq k  \\ 
       [d_i,d_j]=1, \ \textrm{for} \ i, j\leq \ell, \ 
      [d_i,y_j]=1, \ \textrm{for} \ i\leq \ell, \ j\leq n \\
      [y_i,y_j]=1, \ \textrm{for} \ i\neq j\leq n, \ 
      d_i= e_i^\gamma, \ \textrm{for} \ i\leq k, \\
      d_i= z_{k-i}^\gamma,  \ \textrm{for} \ k< i\leq \ell
  \end{matrix}
\right \rangle$$

Now, using Tietze transformations (see \cite[Chapter 2, Section 6]{MR2396717}), we change the presentation to the following:

$$ 
G^{i+2}=\left \langle
    \begin{matrix} G^{i} \\ 
    e_1, \ldots, e_k, z_1, \ldots, z_m \\
    y_1, \ldots, y_n\\
    \end{matrix}\ \middle| \ 
       \begin{matrix}
       [e_i,e_j]=1, \ \textrm{for} \ i, j\leq k, \ [z_i,z_j]=1, \ \textrm{for} \ i\neq j\leq m \  \\ 
      [e_i,z_j]=1, \ \textrm{for} \ i\leq k, \ j\leq m, 
      e_i=f(e_i), \textrm{for} \ \ i\leq k  \\ 
      [y_i,y_j]=1, \ \textrm{for} \ i\neq j\leq n, \ [e_i,y_j^{\gamma^{-1}}]=1, \ \textrm{for} \ i\leq k, \ j\leq n \\ 
      [z_i,y_j^{\gamma^{-1}}]=1, \ \textrm{for} \  i\leq m, \ j\leq n  
    
  \end{matrix}
\right \rangle$$

We can now see $G^{i+2}$ as the amalgamated free product $$G^i*_{E}\langle e_1, \ldots, e_k, z_1, \ldots, z_m, y_1^{\gamma^{-1}}, \ldots, y_n^{\gamma^{-1}}\rangle$$ 
Since $\langle z_1, \ldots, z_m, y_1^{\gamma^{-1}}, \ldots, y_n^{\gamma^{-1}}\rangle$ is isomorphic to $\Z^{m+n}$ and (the image of) $E$ is maximal Abelian in $G^i$, we can 
give $G^{i+2}$ the structure of an Abelian floor over $G^i$. We call this floor $\hat{\mathcal{G}}(\hat{G}^{i+1}, G^i)$, where $\hat{G}^{i+1}=G^{i+2}$ and we continue the construction of $\hat{\mathcal{T}}(\hat{G}, \F)$ 
with adding the floors of $\mathcal{T}$, $(\mathcal{G}(G^{j+1}, G^j), r_{j+1})$, for $i+2\leq j<m$, on top of $\hat{G}^{i+1}$. It is immediate that $\hat{G}$ is $G$.

\end{proof}

In light of the above lemmas we, from now on, assume that all towers satisfy the properties listed after Remark \ref{AbelRetract}. When we want to emphasize the dependence of a result on these properties we will write that the tower $\mathcal{T}$ satisfies the convention. We will further justify this assumption when we define test sequences and their properties in Section \ref{TestSequences}.

We record the following easy consequence of the above convention for later use.

\begin{lemma}\label{CyclicEdges}
Let $\mathcal{T}(G,\F)$ be a tower that satisfies the convention. Then all edge groups in the floor splittings $(\mathcal{G}(G^{i+1}, G^i), r_{i+1})$, for $i<m$, are cyclic. 
\end{lemma}
\begin{proof}
The statement is trivially true for edge groups of surface floors since a maximal boundary subgroup is cyclic. We need to deal with Abelian floors. 
Since towers are limit groups and maximal Abelian subgroups in limit groups are centralizers of nontrivial elements, we need to prove that any nontrivial 
element of $G^i$ that does not commute with any conjugate of a peg that belongs to $G^{i-1}$ (as this is part of our convention) has cyclic centralizer in $G^i$. 

Indeed, for $i=0$, $G^0$ is a free group, hence any nontrivial element has cyclic centralizer. Assume the statement to be true for $i=n$, we show it is true for $n+1$. 
Consider, a nontrivial element $a\in G^{n+1}$ that does not commute with any conjugate of a peg introduced thus far. Without loss of generality $a$ is not a proper power in $G^{n+1}$, otherwise 
we can replace $a$ with its largest root (recall that the centralizer of $a$ is a priori a finite rank free Abelian group). We take cases according to whether $\mathcal{G}(G^{n+1}, G^n)$ is a surface or Abelian floor.

\begin{itemize}
\item Suppose $\mathcal{G}(G^{n+1}, G^n)$ is a surface floor, with surface vertex group $\pi_1(\Sigma_{g,k})$. Then no new peg is introduced in $G^n$. If $a$ acts hyperbolically in the tree that corresponds to the surface floor splitting, then any element, $c$, that commutes with $g$ share the same axis with $g$. Since translation lengths are integer valued and the spitting is 2-acylindrical we get $g^q=c^p$, for some integers $p,q$, and since $g$ does not have proper roots we have that $g^k=c$, hence the centralizer of $g$ in $G^{n+1}$ is cyclic. 

If $a$ fixes a single vertex in the tree (either a translate of the vertex, $x$, stabilized by $G^n$ or a translate of the vertex, $y$, stabilized by $\pi_1(\Sigma_{g,k})$), then any element, $c$, that commutes with $a$ fixes the same vertex. If this vertex is in the orbit of $y$, then the vertex group is a free group, hence the centralizer of $a$ is cyclic. If, on the other hand, this vertex is in the orbit of $x$, then, this time by the induction hypothesis, the centralizer of $a$ is cyclic again.

Finally, $a$ may fix an edge but not a segment of length greater than $2$. In addition, up to conjugating $a$, we may assume that it fixes $x$. For any $b$ that commutes with $a$ we get that $ab^{-1}x=b^{-1}x$. Therefore, 
either $b$ fixes $x$ or it sends $x$ to a neighboring $y$, but the orbits of $x$ and $y$ under the action of $G^{n+1}$ are disjoint, hence $b$ fixes $x$ as well and the result follows by the induction hypothesis as before.    

\item Suppose $\mathcal{G}(G^{n+1}, G^n)$ is an Abelian floor and $\Z\oplus\Z^r$ the free Abelian vertex group. In this case a new peg is introduced in $G^n$, but by our 
assumption $a$ does not commute with any conjugate of it. As before if $a$ acts hyperbolically in the tree corresponding to the splitting, then its centralizer is cyclic. 

If $a$ fixes a unique vertex then this cannot carry a vertex group which is a conjugate of $\Z\oplus \Z^r$, therefore it necessarily fixes a vertex in the orbit of 
the vertex stabilized by $G^n$. In this case, any element of $G^{n+1}$ that commutes with $a$ also fixes the same vertex, hence, by the inductive hypothesis, the centralizer of $a$ is cyclic. 

Finally, if $a$ fixes an edge, then it commutes with a conjugate of the newly introduced peg, which is impossible by the hypothesis.    
  
\end{itemize}
\end{proof}

\begin{definition}[Abelian Pouch]
Let $\mathcal{T}(G, \F)$ be a tower and $G^m>G^{m-1}>\ldots>G^1>G^0:=\F*\F_n$ the sequence of the groups of its floors. Let $j<m$ be the least number such that $(\mathcal{G}(G^{j+1}, G^j), r_{j+1})$ is a surface floor or an Abelian floor whose peg nested analysis contains an element of $\F_n$. Then $\tilde{G}^j$ is the group obtained from $\F$ by adding on top of it the first $j-1$ (Abelian) floors of $\mathcal{T}$ and is called the Abelian pouch of $\mathcal{T}$.  
\end{definition} 

Note that, in the above definition, $\tilde{G}^j$ differs from $G^j$ by a free factor, i.e. $G^j=\tilde{G}^j*\F_n$.

\begin{figure}[ht!]
\centering
\includegraphics[width=.43\textwidth]{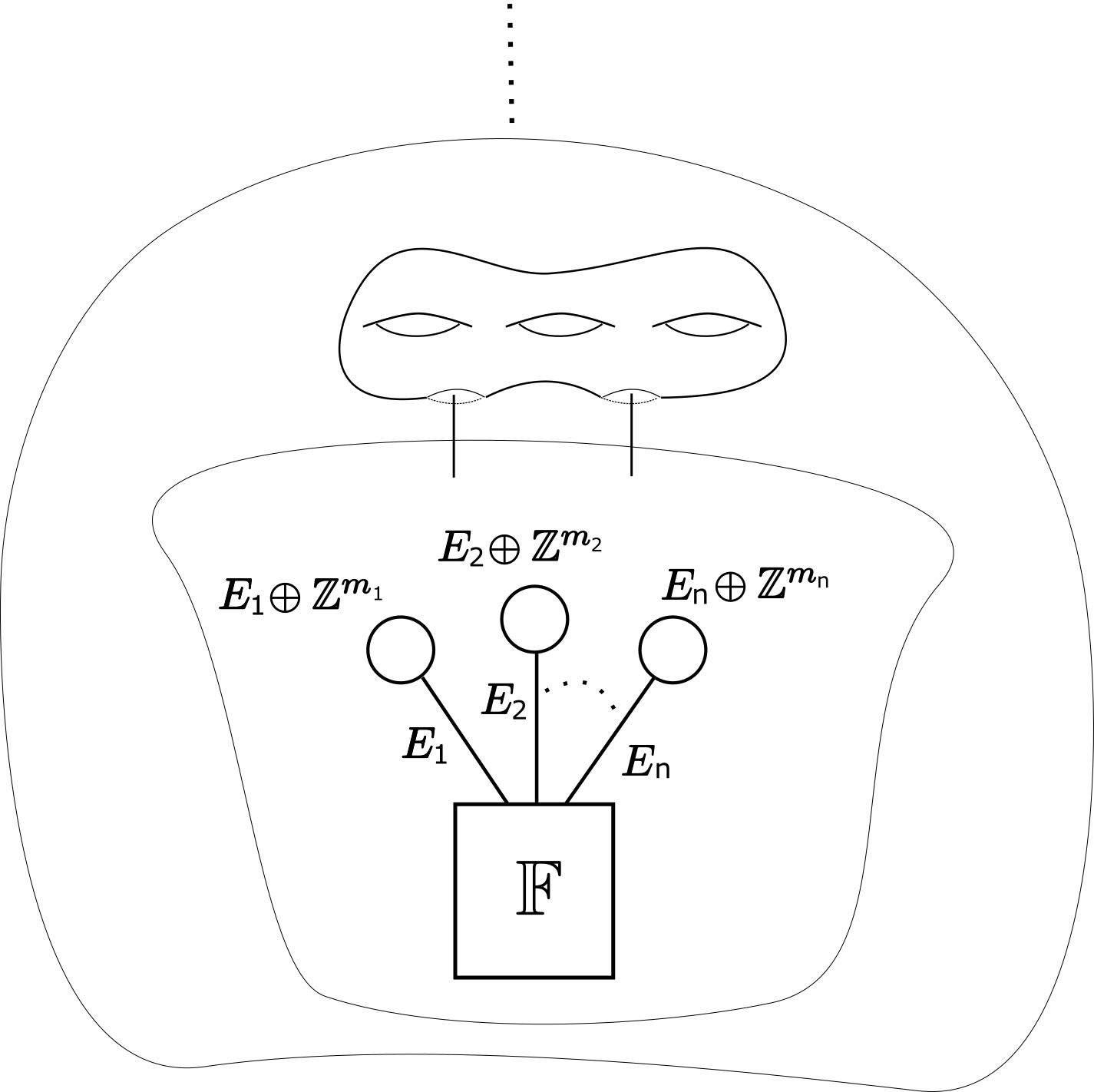}
\caption{The Abelian pouch}
\label{AbPouch}
\end{figure}

\subsection{Multiplets of towers}

We fix a tower $\mathcal{T}(G, \F)$ and we construct a new one by adding multiple times the floors of $\mathcal{T}$ on the ground floor $\F$. 
The group that corresponds to the end product of this construction admits a graph of groups decomposition as a star of groups where the common subgroup is the Abelian 
pouch of $\mathcal{T}$.  

\begin{definition}\label{MultiAbF}
Let $\mathcal{T}(G, \F)$ be a tower of height $m$ and $G^m>G^{m-1}>\ldots>G^1>G^0:=\F*\F_n$ the sequence of the groups of its floors. Suppose the Abelian pouch of  $\mathcal{T}$ is $\F$. 
Then the $N$-multiplet of $\mathcal{T}$, $\mathcal{T}^N$, is the tower with $m\cdot N$ floors, constructed recursively as follows:
\begin{itemize}
\item The ground floor of $\mathcal{T}^N$ is $\F*\F_n*\F_n*\ldots*\F_n$, where $\F_n$ is added $N-1$ times to $\F*\F_n$;
\item the first $m$-floors of $\mathcal{T}^N$ are identical to the floors of $\mathcal{T}$, only added on top of $(\F*\F_n)*\F_n*\ldots*\F_n$;
\item suppose for $i<N$, the $i$-multiplet of $\mathcal{T}$, $\mathcal{T}^i$, has been constructed;  
\item the floors from $im+1$ to $(i+1)m$, for $1\leq i<N$, are identical to the floors of $\mathcal{T}$, only added on top of $\mathcal{T}^i$ and with the following amalgamation embedding modifications: 
\begin{enumerate}
\item Let $g$ be the element that generates the image of an amalgamation embedding in the original tower;
\item We replace each element in the nested analysis of $g$ with the corresponding element in the $i$-th copy of $\F_n$, and the $i$-th copies of the floors constructed thus far within the range $im+1$ to $(i+1)m$. We use these elements and the memory of how $g$ was built in order to get a new element $g'$ which will be the image of the generator of the edge group under the new amalgamation embedding.
\end{enumerate} 
\end{itemize}
\end{definition}

The group that corresponds to the $N$-multiplet of a tower $\mathcal{T}(G,\F)$ with Abelian pouch equal to $\F$, is the star of groups $G_N=*_{\F}\{G_i\}_{i\leq N}$ with each $G_i$ isomorphic to $G$.

\begin{figure}[ht!]
\centering
\includegraphics[width=.5\textwidth]{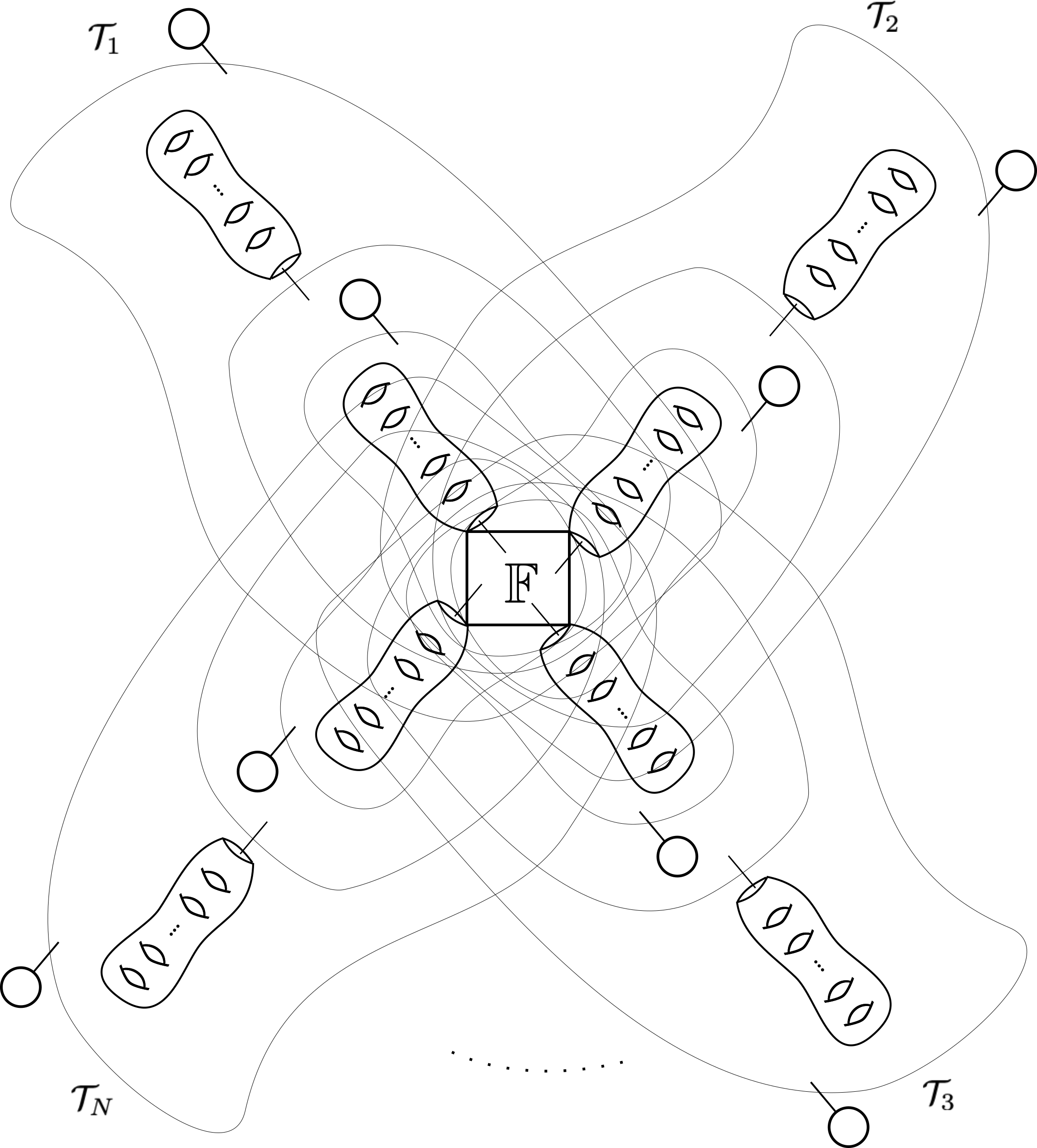}
\caption{The $N$-multiplet of $\mathcal{T}$}
\label{MultipletTower}
\end{figure}

When the Abelian pouch of the tower is not just $\F$, one has to 
take into account that when we attach the pegs of the Abelian pouch for the second, third or $N$-th time the edge groups are not maximal Abelian subgroups anymore, hence we attach them all together in the beginning. This is  
similar to the merging lemma \ref{MergeAbelianFloors}. 

\begin{definition}[Multiplets of Towers]\label{MultiAb}
Let $\mathcal{T}(G, \F)$ be a tower of height $m$ and $G^m>G^{m-1}>\ldots>G^1>G^0:=\F*\F_n$ the sequence of the groups of its floors. Suppose the Abelian pouch of $\mathcal{T}$ consists of $k$ Abelian floors, with 
free Abelian vertex groups $E_1\oplus\Z^{m_1}, E_2\oplus\Z^{m_2}, \ldots, E_k\oplus\Z^{m_k}$ . 
Then the $N$-multiplet of $\mathcal{T}$, $\mathcal{T}^N$, is the tower with $m + (m-k)\cdot (N-1)$ many floors constructed recursively as follows:
\begin{itemize}
\item The ground floor of $\mathcal{T}^N$ is $\F*\F_n*\ldots*\F_n$, where $\F_n$ is added $N-1$ times to $\F*\F_n$;
\item the first $k$ floors are Abelian with free Abelian vertex groups $E_i\oplus \Z^{m_i}\oplus \Z^{m_i}\oplus\ldots\oplus \Z^{m_i}$ ( adding $N-1$-times the summand $\Z^{m_i}$ to $E_i\oplus\Z^{m_i}$), for $i\leq k$. Moreover, each $E_i$ embeds in $\F$ as it embeds in the original tower;
\item the floors from $k+1$ to $m$ are identical to the floors from $k+1$ to $m$ of $\mathcal{T}$, only added on top of the $k$ floors as described in the previous point; 
\item suppose, for $i<N$, the $i$-multiplet of $\mathcal{T}$, $\mathcal{T}^i$, has been constructed; 
\item the floors $im +1$ to $(i+1)m-k$, for $1\leq i <N$ are identical to the floors $k+1$ to $m$ of $\mathcal{T}$, only added on top of $\mathcal{T}^i$ and with the following amalgamation embeddings modifications: 
\begin{enumerate}
\item Let $g$ be the element that generates the image of an amalgamation embedding in the original tower between the floors $k+1$ to $m$;
\item We replace each element in the nested analysis of $g$ with the corresponding element in the $i$-th copy of $\F_n$, the $i$-th copies of $\Z^{m_j}$, for $j\leq k$, and the $i$-th copies of the floors constructed thus far within the range $im+1$ to $(i+1)m-k$. We use these elements and the memory of how $g$ was built in order to get a new element $g'$ which will be the image of the generator of the edge group under the new amalgamation embedding.
\end{enumerate} 
\end{itemize}
\end{definition}

The group that corresponds to the $N$-multiplet of $\mathcal{T}(G,\F)$ is still a star of groups, but now the common subgroup is the Abelian pouch of $\mathcal{T}$ and each factor is isomorphic to the tower of the first $m$ floors of $\mathcal{T}^N$.

\begin{example}\label{ExDouble}
Let $\mathcal{T}(G, \F)$ be the tower of height $3$, that consists of the following floors:
\begin{itemize}
\item The ground floor $G^0$ is the free product of free groups $\F*\F_2:=\langle a_1, a_2\rangle*\langle x_1, x_2\rangle$;
\item The first floor $\mathcal{G}(G^1, G^0)$ is an Abelian floor with free Abelian vertex group $\Z\oplus \Z:=\langle z, z_1\rangle$ that is attached on $G^0$ through the amalgamation map to the maximal Abelian subgroup of $G^0$ generated by $a_1^2a_2^3$. Hence, $G^1$ admits the following presentation: 
$$\langle a_1, a_2, x_1, x_2, z, z_1 \ | \ [z,z_1]=1, \ z=a_1^2a_2^3\rangle$$ 
\item The second floor $\mathcal{G}(G^2, G^1)$ is a surface floor with surface type vertex an orientable surface $\Sigma_{1, 2}$ of genus $1$ with two boundary components, its fundamental group is $\langle e_1, e_2, b_1, b_2 \ |$ $ \ b_1b_2=[e_1, e_2] \rangle$ and it is attached through the amalgamation maps sending $b_1$ to the element of $G^1$ with canonical form (with respect to $\mathcal{G}(G^1, G^0)$) $\beta:=w_1(a_1,a_2) z_1^3w_2(a_1, a_2)$ $z_1^2w_3(a_1, a_2, x_1, x_2)$, and sending $b_2$ to a conjugate $(\beta^{-1})^\gamma$ of $\beta^{-1}$,  by an element $\gamma\in \F_2$ that does not commute with $\beta$. Furthermore, the element $w_3(a_1, a_2, x_1, x_2)$ has canonical form,  $w_3^1(a_1, a_2)w_3^2(x_1,x_2)$ with respect to $\F*\F_2$. We choose the edge whose edge group embeds to $b_1$ to be the maximal subtree of $\mathcal{G}(G^1, G^0)$. Hence, $G^2$ admits a presentation:
$$\langle G^1, e_1, e_2, b_1, b_2, t \ | \ b_1b_2=[e_1, e_2], \ b_1=\beta, \ t^{-1}b_2 t=(\beta^{-1})^\gamma \rangle $$
There is an obvious retraction $r_1:G^2\rightarrow G^1$ that sends $b_1$ to $\beta$, $b_2$ to $(\beta^{-1})^\gamma $, $e_1$ to $\beta$, $e_2$ to $\gamma$ and $t$ to $1$. Since $\beta$ does not commute with $\gamma$, $r_1(\pi_1(\Sigma_{1,2})$ is nonabelian.
\item The third floor $\mathcal{G}(G^3, G^2)$ is an Abelian floor with free Abelian vertex group $\Z\oplus \Z^2:=\langle y, z_2, z_3\rangle$ that is attached through the amalgamation map to the maximal Abelian subgroup of $G^2$ generated by the element with reduced form $tg_1ts_1g_2t^{-1}s_2g_3s_3t^{-2}$, where $s_i\in \pi_1(\Sigma_{1,2})$ and $g_j\in G^1\setminus G^0$. We further analyze $g_1, g_2$ and $g_3$, their normal forms (with respect to $\mathcal{G}(G^1, G^0)$) are  $z_1^2g_1^1(a_1,a_2)z_1^4$, $z_1^5$ and $g_3^1(a_1,a_2)z_1^7g_3^2(x_1, x_2)$ respectively.  Hence, $G^3$ admits the following presentation: 
$$\langle G^2, y, z_2, z_3 \ | \ [y, z_2]=[y,z_3]=[z_2,z_3]=1, \ y=tg_1ts_1g_2t^{-1}s_2g_3s_3t^{-2}\rangle$$
\end{itemize}
The double $\mathcal{T}^2$ of the above tower is a tower of height $5$ with the following floors:
\begin{itemize}
\item The ground floor $\tilde{G}^0$ is the free product $\F*\F_2*\F_2=\langle a_1, a_2 \rangle*\langle x_1, x_2\rangle*\langle x_1', x_2'\rangle$;
\item The first floor $\tilde{\mathcal{G}}(\tilde{G}^1, \tilde{G}^0)$ is an Abelian floor with free Abelian vertex group $\Z\oplus \Z\oplus\Z:=\langle z, z_1, z_1'\rangle$ that is attached on $\tilde{G}^0$ through the amalgamation map to the maximal Abelian subgroup of $\tilde{G}^0$ generated by $a_1^2a_2^3$. Hence, $\tilde{G}^1$ admits the following presentation: 
$$\langle a_1, a_2, x_1, x_2, x_1', x_2', z, z_1, z_1' \ | \ [z,z_1]=[z,z_1']=[z_1, z_1']=1, \ z=a_1^2a_2^3\rangle$$ 
\item The second floor $(\tilde{\mathcal{G}}(\tilde{G}^2, \tilde{G}^1), \tilde{r}_1)$ is a surface floor identical to the second floor of $\mathcal{T}$ only added on top of $\tilde{G}^1$. Hence, $\tilde{G}^2$ admits a presentation:
$$\langle \tilde{G}^1, e_1, e_2, b_1, b_2, t \ | \ b_1b_2=[e_1, e_2], \ b_1=\beta, \ t^{-1}b_2 t=(\beta^{-1})^\gamma \rangle $$
We define $\tilde{r}_1$ to agree with $r_1$ on $\pi_1(\Sigma_{1,2})$ and $t$ and be the identity on $\tilde{G}^1$, it is immediate that it is a retraction with $\tilde{r}_1(\pi_1(\Sigma_{1,2}))$ nonabelian.
\item The third floor $\tilde{\mathcal{G}}(\tilde{G}^3, \tilde{G}^2)$ is an Abelian floor identical to the third floor of $\mathcal{T}$ only added on top of $\tilde{G}^2$. Hence, it admits a presentation: 
$$\langle \tilde{G}^2, y, z_2, z_3 \ | \ [y, z_2]=[y,z_3]=[z_2,z_3]=1, \ y=tg_1ts_1g_2t^{-1}s_2g_3s_3t^{-2}\rangle$$
\item The fourth floor $(\tilde{\mathcal{G}}(\tilde{G}^4, \tilde{G}^3), \tilde{r}_3)$ is a surface floor with identical surface type vertex group as the second floor, i.e. $\langle e'_1, e'_2, b'_1, b'_2 \ |$ $ \ b'_1b'_2=[e'_1, e'_2] \rangle$. In order to define the amalgamation embeddings of the boundary elements $b_1', b_2'$ we need to understand the nested analysis of $\beta:=w_1(a_1,a_2) z_1^3w_2(a_1, a_2)$ $z_1^2w_3(a_1, a_2, x_1, x_2)$ and $\beta^\gamma$. The nested analysis of $\beta$ contains at level $-2$, $An_{\mathcal{T}}^{-2}(\beta)$, the element $w_3^2(x_1,x_2)$ and at level $0$, $An_{\mathcal{T}}^{0}(\beta)$, the elements $z_1^3$ and $z_1^2$. Now since $\gamma$ is an element of $\F$, the nested analysis of $\beta$ and $\beta^\gamma$ agree at the particular levels. Thus, we attach $b_1'$ to the element $\beta':=w_1(a_1,a_2) z_1'^3w_2(a_1, a_2)$ $z_1'^2w_3(a_1, a_2, x_1', x_2')$, where $w_3(a_1, a_2, x_1', x_2'):=w_3^1(a_1,a_2)w_3^2(x_1', x_2')$,  and $b_2'$ to the element $\beta'^{\gamma}$. Hence, $\tilde{G}^4$ admits the following presentation:
$$\langle \tilde{G}^3, e_1', e_2', b_1', b_2', t' \ | \ b_1'b_2'=[e_1', e_2'], \ b_1'=\beta', \ t'^{-1} b_2' t'=(\beta'^{-1})^\gamma \rangle $$
We define $\tilde{r}_3:\tilde{G}^4\rightarrow \tilde{G}^3$ to agree with $r_1$ on $\pi_1(\Sigma_{1,2})'$ (mapping elements to the corresponding $'$-elements) and $t'$ (which is sent to $1$) and be the identity on $\tilde{G}^3$, it is immediate that it is a retraction with $\tilde{r}_3(\pi_1(\Sigma_{1,2})')$ nonabelian.
\item The fifth floor $\tilde{\mathcal{G}}(\tilde{G}^5, \tilde{G}^4)$ is an Abelian floor whose free Abelian vertex group $\Z\oplus\Z^2:=\langle y',z'_2, z'_3\rangle$ is identical to the free Abelian vertex group of the third floor of $\mathcal{T}$ only added on top of $\tilde{G}^4$. Again, in order to define the amalgamation embedding we need to understand the nested analysis of the element $tg_1ts_1g_2t^{-1}s_2g_3s_3t^{-2}$. Following the analysis given 
in the fourth point of this example we attach $y'$ to the element $t'g'_1s'_1g'_2t'^{-1}s'_2g'_3s'_3t'^{-2}$, where $g'_1:=z_1'^2g_1^1(a_1,a_2)z_1'^4$, $g'_2:=z_1'^5$ and $g'_3:=g_3^1(a_1,a_2)z_1'^7g_3^2(x'_1, x'_2)$, moreover, 
$s'_j$ is the element of $\pi_1(\Sigma_{1,2})'$ that corresponds to $s_j$ of $\pi_1(\Sigma_{1,2})$, for $j\leq 3$.
Hence, $\tilde{G}^5$ admits the presentation: 
$$\langle \tilde{G}^4, y', z_2', z_3' \ | \ [y', z_2']=[y',z_3']=[z_2',z_3']=1, \ y'=t'g_1't's_1'g_2't'^{-1}s_2'g_3's_3't'^{-2}\rangle$$
\end{itemize}
 
\end{example}

\subsection{Closures of towers}

In this subsection we define the notion of a closure of a tower. 

\begin{definition}[Abelian closure]
Let $\Z\oplus\Z^m:=\langle z, z_1, z_2, \ldots, z_m \rangle$ and  $\Z\oplus A^m:=\langle z, a_1, a_2,$ $\ldots, a_m\rangle$ be free Abelian groups of rank $m+1$. Then an embedding 
$f:\Z\oplus\Z^m\rightarrow \Z\oplus A^m$ such that $f(z)=z$ is called a closure embedding.
\end{definition}

Equivalently, we can see that $f:\Z\oplus\Z^m\rightarrow \Z\oplus A^m$ is a closure embedding if and only if $f(z)=z$ and $f(\Z\oplus\Z^m)$ has finite index in $\Z\oplus A^m$. 
We call the latter group the {\em group closure of $\Z\oplus\Z^m$} with respect to $f$.

We now explain how a group closure $\Z\oplus A^m$ of rank $m+1$ defines a coset of a finite index subgroup of the cartesian product of (the group of) integers $Z^m:=\langle (1, 0, \ldots, 0)^T, (0,1,\ldots, 0)^T,$ $\ldots, (0, 0, \ldots, 1)^T\rangle$ 
such that a morphism $h:\Z\oplus \Z^m\rightarrow \F$, with $z\mapsto a$ and $z_i\mapsto a^{k_i}$, for $i\leq m$ (where $a$ is nontrivial and has no roots), extends to a morphism $g:\Z\oplus A^m\rightarrow \F$ (where we identify $\Z\oplus\Z^m$ by its image in $\Z\oplus A^m$ under $f$) 
if and only if $(k_1, k_2, \ldots, k_m)$ belongs to this coset. Note that we will always assume that $h$ sends $z$ (the element attached to the lower floor through the amalgamation embedding) to a nontrivial element without roots. This will be justified later in Section \ref{TestSequences}, since test sequences will always have this property (see Fact \ref{PowerPegs}). 

Suppose $f(z_i)=\mu_{i0}z+\mu_{i1}a_1+\ldots+\mu_{im}a_m$ and consider the linear system of equations over $\Z$ given by: 

\begin{table}[!htpb]
\centering
\begin{tabular}{r c r c c c r c r }
$\mu_{10}$ & $+$ & $ \mu_{11}x_1$ & $+$ & $\ldots$ & $+$ & $\mu_{1m}x_m$ & $=$ & $k_1$ \\
$\mu_{20}$ & $+$ & $ \mu_{21}x_1$ & $+$ & $\ldots$ & $+$ & $\mu_{2m}x_m$ & $=$ & $k_2$ \\
 & && &\ldots  & & & & \\
$\mu_{m0}$ & $+$ & $ \mu_{m1}x_1$ & $+$ & $\ldots$ & $+$ & $\mu_{mm}x_m$ & $=$ & $k_m$ \\
\end{tabular}
\end{table}
\newpage
We define $M:=(\mu_{ij})_{1\leq i\leq m \atop{0\leq j\leq m}}$ to be the $m\times m+1$ coefficient matrix of the above system. We want to solve the Diophantine system $\exists  x_1, \ldots, x_m [M.(1, x_1, \ldots, x_m)^T=(k_1, \ldots, k_m)^T]$ in $\Z$. Equivalently, 
we seek the solutions of $\exists x_1, \ldots, x_m [M_m.(x_1, \ldots, x_m)^T=(k_1 - \mu_{10}, \ldots, k_m - \mu_{m0})^T]$, where $M_m$ is the $m\times m$ matrix that consists of the last $m$ columns of $M$. 
The "column space" of $M_m$, i.e. the $\Z$-linear combinations of its columns, generate a finite index subgroup $\Delta$ of $Z^m$ (since $f$ is injective) and the Diophantine system admits a solution if and only if 
$(k_1, \ldots, k_m)^T$ takes values in the coset $(\mu_{10}, \mu_{20}, \ldots, \mu_{m0})^T+\Delta$. 

A closure of a tower is obtained by "augmenting" the free Abelian vertex group of each Abelian floor by a group closure with respect to some closure embedding. It is still a tower with the same number, type and order of floors. 

\begin{definition}[Tower closure]
Let $\mathcal{T}(G, \F)$ be a tower of height $m$ and $\{E_i\oplus\Z^{m_i}\}_{i\leq k}$ be the collection of the free Abelian vertex groups of its Abelian floors. Let, for each $i\leq k$,   
$f_i:E_i\oplus \Z^{m_i}\rightarrow E_i\oplus A^{m_i}$ be a closure embedding. Then the closure of $\mathcal{T}(G, \F)$ with respect to $\{f_i\}_{i\leq k}$, denoted $Cl(\mathcal{T}(G,\F))$, 
is a tower of height $m$, defined recursively as follows:
\begin{itemize}
\item The ground floor of $Cl(\mathcal{T})$ is identical to the ground floor of $\mathcal{T}$, i.e. $Cl(G^0):=\F*\F_n$;
\item  Suppose the tower $Cl(\mathcal{T})_i$ up to the  $i$-th floor $(\hat{\mathcal{G}}(Cl(G^{i+1}), Cl(G^i)), \hat{r}_i)$ has been constructed and has the following properties: 
\begin{enumerate}
\item the order and the type of floors of $Cl(\mathcal{T})_i$ are identical to those of $\mathcal{T}_i$; 
\item $Cl(G^{i+1})\supseteq G^{i+1}$;
\item  The pegs of Abelian floors 
of $Cl(\mathcal{T})_i$ as subgroups of $Cl(G^i)$ are identical with the pegs of the corresponding Abelian floors of $\mathcal{T}_i$ as subgroups of $G^i$;
\item the surface type vertex groups of the surface floors of $Cl(T)_i$ are identical to the surface type vertex groups of the corresponding surface floors of $\mathcal{T}_i$ and 
the boundary components are attached through the amalgamation embeddings to the same elements as the boundary components of the corresponding surface floors of $\mathcal{T}_i$; 
\end{enumerate}
\item We construct the $i+1$-th floor of $Cl(\mathcal{T})$. 
\begin{enumerate}
\item If the $i+1$-th floor of $\mathcal{T}$ is a surface floor, then the $i+1$-th floor of $Cl(\mathcal{T})$ is identical except is added on top of $Cl(G^{i+1})$ which contains $G^{i+1}$. The retraction $\hat{r}_{i+1}$ 
coincides with $r_{i+1}$ on the surface type vertex and the Bass-Serre elements and is the identity on $Cl(G^{i+1})$. Points 1) to 4) in the recursive assumption hold by construction.
\item If the $i+1$-th floor of $\mathcal{T}$ is an Abelian floor, with free Abelian vertex group $E_j\oplus\Z^{m_j}$ for some $j\leq k$, then the $i+1$-th floor of $Cl(\mathcal{T})$ is an Abelian floor with free Abelian 
vertex group, $E_j\oplus A^{m_j}$, the group closure of $E_j\oplus\Z^{m_j}$ with respect to $f_j$, and it is attached (in $Cl(G^{i+1})$) to the same cyclic subgroup that $E_j\oplus\Z^{m_j}$ was attached (in $G^{i+1}$). 
This latter group remains maximal Abelian in $Cl(G^i)$, since the only maximal Abelian groups in $G^i$ that are not maximal Abelian in $Cl(G_i)$ are free Abelian of rank $>1$. Again points 1) to 4) in the recursive assumption hold by construction.  
\end{enumerate}
\end{itemize}
\end{definition}

The group, $Cl(G)$, that corresponds to a closure of a tower $\mathcal{T}(G, \F)$ contains $G$ as a subgroup. A {\em covering set of closures} of $\mathcal{T}$ is a finite set of towers, $\mathcal{T}_1, \ldots, \mathcal{T}_k$ which are closures of $\mathcal{T}$ and for any morphism $h:G\rightarrow \F$ which is the identity on $\F$, there exists $i\leq k$ such that $h$ extends to the group that corresponds to $\mathcal{T}_i$. Equivalently, for each Abelian floor, $\Z\oplus\Z^m$, of $\mathcal{T}$ the union of cosets, $\bar \mu_i+\Delta^i$, that each corresponds to the group closure of $\Z\oplus\Z^m$ in $\mathcal{T}_i$, cover $Z^m$. 

\begin{example}
Consider the tower of height $3$ of example \ref{ExDouble} together with the closure embeddings $f_1:\langle z, z_1\rangle \rightarrow \langle z, a_1\rangle$ given by the rule $z\mapsto z$, $z_1\mapsto a_1^2$ and $f_2:\langle y, z_2, z_3\rangle\rightarrow \langle y, a_2, a_3\rangle$ given by the rule $y\mapsto y$, $z_2\mapsto ya_2^4a_3^3$, $z_3\mapsto ya_2^3a_3^2$,  for the first and third floors respectively. The closure embeddings define a tower closure with corresponding group $Cl(G)$. Any morphism $h:G\rightarrow \F$ that is the identity on $\F$ and sends pegs to nontrivial elements that have no roots extends to $Cl(G)$ if and only if it  satisfies the following properties:
\begin{itemize}
\item $h$ sends $z_1$ to $h(z)^{2\ell}$, for some $\ell\in \Z$;
\item  $h$ sends $z_2$ to $h(y)^{k_2}$  and $z_3$ to $h(y)^{k_3}$ with $\begin{bmatrix} k_{2} \\  k_{3} \end{bmatrix}\in \begin{bmatrix} 1 \\  1 \end{bmatrix}+\langle \begin{bmatrix} 4 \\  3 \end{bmatrix}, \begin{bmatrix} 2 \\  2 \end{bmatrix}\rangle$.  
\end{itemize}
Finally, it is an easy exercise in linear algebra to check that the closures defined by the cross product $\{f_1, g_1\}\times\{f_2, g_2\}$, where $g_1$ is defined by the rule $z\mapsto z$, $z_1\mapsto za_1^2$ and $g_2$ by the rule $y\mapsto y$, $z_2\mapsto a_2^4a_3^3$,  $z_3\mapsto a_2^3a_3^2$ form a covering closure for the original tower.
\end{example}

Closures of towers were introduced for generalizing Merzlyakov's result about formal solutions in nonabelian free groups to groups that admit the structure of a tower (see \cite[Theorem 1.18]{MR1972179}). 

\begin{fact}\label{MerzTowers}
Let $G$ be a group and $\mathcal{T}(G,\F)$ be a tower structure on $G$. Let $\F=\langle \bar{a}\rangle$ and $G=\langle \bar{x}, \bar{a} \ | \ T(\bar{x}, \bar{a})\rangle$ a (finite) presentation of $G$.
 Let  
$$\F\models \forall\bar{x}\Bigl(T(\bar{x}, \bar{a})=1\rightarrow\exists\bar{y}\bigl(\Sigma(\bar{x}, \bar y, \bar a)=1\land \Psi(\bar x, \bar y, \bar a)\neq 1\bigr)\Bigr)$$ 
where $\Sigma$ and $\Psi$ are finite sets of words in the free group $\langle \bar x, \bar y, \bar a\rangle$.

Then there exist:
\begin{enumerate}
\item  a covering set of closures, $Cl_1(\mathcal{T}), \ldots, Cl_k(\mathcal{T})$,  of $\mathcal{T}$, with corresponding groups, $Cl_i(G)$, generated by $\langle \bar{s}_i, \bar{x}, \bar{a}\rangle$,  for $i\leq k$;
\item for each closure $Cl_i(\mathcal{T})$ a "formal solution" $\bar w(\bar{s}_i, \bar{x},\bar{a})$ with $|\bar{w}|=|\bar{y}|$.
\end{enumerate}
such that for each $i\leq k$, the set of words $\Sigma(\bar{x}, \bar w(\bar{s}_i, \bar{x},\bar{a}), \bar{a})$ is trivial in $Cl_i(G)$ and in addition there exists a morphism $h_i:Cl_i(G)\rightarrow\F$, which is the identity on $\F$,  
that does not kill any of the words in $\Psi(\bar{x}, \bar w(\bar{s}_i, \bar{x},\bar{a}), \bar{a})$.
\end{fact}  

In the last subsection we will bring the two constructions, closures and multiplets, together and prove, under some natural assumption, that closures of  multiplets can be "symmetrised" with respect to the corresponding Abelian floors. 

\subsection{Symmetrising closures of multiplets of towers}

We will use the following lemma in order to ``symmetrize" the closure of a multiplet of towers. 

\begin{lemma}\label{SymmetricAbFloors}
Let $\Z\oplus\Z^m$, $\Z\oplus A^m$ and $\Z\oplus B^m$ be free Abelian groups of rank $m+1$. Let $f_i:\Z\oplus\Z^m\rightarrow \Z\oplus A^m$, for $i\leq k$, be closure embeddings such that the intersection of the corresponding cosets, $(\mu_{i,1}, \ldots, \mu_{i,m})^T+\Delta_i$, is nonempty and equal to $(\nu_1, \nu_2, \ldots, \nu_m)^T+\Delta$, where $\Delta:=\cap_{i\leq k}\Delta_i$. 

Then there exist closure embeddings, $g_i:\Z\oplus A^m\rightarrow \Z\oplus B^m$, for $i\leq k$, such that $g_i\circ f_i:\Z\oplus\Z^m\rightarrow \Z\oplus B^m$ is a closure embedding with corresponding coset $(\nu_1, \nu_2, \ldots, \nu_m)^T
+\Delta$. 
\end{lemma}
\begin{proof}
The intersection of finite index subgroups, is a finite index subgroup. Hence, $\Delta$ is a finite index subgroup of $Z^m$. In particular, there exists a closure embedding $f:\Z\oplus\Z^m\rightarrow \Z\oplus B^m$ with corresponding coset $(\nu_1, \nu_2, \ldots, \nu_m)^T+\Delta$. Finally, since $(\nu_1, \nu_2, \ldots, \nu_m)^T+\Delta\subseteq (\mu_{i,1}, \ldots, \mu_{i,m})^T+\Delta_i$ for every $i\leq k$, it is straightforward to find $g_i:\Z\oplus A^m\rightarrow \Z\oplus B^m$, such that $g_i\circ f_i=f$, for every $i\leq k$. 
\end{proof}

\begin{proposition}\label{SymmetricTowers}
Let $\mathcal{T}^N(G_N,\F)$ be the $N$-multiplet of the tower $\mathcal{T}(G,\F)$. Let $Cl(\mathcal{T}^N)$ be a closure of 
$\mathcal{T}^N(G_N,\F)$ and $\mathcal{G}(G^{i+1}, G^i)$ an Abelian floor of $\mathcal{T}(G,\F)$, whose peg does not belong to $\F$ (i.e. is an Abelian floor that does not belong to the Abelian pouch) and with free Abelian vertex group $\Z\oplus\Z^{m}$. 

Suppose, for $j\leq N$, $(\mu_{j,1}, \ldots, \mu_{j, m})^T+\Delta_j$ is the coset that corresponds to the Abelian closure of the $j-th$ copy of $\mathcal{G}(G^{i+1}, G^i)$ in $\mathcal{T}^N$ and the closure embedding $f_j:\Z\oplus\Z^m_j\rightarrow \Z\oplus A_j^m$ of $Cl(\mathcal{T}^N)$. Moreover, assume that the intersection of these cosets is nonempty and equal to $(\nu_1, \ldots, \nu_m)^T+\Delta$, where $\Delta:=\cap_{j\leq N}\Delta_j$.  

Then, there exists a closure of $Cl(\mathcal{T}^N)$, i.e. $Cl(Cl(\mathcal{T}^N))$, such that for all $j\leq N$, the closure embeddings $g_j:\Z\oplus A_j^m\rightarrow \Z\oplus B_j^m$  are such that the corresponding cosets of $g_j\circ f_j$ are all the same and equal to, $(\nu_1, \ldots, \nu_m)^T+\Delta$, the common intersection.
\end{proposition}
\begin{proof}
Apply Lemma \ref{SymmetricAbFloors} to the collection of copies of the Abelian floor $\Z\oplus\Z^m$ and the corresponding closure embeddings $f_j:\Z\oplus\Z^m_j\rightarrow \Z\oplus A_j^m$.  
\end{proof}

\begin{remark}
We note in passing that the closure of a closure of a tower is straightforwardly a closure of the original tower, since the composition of closure embeddings is a closure embedding itself. In addition, we may apply the above proposition simultaneously to all Abelian floors for which the hypotheses hold.      
\end{remark}


\section{Test sequences, Envelopes and Implicit Function Theorems}\label{TestSequences}

\subsection{Actions on real trees}\label{RealTrees}

A {\em real tree} is a $0$-hyperbolic geodesic metric space. In this subsection we record some background on how one can obtain a minimal nontrivial action on a real tree from a sequence of actions on $\delta$-hyperbolic spaces. We will use these actions in the next subsection in order to define test sequences. The method we are going to describe is called the {\em Bestvina-Paulin method} (see \cite{MR932860}\cite{MR958589}). 

We fix a finitely generated group $G$ and we consider the set of non-trivial equivariant pseudometrics $d:G\times G\to \R^{\geq 0}$, denoted by $\mathcal{ED}(G)$. 
We equip $\mathcal{ED}(G)$ with the compact-open topology (where $G$ is given the discrete topology). Note 
that convergence in this topology is given by: 
$$(d_i)_{i<\omega}\to d\ \ \textrm{if and only if} \ \ d_i(1,g)\to d(1,g)\ \ (\textrm{in $\R$}) \ \ \textrm{for all}\ \ g\in G$$


We also note that any based $G$-space $(X,*)$ (i.e. a metric space with a distinguished point equipped with an action of 
$G$ by isometries) gives rise to an equivariant pseudometric on $G$ as follows: $d(g,h)=d_X(g\cdot *,h\cdot *)$. The pseudo-metric is $\delta$-hyperbolic if $d_X$ is $\delta$-hyperbolic.  
We say that a sequence of $G$-spaces $(X_i,*_i)_{i<\omega}$ converges to a $G$-space $(X,*)$, if the 
corresponding pseudometrics induced by $(X_i,*_i)$ converge to the pseudometric induced by $(X,*)$ in $\mathcal{ED}(G)$. 

\begin{fact}\label{ConvSpaces}
Let $G$ be a finitely generated group. Let $(X_i, *_i)$, for $i<\omega$, be a sequence of based $\delta_i$-hyperbolic $G$-spaces. Suppose the sequence of the corresponding induced pseudometrics $(d_i)_{i<\omega}$ converge in $\mathcal{ED}(G)$ to a pseudometric $d_{\infty}$ and $\delta_i$ goes to $0$ as $i$ goes to $\infty$. 

Then, there exists a based real $G$-tree $(T,*)$ such that $T$ is spanned by the orbit of the base point $*$ under the action of $G$ and $d_{\infty}$ is the corresponding pseudometric.
\end{fact}

We will obtain actions on hyperbolic spaces through group morphisms as follows: 
a morphism $h:G\to H$ where $H$ is a hyperbolic group induces an action of $G$ on  
$X_H$ (the geometric realization of the Cayley graph of $H$) in the obvious way, thus making $X_H$, which is a hyperbolic space, a $G$-space. 

For a finitely generated group $G$, endowed with a fixed generating set $S$, and $g\in G$ we denote by $\abs{g}_{S}$ the word length of $g$ with respect to $S$, i.e. the length of the shortest word in $S$ representing $g$. When there is no ambiguity we will omit the subscript $S$ and denote the length of an element simply by $\abs{g}$.

\begin{definition}\label{LengthMorph}
Let $G:=\langle g_1, g_2, \ldots, g_m\rangle$ be a finitely generated group. Let $h:G\rightarrow H$ be a morphism to a finitely generated group $H$. Then the length of $h$, denoted $\abs{h}$, is the sum $\abs{h(g_1)}+\abs{h(g_2)}+\ldots+\abs{h(g_m)}$.
\end{definition}

\begin{remark}\label{CareBase}
Let $G,H$ be finitely generated and $h:G\rightarrow H$ a morphism. If we choose the base point of the (geometric realization) of the Cayley graph of $H$ to be the identity element, then the length of the corresponding pseudometric coincides with the length of $h$.  
\end{remark}

\begin{definition}[Convergent sequence]
An infinite sequence of morphisms $(h_n)_{n<\omega}:G\rightarrow \F$, from a finitely generated group $G$ to a free group $\F$ is called convergent if the sequence of pseudometrics $(d_n)_{n<\omega}$, obtained by the actions of $G$ on $X_\F$ via $h_n$, converges in $\mathcal{ED}(G)$. 
\end{definition}  

Actually, formally, a limit group is a group that can be obtained as a quotient of a finitely generated group $G$ by the kernel of an action on a real tree obtained by a convergent sequence $(h_n)_{n<\omega}:G\rightarrow\F$, thus the name. In \cite[Theorem 4.6]{MR1863735} Sela proved that the classes of limit groups and of finitely generated $\omega$-residually free groups coincide.    

Let $G$ be a finitely generated group. We say that a sequence of morphisms $(f_n)_{n<\omega}:G\rightarrow\F$ is {\em stably injective}, if for any $g\in G\setminus \{1\}$, there exists $n_g$, such that $h_n(g)\neq 1$ for all $n>n_g$. An application of equational Noetherianity of free groups yields the following result: a (f.g.) group admits a stably injective sequence if and only if it is a limit group. 

A finitely generated group $G$ that contains a nonabelian free group $\F$ as a (distinguished) subgroup is called a {\em restricted limit group} if there exists a stably injective sequence $(h_n)_{n<\omega}:G\rightarrow\F$, such that, for every $n<\omega$, $h_n\upharpoonright_\F=Id$. Groups that have the structure of a tower over $\F$ are examples of restricted limit groups. In addition, morphisms $h:G\rightarrow\F$ such that $h\upharpoonright_\F=Id$ are called {\em restricted morphisms}.

A convergent stably injective sequence from a nonabelian limit group $L$ induces a faithful action of $L$ on a real tree. The action can be further analyzed by, what is called, the Rips' machine (see \cite{MR1274119}, \cite{MR1465334}).

\begin{fact}
Let $G$ be a restricted limit group. Let $(h_n)_{n<\omega}:G\rightarrow\F$ be a convergent stably injective sequence of restricted morphisms and $G\curvearrowright T$ be the obtained action on a real tree $T$. Then either $G$ splits as a free product or $T$ has a decomposition as a graph of actions where each vertex action is either:
\begin{itemize}
    \item simplicial: a simplicial action on a simplicial tree.
    \item surface type: the vertex action $G_u\curvearrowright Y_u$ has kernel $N_u$ and the induced faithful action is dual to an arational measured foliation on a closed $2$-orbifold with boundary. 
    \item axial type: $Y_u$ is isometric to $\mathbb R$ and the image of $G_u$ in $Isom(Y_u)$ is a finitely generated group acting with dense orbits in $Y_u$.
\end{itemize}
\end{fact}

Since we will only tangentially touch on this, we refer the interested reader to \cite{MR2401220} for the notion of a graph of actions and for more details.      

For defining test sequences, in the next subsection, we will also be needing the notion of a {\em (relatively) short morphism}. For convenience of notation we denote by $G(\bar a)$ a group $G$ generated by the tuple $\bar a$.

\begin{definition}[Relatively short morphism]
Let $\F:=\langle \bar a\rangle$ be a free group. Let $L(\bar y, \bar x, \bar a)$ be a group that contains $G(\bar x, \bar a)$ and $\F$ as subgroups. A (restricted) morphism $h:L\rightarrow \F$ is short with respect to $G$, if for any (restricted) morphism $g:L\rightarrow \F$ that extends $h\upharpoonright_G$ we have $\abs h\leq \abs g$. 
\end{definition}

\subsection{Test sequences and their properties}

We fix a tower $\mathcal{T}(G,\F)$. A test sequence for $\mathcal{T}$ is a sequence of morphisms $(h_n)_{n<\omega}:G\rightarrow\F$ that are the identity on $\F$ and satisfy certain combinatorial conditions that depend on the structure of the tower. The values of test sequences are quite flexible as long as in the limit action one can still "see the tower structure". 

\begin{definition}[Test sequence]\label{TestSequence}
Let $\mathcal{T}(G, \F)$ be a tower and $G^m>G^{m-1}>\ldots>G_1>G_0:=\F*\F_n$ the sequence of the groups of its floors. Let $(h_n)_{n<\omega}:G\rightarrow \F$ be 
a sequence that witnesses that $G$ is a (restricted) limit group. 

We say that $(h_n)_{n<\omega}$ is a test sequence if for every $i\leq m$ and $L_i$, a restricted limit group freely indecomposable with respect to $G^i$, and every convergent, short with respect to $G^i$, sequence of morphisms $(f_n)_{n<\omega}:L_i\rightarrow\F$ that for each $n<\omega$, $f_n$ extends $h_n\upharpoonright_{G^i}$, the following hold:
\begin{itemize}
    \item The convergent sequence $(f_n)_{n<\omega}$ induces a faithful action of $L_i$ on a based real tree $(Y,*)$ and $L_i\curvearrowright Y$ decomposes as a graph of actions $(L_i\curvearrowright T, \{Y_u\}_{u\in V(T)}, \{ p_e\}_{e\in E(T)})$. A further analysis of this graph of actions for the different types of floors yields:
    \begin{enumerate}
    \item If $\mathcal{G}(G^i,G^{i-1})$ is a surface floor with surface type vertex group $\pi_1(\Sigma_{g,l})$, then $L_i$ admits a surface floor decomposition over a subgroup $B^{i-1}\supseteq G^{i-1}$ with the same surface type vertex group and the same edge groups.  
    
    \item If $\mathcal{G}(G^i,G^{i-1})$ is an Abelian floor with free Abelian vertex group $\Z\oplus\Z^m:=\langle z, z_1, z_2,$ $\ldots,z_m\rangle$, then $L_i$ admits an Abelian floor decomposition $B^{i-1}*_{C} C\oplus A$ over a subgroup $B^{i-1}\supseteq G^{i-1}$. Moreover, $z_m$ belongs to $C\oplus A\setminus C$, and  the edge group $C$ (which is maximal Abelian in $B^{i-1}$) contains $\Z\oplus\Z^{m-1}$.    
    \item For the ground floor, if $G^0=\F*\F_n$, then $L_0$ admits a retract onto $\F*\F_n$.
  
 \end{enumerate}
 
\end{itemize}
\end{definition}

A test sequence for $\mathcal{T}(G, \F)$ is constructed from ground floor to top floor using the retractions and the modular automorphism group of each floor.

\begin{definition}[Modular automorphism group]
Let $\mathcal{G}(G^{i+1}, G^i)$ be a floor in a tower $\mathcal{T}(G, \F)$. Then the modular group of $G^{i+1}$ with respect to $\mathcal{G}(G^{i+1}, G^i)$, denoted $Mod_{\mathcal{G}(G^{i+1}, G^i)}(G^{i+1})$,  is generated by the following automorphisms of $G^{i+1}$:
\begin{itemize}
    \item Dehn twists along edges of the splitting $\mathcal{G}(G^{i+1}, G^i)$ that fix $G^i$ pointwise.
    \item If $\mathcal{G}(G^{i+1}, G^i)$ is a surface floor, then Dehn twists along essential simple closed curves of corresponding surface that fix $G^i$ pointwise.
    \item If $\mathcal{G}(G^{i+1}, G^i)$ is a an Abelian floor with $\Z\oplus\Z^k$ the free Abelian vertex group, then the natural extension, that fixes $G^i$ pointwise, of an automorphism of $\Z\oplus\Z^k$ that fixes $\Z$. 
\end{itemize}
\end{definition}

When no confusion may arise we will denote the modular group of $G^{i+1}$ with respect to $\mathcal{G}(G^{i+1}, G^i)$, simply by  $Mod(G^{i+1})$.

After defining the modular automorphism group and before constructing a test sequence we add some pieces of definition. 

\begin{definition}\label{FactorsThrough}
Let $\mathcal{T}(G, \F)$ be a tower and $G^m>G^{m-1}>\ldots>G_1>G_0:=\F*\F_k$ the sequence of the groups of its floors. We say that a morphism $h:G\rightarrow \F$ factors though the tower $\mathcal{T}$ if there exist modular automorphisms $\alpha_i\in Mod(G^i)$, for $1\leq i\leq m$, and a (restricted) morphism $h_0:G_0\rightarrow\F$ such that $h$ has the form: 
$$h_0\circ r_1\circ \alpha_1\circ r_2\circ\alpha_2\circ\ldots\circ r_m\circ\alpha_m$$
\end{definition}

\begin{definition}
Let $(h_n)_{n<\omega}:G\rightarrow\F$ be a sequence of morphisms. For $a,b\in G$ we say that $a$ dominates the growth of $b$ under $(h_n)_{n<\omega}$ if  
$$ \frac{|h_n(b)|_{\F}}{|h_n(a)|_{\F}}\rightarrow 0 \ \textrm{as} \ n \  \textrm{goes to} \ \infty $$
And for $A, B$ subgroups of $G$, we say that $A$ dominates the growth of $B$ under $(h_n)_{n<\omega}$ if every nontrivial element of $A$ dominates the growth of any element of $B$ under $(h_n)_{n<\omega}$.   
\end{definition}

Finally, we give an informal constructive account of test sequences and rely on the proofs of \cite[Propositions 1.3, 1.8, 1.9 \& Theorem 1.18]{MR1972179} for the justification that they indeed satisfy Definition \ref{TestSequence}.

\begin{definition}
Let $\mathcal{T}(G,\F)$ be a tower and $G^m>G^{m-1}>\ldots>G_1>G_0:=\F*\F_k$ the sequence of the groups of its floors. We construct a test sequence, $(h_n)_{n<\omega}:G\rightarrow\F$, recursively: 
\begin{itemize}
    \item For $G_0:=\F*\F_k$, with $\F_k:=\langle x_1, \ldots, x_k\rangle$ we define  a sequence of (restricted) morphisms $(f_n)_{n<\omega}:\F*\F_k\rightarrow\F$ such that, for each $n<\omega$,  $f_n(\bar{x})$ satisfies the small cancellation property $C'(1/n)$. Moreover, for each $i,j<k$ there are $c_{i,j},c'_{i,j}\in \R^+$ such that $c_{i,j}<\frac{\abs{f_n(x_i)}_{\F}}{\abs{f_n(x_j)}_{\F}}<c'_{i,j}$, for all but finitely many $n$.
    \item Suppose a test sequence has been defined for the $i$-th group of the tower $G^i$. We define a test sequence for the $i+1$-th floor.
    \begin{enumerate}
        \item Suppose $(\mathcal{G}(G^{i+1}, G^i), r_{i+1})$ is a surface floor, with surface type vertex group $\pi_1(\Sigma_{g,b})$ and $(g_n)_{n<\omega}:G^i\rightarrow\F$ is a test sequence. Then, for each $n<\omega$, $h_n:G^{i+1}\rightarrow\F$ is defined by $g_n \circ r_{i+1}\circ\alpha_n$, where $\alpha_n$ is a modular automorphism in $Mod(G^{n+1})$. Moreover, $\alpha_n$ satisfies the combinatorial conditions $(i)-(ix)$ in pages 222-223 of \cite{MR1972179} and every element nonconjugate into a maximal boundary subgroup in $\pi_1(\Sigma_{g,b})$ dominates the growth of any element in $G^i$ under $(h_n)_{n<\omega}$.
        \item Suppose $(\mathcal{G}(G^{i+1}, G^i), r_{i+1})$ is an Abelian floor, with free Abelian vertex group $\Z\oplus\Z^p:=\langle z, z_1, \ldots, z_p\rangle$ $(g_n)_{n<\omega}:G^i\rightarrow\F$ is a test sequence. Then, for each $n<\omega$, $h_n:G^{i+1}\rightarrow\F$ is defined by $g_n \circ r_{i+1}\circ\alpha_n$, where $\alpha_n$ is a modular automorphism in $Mod(G^{n+1})$. Moreover,  $z^j$ dominates the growth $z_{j-1}$, for $1<j\leq p$, and $z_1$ dominates the growth of $G^i$ under $(h_n)_{n<\omega}$.
    \end{enumerate}
\end{itemize}
\end{definition}

In pages 219-226 of \cite{MR1972179} Sela proves the existence of test sequences and records some of their properties. In particular property (xiv) in page 225 states:

\begin{fact}\label{PowerPegs}
Let $\mathcal{T}(G,\F)$ be a tower and $(h_n)_{n<\omega}:G\rightarrow\F$ be a test sequence for $\mathcal{T}$. Let $g$ be the generator of the peg that corresponds to some Abelian floor $\mathcal{G}(G^{i+1}, G^i)$. 
Then, for all $n<\omega$, $h_n(g)$ is not a proper power in $\F$. 
\end{fact} 

In addition, values of elements in a chosen basis of a free Abelian vertex group of an Abelian floor are flexible to choose in the following sense (see \cite[Lemma 1.21]{MR1972179}). 

\begin{fact}\label{BasisFlex}
Let $\mathcal{T}(G,\F)$ be a tower. Let $\Z\oplus\Z^{m}=\langle z, z_1, \ldots, z_m\rangle$ be the free Abelian vertex group of an Abelian floor in 
$\mathcal{T}$ and $k_1<k_2$ be integers. Then, there exists a test sequence, $(h_n)_{n<\omega}:G\rightarrow\F$, for $\mathcal{T}$ 
such that, for any chosen $i\leq m$,  the $n$-th value, $h_n(z_i)$, of $z_i$ is of the form $h_n(z)^{rk_2+k_1}$, for some integer $r$. 
\end{fact}

Although the following fact is not formally stated in \cite{MR1972179} one can easily deduce it by tweaking the powers of Dehn twists or free Abelian automorphisms that appear in the modular automorphisms used to define a test sequence. It, informally, says that given a floor in a tower and a test sequence, we can always find a new test sequence with higher growth rate in the vertex group (surface type or Abelian type) of the floor.       

\begin{fact}\label{TweakGrowth}
Let $\mathcal{T}(G,\F)$ be a tower and $G^m>G^{m-1}>\ldots>G_1>G_0:=\F*\F_k$ the sequence of the groups of its floors and $(h_n)_{n<\omega}:G\rightarrow\F$ a test sequence for $\mathcal{T}$. Fix $i<m$ and let $\mathcal{G}(G^{i+1}, G^i)$ be the $i$-th floor of $\mathcal{T}$. Then, there exists a test sequence $(f_n)_{n<\omega}:G\rightarrow\F$ with the following properties:
\begin{itemize}
    \item $f_n\upharpoonright_{G^i}=h_n\upharpoonright_{G^i}$, for all $n<\omega$.
    \item Suppose $\mathcal{G}(G^{i+1}, G^i)$ is a surface floor, with surface type vertex group $\pi_1(\Sigma_{g,b})$. Then the  growth rate of every element in $\pi_1(\Sigma_{g,b})$ non-conjugate into a maximal boundary subgroup of $\pi_1(\Sigma_{g,b})$, under $(f_n)_{n<\omega}$ dominates the growth rate of any element in $\pi_1(\Sigma_{g,b})$ under $(h_n)_{n<\omega}$, i.e. 
    $ \frac{|h_n(b)|_{\F}}{|f_n(a)|_{\F}}\rightarrow 0$ as $n$ goes to $\infty$, for given $a,b$ as previously defined.
    \item Suppose $\mathcal{G}(G^{i+1}, G^i)$ is an Abelian floor, with free Abelian vertex group $\Z\oplus\Z^p$. Then the growth rate of every nontrivial element in $\Z^p$ under $(f_n)_{n<\omega}$ dominates the growth rate of any element in $\Z^p$ under $(h_n)_{n<\omega}$.  
\end{itemize}
\end{fact}

\begin{remark}\label{OrderingTowers} 
This is a rather crucial observation about test sequences. If $\mathcal{T}(G,\F)$ is a tower and $G^m>G^{m-1}>\ldots>G_1>G_0:=\F*\F_k$ the sequence of the groups of its floors, then there might be many different orders in which its (fixed set of) floors may appear. For example, the peg of an Abelian floor $\mathcal{G}(G^{i+1}, G^i)$ might belong to $G^j$, for some $j<i$. In this case, we may change the order of floors by "moving" the particular Abelian floor or part of it (a direct factor of the free Abelian vertex group) lower (see Lemma \ref{ChangeFloorOrder}). This applies, in particular, to multiplets of towers, where the different copies of the original tower can be "glued" in different order and, hence, may be given different "priority" when defining a test sequence.

All statements about test sequences of towers implicitly assume that some (compatible) order of floors and an order of a fixed basis for each free Abelian vertex group have been chosen. Since the particular order is not important, for simplicity of notation in our statements, we suppress this in the sequel.       
\end{remark}

\subsection{Elementary equivalence relations in towers}\label{TowersEquiv}

We will next prove some propositions that connect the properties of an element or equivalence class in a tower with the properties of its images under a test sequence. 
Recall that since towers are limit groups all elementary equivalence relations make sense in them as well. Before proving the propositions we record a "local" version of Fact \ref{MerzTowers} (its proof is an easy adaptation of the proof 
of \cite[Theorem 1.18]{MR1972179}) and we prove an easy lemma. 

\begin{fact}\label{TestMerzTowers}
Let $\mathcal{T}(G,\F)$ be a tower with $G:=G(\bar u, \bar x,\bar a)$ and $\Sigma(\bar x, \bar y, \bar a)=1$ be a system of equations over $\F$. Let $(h_n)_{n<\omega}:G\rightarrow \F$ be a test sequence for $\mathcal{T}$. Suppose   for every $n$, there exists $\bar c_n$ in $\F$ such that $\F\models \Sigma(h_n(\bar x), \bar c_n, \bar a)=1$. 

Then there exists a closure, $Cl(\mathcal{T})$, with $Cl(G):=Cl(G)(\bar v,\bar u, \bar x, \bar a)$, and a tuple of words $\bar w(\bar v, \bar u, \bar x, \bar a)$ such that 
\begin{itemize}
\item a subsequence of $(h_n)_{n<\omega}$ extends to a test sequence of $Cl(\mathcal{T})$.  
\item every word in $\Sigma(\bar x, \bar w(\bar v, \bar u, \bar x, \bar a), \bar a)$ is trivial in $Cl(G)$.  
\end{itemize}
\end{fact}

\begin{lemma}\label{SameFloorConj}
Let  $\mathcal{T}(G,\F)$ be a tower and $G^m>G^{m-1}>\ldots>G_1>G_0:=\F*\F_n$ the sequence of the groups of its floors. Let $g, h$ be conjugates in $G$. Suppose $i=height(g)\geq height(h)$. Then any conjugating element $\gamma$ has height at most $i+1$. As a matter of fact, $\gamma$ has height at most $i$, except when the $height(\gamma)$-floor is Abelian and either $h$ or $g$ is conjugated to the edge group.  

In addition, there is always an element of height at most $i$ that conjugates $g$ to $h$.
\end{lemma}
\begin{proof}
We assume for a contradiction that $g^\gamma=h$ for some $\gamma$ of height $j$ with $j>i+1$. We take cases according to whether the floor $\mathcal{G}(G^{j}, G^{j-1})$ is surface or Abelian.
\begin{itemize}
\item Suppose $\mathcal{G}(G^{j}, G^{j-1})$ is a surface floor. Recall that since the edge groups embed onto boundary subgroups which are maximal Abelian (and pairwise non-conjugate) in the fundamental group of the surface the action of $G^j$ that corresponds to the splitting is $1$-acylindrical near the surface type vertices (i.e. the vertices in the orbit of the surface type vertex). Now since $\gamma$ does not fix the vertex stabilized by $G^{j-1}$, say $x$, we have that the segment $[x, \gamma.x]$ must be fixed by $h$, but this is a path of length strictly greater than $1$ containing a surface type vertex, hence $h$ is trivial, a contradiction. 
\item Suppose $\mathcal{G}(G^{j}, G^{j-1})$ is an Abelian floor with free Abelian vertex group $\Z\oplus\Z^m$. The action of $G^j$ that corresponds to this splitting is $1$-acylindrical near the vertices which are not fixed by conjugates of the free Abelian group. Now, since $\gamma$ does not fix the vertex stabilized by $G^{j-1}$, say $x$, we must have that the length of the segment $[x,\gamma.x]$ is not greater than $2$. In particular, $\gamma$ admits a normal form $z_1g_1$, where $g_1$ might be trivial, up to conjugation by an element of $G^{j-1}$ (that we may omit since we can conjugate $g$ and $h$ by the inverse of the same element). In addition, $h$, since it fixes $[x,\gamma.x]$, must belong to the (image of the) edge group $\Z$ (in $G^{j-1}$). Hence, we may assume that $j-1=i$, again a contradiction. Notice that $h$ in this case commutes with any element of the free Abelian vertex group $\Z\oplus\Z^m$.    
\end{itemize}  

Finally, for the last claim of the lemma, we argue as follows. We only need to check the case where the $height(\gamma)$-floor is Abelian. Since $g^\gamma=h$, we get $z_1g_1gg_1^{-1}z_1^{-1}=h$, that yields, that  $g^{g_1}=h$. Since $g_1$ belongs to $G^i$ we get what we want. 
\end{proof}

\begin{proposition}\label{ConjDiff}
Let $\mathcal{T}(G,\F)$ be a tower and $(f_n)_{n<\omega}:G\rightarrow\F$ be a test sequence for $\mathcal{T}$. Let $g, h$ be non-commuting conjugation elements in $G$. Then $f_n(g), f_n(h)$ are non-commuting conjugation elements in $\F$ for all but finitely many $n<\omega$.
\end{proposition}
\begin{proof}
For the sake of clarity, we first give a proof for a hyperbolic tower $\mathcal{T}$. Suppose, for a contradiction, that $f_n(g), f_n(h)$ are commuting conjugation elements in $\F$ for infinitely many $n<\omega$. The refined sequence consisting of these indices is still 
a test sequence, which we still denote by $(f_n)_{n<\omega}$. Hence we have $\F\models \exists z ([f_n(g),f_n(h)^z]=1)$. In particular, by Fact \ref{TestMerzTowers}, there exists a formal solution $w$ in $G$ such that $[g,h^w]=1$ in $G$, a contradiction. 

If $\mathcal{T}$ is not hyperbolic, then the previous argument gives that $g$ and $h$ are commuting conjugation elements in a closure, $Cl(\mathcal{T})$, of $\mathcal{T}$. We note in passing that if $g,h$ are commuting conjugation elements, then any conjugates of $g$ and $h$ are commuting conjugation elements. Without loss of generality we may assume that $i=height([g]^1_1)\geq height([h]^1_1)$ in $\mathcal{T}$, where by $height([g]^1_1)$ we take $min\{height(a) \ | \ a\in [g]^1_1\}$. As we may restrict the test sequence to $\mathcal{T}_i$, we can use Fact \ref{TestMerzTowers} on $\mathcal{T}_i$ and obtain $w\in Cl(G^i)$ such that $[g,h^w]=1$. If $i=0$, then $\mathcal{T}_0$ is hyperbolic and the result follows. If $i>0$, we take cases according to whether the floor $\mathcal{G}(G^i, G^{i-1})$ is Abelian or surface. 
\begin{itemize}
\item Suppose $\mathcal{G}(G^i, G^{i-1})$ is an Abelian floor of the form $G^{i-1}*_\Z (\Z\oplus \Z^{m_i})$. If $g$ is elliptic, then by the choice of $i$ it must be conjugated into the free Abelian vertex group. 
In particular, we may assume that $g$ belongs to $\Z\oplus \Z^{m_i}$. Since $h^w$ commutes with $g$ in $Cl(G^i)$ and the closure, $\Z\oplus A^{m_i}$, of $\Z\oplus \Z^{m_i}$ is maximal Abelian, we must have that $h^w$ belongs to $\Z\oplus A^{m_i}$. Hence $h$ is elliptic as well but in the splitting $Cl(G^{i-1})*_\Z(\Z\oplus A^{m_i})$. We observe that since the edge group in both splittings, $G^{i-1}*_\Z (\Z\oplus \Z^{m_i})$ and $Cl(G^{i-1})*_\Z(\Z\oplus A^{m_i})$, remains the same, ellipticity/hyperbolicity of elements of $G^i$ does not change when we see them as elements of $Cl(G^i)$. Now, since $h\in G^i$, it is elliptic in $G^{i-1}*_\Z (\Z\oplus \Z^{m_i})$. In particular, it can be conjugated into the free Abelian group that $g$ lives in by an element of $G^i$. As a matter of fact $w$ is a product $w_1\cdot w_2$, with $w_2\in G^i$ and $w_1$ in $\Z\oplus A^{m_i}$. Hence, $h^{w_2}$ commutes with $g$, a contradiction. 

Now assume that $g$ is hyperbolic. 
Since $h^w$ commutes with $g$ and cannot fix a subtree that contains a line (the action is acylindrical), it must be hyperbolic as well.
In addition, for the same reason,  $w.A_h=A_g$ in the tree that corresponds to the splitting $Cl(G^{i-1})*_\Z(\Z\oplus A^{m_i})$. Let $x$ be the vertex stabilized by $Cl(G^{i-1})$ and $y$ the vertex stabilized by $\Z\oplus A^{m_i}$. Note that the edge between $x$ and $y$ is stabilized by $\Z$. We may assume, up to conjugation by an element of $G$, that the edge $[x,y]$ belongs to both axes. Hence, $w.[x,y]=g.[x,y]$ for some $g\in G$. In particular, $g^{-1}w$ fixes $[x,y]$ and therefore it belongs to $\Z$, consequently $w$ belongs to $G$, a contradiction. 

\item Suppose $\mathcal{G}(G^i, G^{i-1})$ is a surface floor with surface type vertex group $\pi_1(\Sigma_{g,b})$. We first assume that $g$ is elliptic. By the choice of $i$ and $1$-acylindricity near the surface type vertex, we may assume, up to conjugation, that $g$ fixes the vertex fixed by $\pi_1(\Sigma_{g,b})$ and no other vertex. The same holds in the splitting of $Cl(G^i)$. In particular, the centralizer of $g$ in $Cl(G^i)$ is a subgroup of the surface group $\pi_1(\Sigma_{g,b})$. Hence, $h^w$, also belongs to $\pi_1(\Sigma_{g,b})$ and in particular it is elliptic in the splitting of $Cl(G^i)$. But then $h$ is also elliptic and since $h$ belongs to $G^i$, we get, as in the previous case, that $h$ is elliptic in $\mathcal{G}(G^i, G^{i-1})$. Therefore, we may assume that, up to conjugation, $h$ belongs to $\pi_1(\Sigma_{g,b})$. Now, since $h^w$ also belongs to the same vertex group, we get, by $1$-acylindricity, that $w$ belongs to $\pi_1(\Sigma_{g,b})$, a contradiction. 

We next assume that $g$ is hyperbolic. 
Thus, by the standard argument as in the previous case $h$ must be hyperbolic as well and in addition $w.A_h=A_g$. Both axes must contain an edge that is in the orbit of a common edge, say $e$, under the action of $G^i$. Hence, up to conjugation by elements of $G^i$, we may assume that $e$ belongs to both axes. Therefore, $w.e=\gamma e$, for some $\gamma$ in $G^i$. Now since $Stab(e)\subseteq G^i$, we have that $\gamma^{-1}w$ belongs to $G^i$, therefore, $w$ belongs to $G^i$, a contradiction.  

\end{itemize}
\end{proof}

We prove similar results for the families of cosets and double cosets with conjugation.

\begin{proposition}\label{CosetDiff}
Let $\mathcal{T}(G,\F)$ be a tower. Suppose $(g_1,h_1)$ and $(g_2,h_2)$ are not $_cE_2$ equivalent in $G$. Then, for any test sequence, $(f_n)_{n<\omega}:G\rightarrow\F$, $(f_n(g_1), f_n(h_1))$ and $(f_n(g_2), f_n(h_2))$ are not $_cE_2$ equivalent in $\F$, for all but finitely many $n$. 
\end{proposition}
\begin{proof}
Since $(g_1,h_1)$ and $(g_2,h_2)$ are not $_cE_2$ equivalent, either $h_1, h_2$ are not commuting conjugates or for every $\gamma$ such that $[h_1^\gamma,h_2]=1$ we have that  $g_2^{-1}g_1^\gamma$ does not belong to $C_G(h_2)$. In the first case, the result follows from Proposition \ref{ConjDiff}. For the latter case, we assume, for a contradiction, that for a test sequence $(f_n)_{n<\omega}$ and for each of infinitely many indices the following holds: 
$$\F\models\exists z,w \bigl([f_n(h_1)^w,f_n(h_2)]=1 \land [f_n(h_2),z]=1\land f_n(g_2^{-1})f_n(g_1)^w=z \bigr)$$
By Fact \ref{TestMerzTowers} applied to the refined sequence, still denoted $(f_n)_{n<\omega}$, there exists a closure $Cl(\mathcal{T})(Cl(G), \F)$, and elements $z,w$ in $Cl(G)$, such that $h_1^w$ and $z$ commute with $h_2$, and $g_2^{-1}g_1^w=z$. We follow the analysis of cases of the proof of Proposition \ref{ConjDiff}. We assume that $i=height([h_2]^1_1)\geq height([h_1]^1_1)$ in $\mathcal{T}$, where by $height([h]^1_1)$ we take $min\{height(a) \ | \ a\in [g]^1_1\}$. Note that if $(g_1,h_1)$ is $_cE_2$-equivalent to $(g_2,h_2)$ in $G$, then for any $\gamma,\delta\in G$, we have that $(g_1^\gamma,h_1^\gamma)$ is $_cE_2$-equivalent to $(g_2^\delta,h_2^\delta)$.

Now, since $h_1^w$ commutes with $h_2$, we get, from Lemma \ref{SameFloorConj}, that $w$ belongs to $Cl(G^i)$, unless $h_2$ is conjugated to the peg of an Abelian floor of a higher floor that we may assume it is the $(i+1)$ floor. In the latter case, assume for a contradiction that $h_1^w$ belongs to the free Abelian vertex group of the $(i+1)$-th floor and not to the edge group of this floor. But in this case $h_1$ as well must belong to a conjugate of the free Abelian vertex group and not to the conjugate of the edge group, so it must have height (up to conjugation) $i+1$, a contradiction to our assumption. In particular, again by Lemma \ref{SameFloorConj}, $w$ has height at most $i+1$.

We deal with the case where $h_2$ is not conjugated to a peg of an Abelian floor. From the analysis of the previous proposition, we have that either $w$ belongs to $G$ and consequently $g_2^{-1}g_1^w$ belongs to $G$ and therefore $z$ belongs to $G$, or the $i$-th floor is Abelian and $w=w_1w_2$, where $w_2$ belongs to $G$, $h_2$ belongs to the free abelian vertex group and $w_1$, $h_1^{w_2}$ commute with $h_2$. In the latter case, we have that $g_2^{-1}w_1g_1^{w_2}=zw_1$. We may assume that neither $g_2$ nor $g_1^{w_2}$ commutes with $w_1$, otherwise $g_2^{-1}g_1^{w_2}$ commutes with $h_2$ and we get a contradiction. Hence, arguing by normal forms, this is only possible if $w_1$ belongs to $G$. Indeed, suppose $w_1$ does not belong to $G$, we take cases according to the heights of $g_2, w_1$ and $g_1^{w_2}$. Assume first that $w_1$ has the largest height. Let $\gamma_1\gamma_2\ldots \gamma_m$ be the normal form of $g_2^{-1}$ and $\delta_1\delta_2\ldots \delta_k$ be the normal form of $g_1^{w_2}$ with respect to the $i$-th floor decomposition (which is Abelian). Then, since $\gamma_m w_1 \delta_1$ can at worst be an element of the free Abelian vertex group that does not belong to the edge group (this is because $w_1\in \Z\oplus A^{m_i}\setminus \Z\oplus \Z^{m_i}$) we get that $\gamma_1\gamma_2\ldots \gamma_m w_1 \delta_1\delta_2\ldots \delta_k$ will have a normal form with either $\gamma_1$ in $G^{i-1}$ or $m>1$ and $\gamma_2$ in $G^{i-1}$, but the normal form of $zw_1$ has length $1$ and belongs to $Z\oplus A^{m_i}$, a contradiction. Next assume that $g_2$ or $g_1^{w_2}$ has the largest height, say $j>i$. As before let $\gamma_1\gamma_2\ldots \gamma_m$ be the normal form of $g_2^{-1}$ and $\delta_1\delta_2\ldots \delta_k$ be the normal form of $g_1^{w_2}$ with respect to the $j$-th floor decomposition. In order to complete the proof we need to show  that $\gamma_m w_1 \delta_1$ does not belong to an edge group of the decomposition. This is indeed the case, since every edge group in the various decomposition of $Cl(G)$ belongs to $G$, while a recursive argument going deeper down to level $i$ if necessary for the normal forms of $\gamma_m$ and $\delta_1$ shows that $\gamma_m w_1 \delta_1$ does not belong to $G$. Hence, $w=w_1w_2$ belongs to $G$, therefore $z$ belongs to $G$ as well, again a contradiction. 

Now suppose that $h_2$ belongs to a peg of an Abelian floor of higher level and we may assume that it is of level $i+1$. Recall that in this case $w$ has height at most $i+1$. We will next show that if $height(w)=i+1$, then $w$ must belong to the free Abelian vertex group of the $(i+1)$-floor. We consider the $(i+1)$-th Abelian floor splitting, since $w$ has height $i+1$, it will move the vertex stabilized by $Cl(G^i)$, say $x$, to a vertex different than the vertex stabilized by the free Abelian vertex group of the splitting, say $y$. Since $h_1^w=g$ we have that $w.A_{h_1}=A_g$. The axis (characteristic set) of $g$ is either the vertex $y$, or the star around $y$. Similarly, the axis of $h_1$ is either the vertex $x$, or the star around $y$. Hence, the only possibility is that both $h_1, g$ belong to the edge group of the splitting and $w$ in the free Abelian vertex group. In this case, we have $g_2^{-1}g_1^w=z$, where both $z$ and $w$ commute with $h_2$. We may assume that neither $g_2$ nor $g_1$ commutes with $w$, otherwise $g_2^{-1}g_1$ commutes with $h_2$ and we are done (since also $h_1$ commutes with $h_2$). If $w$ does not belong to $G$, we argue, using normal forms, as in the previous paragraph to obtain a contradiction. Hence, $w$ must belong to $G$ and we are done. We are finally left with the case where $w$ has height at most $i$, but this case also reduces to the previous paragraph.

\end{proof}

One has to work slightly harder to get the following.

\begin{proposition}\label{DCosetDiff}
Let $\mathcal{T}(G,\F)$ be a tower. Suppose $(z_1,g_1,h_1)$ and $(z_2,g_2,h_2)$ are not $_cE_3$ equivalent in $G$. Then, for any test sequence, $(f_n)_{n<\omega}:G\rightarrow\F$, $(f_n(z_1),f_n(g_1), f_n(h_1))$ and $(f_n(z_2),f_n(g_2), f_n(h_2))$ are not $_cE_3$ equivalent in $\F$, for all but finitely many $n$. 
\end{proposition}
\begin{proof}
Since $(z_1,g_1,h_1)$ and $(z_2,g_2,h_2)$ are not $_cE_3$ equivalent, either $z_1, z_2$ and $h_1, h_2$ are not commuting conjugates (by the same element) or for every element $\gamma$ such that $z_1^\gamma$ commutes with $z_2$ and $h_1^\gamma$ commutes with $h_2$, we have that $a g_1^\gamma b\neq g_2$, for any $a$ that commutes with $z_2$ and any $b$ that commutes with $h_2$. 

For the former case, we assume for a contradiction, that there exists a closure $Cl(\mathcal{T})(Cl(G), \F)$ and $w\in Cl(G)$ such that $z_1^w$ commutes with $z_2$ and $h_1^w$ commutes with $h_2$. Without loss of generality $i=height([z_2]_1^1)$ is the largest and all elements belong to $Cl(G^i)$ (since we may apply the test sequence truncated at $G^i$). Following the proof of Proposition \ref{ConjDiff} the only case where $w$ does not belong to $G$, is when the $i$-th floor is of Abelian type and $z_2$ is elliptic (and necessarily belongs to the free abelian vertex group and not in the edge group). As a matter of fact, in this case, $w=w_1w_2$, where $w_2$ is in $G^i$ and conjugates $z_1$ to the free Abelian vertex group, while $w_1$ belongs to the free Abelian vertex group. Now, we takes cases for the couple $h_1, h_2$. If $h_2$ acts hyperbolicaly, then similarly to the proof of Proposition \ref{ConjDiff} we get that $w$ belongs to $G^i$ (one has to be careful as we are not allowed to conjugate $h_2$, but this does not cause any serious difficulty as we can conjugate $h_1$). If $h_2$ is elliptic and it belongs to a conjugate of the free Abelian vertex group, then it is not hard to see, arguing by normal forms, that either $w_1$ belongs to $G^i$ or $h_2$ must belong to the free Abelian vertex group itself and therefore $h_1^{w_2}$ commutes with $h_2$ as well. Finally, suppose $h_2$ is a conjugate of a nontrivial element of $G^{i-1}$ (which does not belong to the peg of the free Abelian vertex group), say $\alpha g \alpha^{-1}$. Then, $h_1^w$ belongs to the centralizer of that element and consequently $h_1^{w_2}$, which belongs to $G^i$, must belong to  $w_1^{-1}\alpha C(g)\alpha^{-1} w_1$ as well. But the centralizer of $g$ has trivial intersection with the free Abelian vertex group that $w_1$ belongs to, hence $w_1$ again must belong to $G^i$, and we can conclude the proof of this case.

For the latter case, we assume for a contradiction, that there exists a closure $Cl(\mathcal{T})(Cl(G), \F)$ and $w, c, d\in Cl(G)$ such that $z_1^w$ and $c$ commute with $z_2$, $h_1^w$ and $d$ commute with $h_2$, and $cg_1^wd=g_2$. 

We first assume that $w$ belongs to $Cl(G)\setminus G$. We take cases according to the various heights. We may assume that $i=height([z_2]_1^1)\geq height([g]_1^1)$, for $g$ any of the $z_1, h_1, h_2$. If $height(w)=j>i$, then we must have that the $j$-th floor is of Abelian type and $w=w_1w_2$, where $w_1$ belongs to the free Abelian vertex group, $\Z\oplus A^{m_j}\setminus \Z\oplus \Z^{m_j}$, and $w_2$ belongs to $G^{j-1}$. In addition, $z_2$ and $h_2$ belong to the edge group of the $j$-th floor and $z_1^{w_2}$, $h_1^{w_2}$ commute with that edge group. Hence, we get $w_1cw_2g_1w_2^{-1}dw_1^{-1}=g_2$. In this case, if both $w_1c$, $dw_1^{-1}$ belong to $G^j$, then we may replace $w$ with $w_2$, $c$ with $w_1c$ and $d$ with $dw_1^{-1}$ and get a contradiction. Similarly if $g_1^{w_2}$ commutes with the edge group and $w_1c\cdot dw_1^{-1}$ belongs to $G^j$. In any other case, arguing by normal forms with respect to the $j$-th floor splitting, we get that $w_1cw_2g_1w_2^{-1}dw_1^{-1}$ belongs to $Cl(G)\setminus G$, while $g_2$ belongs to $G$, a contradiction. We are left with the case where $height(w)=j\leq i$. But also in this case we must have that the $i$-th floor is of Abelian type, and $w=w_1w_2$, where $w_1$ belongs to the free Abelian vertex group, $\Z\oplus A^{m_i}\setminus \Z\oplus \Z^{m_i}$, and $w_2$ belongs to $G^{i-1}$. In addition, $z_2$ and $h_2$ belong to the edge group of the $j$-th floor and $z_1^{w_2}$, $h_1^{w_2}$ commute with that edge group. The proof now is identical to the previous case.




We next assume that $w$ belongs to $G$. 
We first observe that if $z_2$ (or symmetrically $h_2$) does not commute with a conjugate of a peg, then its centralizer is contained in $G$, hence $c$ belongs to $G$ and consequently, as $g_1, g_2, w$ belong to $G$, $d$ does as well, a contradiction. Hence, we may assume that both $z_2, h_2$ commute with a conjugate of a peg  and both $c, d$ do not belong to $G$. Say $z_2$ commutes with $\gamma p_1\gamma^{-1}$, for some $\gamma\in G$ and $h_2$ commutes with $\delta p_2 \delta^{-1}$, for some $\delta\in G$. Without loss of generality we may assume that $p_1$ is not lower than $p_2$ in $\mathcal{T}$. Using that $c$ is in $C_{Cl(G)}(p_1)^\gamma$ and $d$ in $C_{Cl(G)}(p_2)^\delta$, we get that $\gamma^{-1}g_1^w\delta=\hat{c}\gamma^{-1}g_2\delta \hat{d}$, where $\hat{c}$ commutes with $p_1$ and $\hat{d}$ commutes with $p_2$.  We next take cases according to which element in the last equality has maximum height. We denote by $ht(g)$ the height of the element $g$ in $\mathcal{T}$. We also set, for convenience of notation,  $\gamma^{-1}g_1^w\delta:=\hat{g}_1$ and $\gamma^{-1}g_2\delta:=\hat{g}_2$.   
\begin{itemize}
    \item Suppose $i=ht(\hat{c})\geq max\{ht(\hat{d}), ht(\hat{g}_1), ht(\hat{g}_2)\}$. Then, since $\hat{c}$ commutes with an Abelian peg, $\mathcal{G}(Cl(G)^i,Cl(G)^{i-1})$ must be an Abelian floor and $\hat{c}$ must belong to the free Abelian vertex group and not to $Cl(G)^{i-1}$. Let $\Z\oplus \Z^{m_i}$ be the free Abelian group in $\mathcal{T}$ and $\Z\oplus A^{m_i}$ its closure in $Cl(\mathcal{T})$. We now take cases according to whether $\hat{g}_1$ commutes with $p_1$ or not. 
    \begin{itemize}
        \item Suppose $\hat{g}_1$ commutes with $p_1$. Then $\hat{c}^{-1}\hat{g}_1$ belongs to $\Z\oplus A^{m_i}\setminus\Z\oplus\Z^{m_i}$ and consequently $\hat{g}_2\hat{d}$ must belong too. In particular, $\hat{d}$ must have an element in its reduced form with respect to $\mathcal{G}(Cl(G)^i,Cl(G)^{i-1})$ that belongs to $\Z\oplus A^{m_i}\setminus\Z\oplus\Z^{m_i}$. But since $\hat{d}$ commutes with a peg, it must commute with $p_1$ and belong to $\Z\oplus A^{m_i}\setminus\Z\oplus\Z^{m_i}$. Thus, $\hat{g}_2$ must belong to $\Z\oplus A^{m_i}$ as well, and because $\hat{g}_2$ belongs to $G$, it must belong to $\Z\oplus\Z^{m_i}$. In particular, $\hat{g}_1=g\hat{g}_2$ for some $g\in\Z\oplus\Z^{m_i}$ (take $g=\hat{g}_1\hat{g}_2^{-1}$). Therefore, $g_1^w=\gamma g \gamma^{-1} g_2$,  a contradiction. 
    
        \item Suppose $\hat{g}_1$ does not commute with $p_1$. We observe that the reduced form, with respect to $\mathcal{G}(Cl(G)^i,Cl(G)^{i-1})$, of  $\hat{c}^{-1}\hat{g}_1$ begins with an element of $\Z\oplus A^{m_i}\setminus\Z\oplus\Z^{m_i}$ ( $\hat{c}^{-1}\cdot \hat{g}_1$ may only partially cancel since $\hat{g}_1$ is in $G$). In particular, the reduced form of $\hat{d}$ must contain an element of $\Z\oplus A^{m_i}\setminus\Z\oplus\Z^{m_i}$ (as this is the only way an element of $\Z\oplus A^{m_i}\setminus\Z\oplus\Z^{m_i}$ can appear in $\hat{g}_2\hat{d}$). Consequently, since $\hat{d}$ commutes with a peg, it must commute with $p_1$ and therefore it belongs to $\Z\oplus A^{m_i}\setminus\Z\oplus\Z^{m_i}$. But this is a contradiction, since either $\hat{g}_2$ must commute with $p_1$ and then $\hat{g}_2\hat{d}$ has length $1$, or $\hat{g}_2$ does not commute with $p_1$ and the reduced forms $\hat{c}^{-1}\hat{g}_1$ and $\hat{g}_2\hat{d}$ begin with a different factor. 
    \end{itemize}    
    \item Suppose $i=ht(\hat{g}_1)\geq max\{ht(\hat{d}), ht(\hat{c}), ht(\hat{g}_2)\}$. We further take sub-cases according to the type of the $i$-th floor.
    \begin{enumerate}
        \item Suppose $\mathcal{G}(Cl(G)^i,Cl(G)^{i-1})$ is a surface floor with surface type vertex group $\pi_1(\Sigma_{g,n})$. In this case both $\hat{c}, \hat{d}$ belong to $Cl(G)^{i-1}$. Moreover, we may see $\mathcal{G}(Cl(G)^i,$ $Cl(G)^{i-1})$ as a sequence of an amalgamated free product and several HNN extensions, i.e. $Cl(G)^i_0:=Cl(G)^{i-1}*_\Z \pi_1(\Sigma_{g,n})$, $Cl(G)^i_1:=Cl(G)^i_0*_\Z$, $\ldots$, $Cl(G)^{i}=Cl(G)^i_k:=Cl(G)^i_{k-1}*_\Z$. If there are no HNN extensions, then the reduced form of $\hat{c}^{-1}\hat{g}_1$ starts with an element of $Cl(G)^{i-1}\setminus G^{i-1}$ and the reduced form of $\hat{g}_2\hat{d}$ starts with an element which lies either in $G^{i-1}$ or in $\pi_1(\Sigma_{g,n})$ (there can only be partial cancelation between $\hat{g}_2$ and $\hat{d}$ because the reduced form of $\hat{g}_2$ must contain an element in $\pi_1(\Sigma_{g,n})\setminus G^{i-1}$). One further notices that an element of $G^{i-1}$ does not belong to a coset of an element of $Cl(G)^{-1}\setminus G^{i-1}$ relative to the edge group, hence the reduced forms do not give the same element, a contradiction. So, we may assume that there is at least one HNN extension and, without loss of generality, we may also assume that $\hat{g}_1$ does not belong to any subgroup of $Cl(G)^i$ in the sequence of  $Cl(G)^i_j$, for $j<k$. Then $\hat{g}_1$ has a reduced form with respect to $Cl(G)^i_{k-1}*_\Z$, i.e. $\hat{g}_1=\hat{g}^0_1t^{\epsilon_1}\hat{g}^1_1\ldots \hat{g}^{m-1}_1t^{\epsilon_m}$, of length strictly greater than $0$. Again, $\hat{c}^{-1}\hat{g}_1$ has a reduced form that starts with an element in $Cl(G)^{i-1}\setminus G^{i-1}$ and the reduced form of $\hat{g}_2\hat{d}$ starts with an element which lies either in $G^{i-1}$ or in $\pi_1(\Sigma_{g,n})$ or is the Bass-Serre element $t^{\pm 1}$ (there can only be partial cancellation between $\hat{g}_2$ and $\hat{d}$ because the reduced form of $\hat{g}_2$ must contain an element in $\pi_1(\Sigma_{g,n})\setminus G^{i-1}$ and/or the Bass-Serre element $t^{\pm 1}$).
        \item Suppose $\mathcal{G}(Cl(G)^i,Cl(G)^{i-1})$ is an Abelian floor. If $\hat{c}$ has a reduced form that contains an element in $Cl(G)^i\setminus Cl(G)^{i-1}$, then $p_1$ is the peg of this Abelian floor and $ht(\hat{c})\geq max \{ht(\hat{g}_1),ht(\hat{g}_2), ht(\hat{d})\}$ and this case has been tackled before. Hence we may assume that both $\hat{c},\hat{d}$ belong to $Cl(G)^{i-1}$ and do not commute with the peg of this floor. In this case, the reduced form of $\hat{c}^{-1}\hat{g}_1$ starts with an element in $Cl(G)^{i-1}\setminus G^{i-1}$ and the reduced form of $\hat{g}_2\hat{d}$ starts with an element either in $G^{i-1}$ or in the free Abelian vertex group of the Abelian floor $\mathcal{G}(Cl(G)^i,Cl(G)^{i-1})$, a contradiction.    
    \end{enumerate}
    \item Suppose $ht(\hat{g}_2)\geq max\{ht(\hat{c}), ht(\hat{d}), ht(\hat{g}_1)\}$. This case is symmetric to the previous case. 
\end{itemize}
In all cases, we cannot have $\hat{g_1}=\hat{w}_1\hat{g}_2\hat{w}_2$, which implies that we cannot have $g_1=w_1g_2w_2$, and the Proposition is proved. 
\end{proof}

\subsection{Infinitude of images}

In this subsection we prove some weaker results but for every basic equivalence relation. 

\begin{lemma}\label{finiteConj}
Let $\mathcal{T}(G,\F)$ be a tower and $g\in G$. Let $(h_n)_{n<\omega}:G\rightarrow \F$ be a test sequence for $\mathcal{T}$. If $g$ cannot be conjugated into $\F$, then it has infinitely many images under $(h_n)_{n<\omega}$ up to conjugacy.  
\end{lemma}
\begin{proof}
Suppose, for a contradiction, that $\{h_n(g)\ | \ n<\omega\}$ is finite. Then, after refining $(h_n)_{n<\omega}$ (that, by abusing notation, is still denoted $(h_n)_{n<\omega}$), we may assume that it has cardinality $1$. In particular, there exists $a\in \F$, such that, for every $n<\omega$, $\F\models \exists y (h_n(g)^y=a)$. Therefore, there exists a closure, $Cl(G)$, of $G$ and an element $w\in Cl(G)$ such that $g^w=a$. Since $a\in\F$ we have that $i=height(g)\geq height(a)$. If $height(g)=0$, then, by Lemma \ref{SameFloorConj}, $g$ and $a$ are conjugates in $\F*\F_n$, and consequently, by \cite[Chapter IV, Theorem 1.4]{MR1812024}, $g$ can be conjugated in $\F$. Suppose the height of $g$ is $i+1$, for some natural number $i$. We consider the splitting $\mathcal{G}(G^{i+1}, G^{i})$. We note that $a$ is elliptic in this splitting as it fixes the vertex stabilized by $G^i$. If $g$ is hyperbolic, then it is also hyperbolic in $\mathcal{G}(Cl(G)^{i+1}, Cl(G)^{i})$ and consequently $g^w$ is, a contradiction. Therefore, we may assume that $g$ is elliptic. Now either a conjugate of $g$ belongs to $G^i$ and the result follows by induction or $g$ does not fix any vertex which is a translate of the vertex stabilized by $G^i$. The latter still holds in $\mathcal{G}(Cl(G)^{i+1}, Cl(G)^{i})$, i.e. $g$ does not fix any translate of the vertex stabilized by $Cl(G)^{i})$, and we get a contradiction.     
\end{proof}

\begin{lemma}\label{CentrPouch}
Let $\mathcal{T}(G,\F)$ be a tower and $H$ its Abelian pouch. Let $a$ be a nontrivial element in $\F$. Then $C_{G}(a)\subseteq H$. 
\end{lemma}
\begin{proof}
Let $c$ be an element that commutes with $a$. We may assume that $i=height(c)>height(H)$. We consider the splitting $\mathcal{G}(G^i, G^{i-1})$. We denote by $x$ the vertex stabilized by $G^{i-1}$ in the Bass-Serre tree that corresponds to this splitting. The element $a$ is elliptic in this splitting and if the splitting corresponds to an Abelian floor then $a$ fixes only the vertex $x$ (all pegs that belong to $\F$ are in lower floors). On the other hand, if it corresponds to a surface floor, then, by $1$-acylyndricity near the surface type vertex, the sub-tree $a$ fixes can only contain one vertex from the orbit of $x$. In both cases, the element $c$ cannot be hyperbolic in $\mathcal{G}(G^i, G^{i-1})$. Hence, it is elliptic. But then it either fixes $x$, which is a contradiction to the choice of $i$, or it moves $x$ outside the sub-tree fixed by $a$, which contradicts the assumption that $c$ commutes with $a$.         
\end{proof}

\begin{proposition}\label{finGenERPouch}
Let $\mathcal{T}(G,\F)$ be a tower. Let $H$ be the Abelian pouch of $\mathcal{T}$. Suppose $g_1, g_2, \ldots, g_n$ belong to $G$ and the $n$-tuple $(g_1, g_2, \ldots, g_n)$ does not depend on the Abelian pouch under some generalized equivalence relation $E_{\{g\}}$. Let $(h_n)_{n<\omega}:G\rightarrow \F$ be a test sequence for $\mathcal{T}$. Then $(g_1, g_2, \ldots, g_n)$ has infinitely many images, up to the $E_{\{g\}}$ equivalence relation, under $(h_n)_{n<\omega}$.
\end{proposition}

The proof of the above proposition is not hard but involved notation-wise. We tackle a special case below and leave the general case for the reader.

\begin{lemma}\label{finConjCosetsPouch}
Let $\mathcal{T}(G,\F)$ be a tower. Let $H$ be the Abelian pouch of $\mathcal{T}$. Suppose $g_1, g_2$ belong to $G$ and the couple $(g_1, g_2)$ does not depend on the Abelian pouch under the equivalence relation, $_c E_2^m$ for some natural number $m$. Let $(h_n)_{n<\omega}:G\rightarrow \F$ be a test sequence for $\mathcal{T}$. Then $(g_1, g_2)$ has infinitely many images, up to the $_c E_2^m$ equivalence relation, under $(h_n)_{n<\omega}$.  
\end{lemma}
\begin{proof}
Suppose, for a contradiction, that $\{_c [(h_n(g_1), h_n(g_2))]_2^m \ | \ n<\omega\}$ is finite. Then, after refining $(h_n)_{n<\omega}$ (that, by abusing notation, is still denoted $(h_n)_{n<\omega}$), we may assume that it has cardinality $1$. In particular, there are $a,b\in \F$ such that, for every $n<\omega$,  
$$\F\models \exists y,w \bigl([h_n(g_2)^w,b]=1 \land [y,b]=1 \land a=h_n(g_1)^wy^m\bigr)$$ 
Therefore, there exists a closure, $Cl(G)$, of $G$ and elements $y,w\in Cl(G)$ such that $g_2^w$ and $y$ commute with $b$ and $a=g_1^wy^m$. Since the arguments are similar to the arguments of the previous results in this section we will be fast. 

First assume that $b$ does not commute with a peg. In this case, the centralizer of $b$ in $Cl(G)$ is contained in $\F$. Then $g_2^w$ and $y$ belong to $\F$ and it can be easily seen that also $w$ belongs to $\F$. In particular, the couple $(g_1, g_2)$ is $_c E_2^m$ equivalent to $(a,b)$ in $G$. 

We now assume that $b$ commutes with a peg. In this case, $w=w_1w_2$, where $w_2$ belongs to $G$, $w_1$ belongs to the free Abelian vertex group of the same peg and $g_2^{w_2}$ commutes with $b$. We may assume that $w_1$ belongs to $\Z\oplus A^m\setminus \Z\oplus \Z^m$, otherwise $(g_1^{w}, g_2^{w})$ belongs to $H$, and consequently $(g_1, g_2)$ depends on $H$, a contradiction. Therefore, $g_1^{w_2}=a^{w_1^{-1}}y^m$ and either $a$ also commutes with $b$ and $(g_1^{w_2}, g_2^{w_2})$ belongs to $H$ (hence $(g_1, g_2)$ depends on $H$, again a contradiction) or $g_1^{w_2}$ and $a^{w_1^{-1}}y^m$ have different normal forms.      
\end{proof}

We next prove a similar result for the cases of cosets and double cosets of powers (without conjugation). 

\begin{lemma}\label{finCosetsPouch}
Let $\mathcal{T}(G,\F)$ be a tower. Let $H$ be the Abelian pouch of $\mathcal{T}$. Suppose $g_1, g_2$ belong to $G$ and the couple $(g_1, g_2)$ does not depend on the Abelian pouch under the equivalence relation, $E_2^m$, of powers of cosets, for some natural number $m$. Let $(h_n)_{n<\omega}:G\rightarrow \F$ be a test sequence for $\mathcal{T}$. Then $(g_1, g_2)$ has infinitely many images, up to the $E_2^m$ equivalence relation, under $(h_n)_{n<\omega}$.  
\end{lemma}
\begin{proof}
Suppose, for a contradiction, that $\{[(h_n(g_1), h_n(g_2))]_2^m \ | \ n<\omega\}$ is finite. Then, after refining $(h_n)_{n<\omega}$ (that, by abusing notation, is still denoted $(h_n)_{n<\omega}$), we may assume that it has cardinality $1$. In particular, there are $a,b\in \F$ such that, for every $n<\omega$,  
$$\F\models \exists y \bigl([h_n(g_2),b]=1 \land [y,b]=1 \land a=h_n(g_1)y^m\bigr)$$ 
Therefore, there exists a closure, $Cl(G)$, of $G$ and an element $w\in Cl(G)$ such that $g_2$ and $w$ commute with $b$ and $a=g_1w^m$. Both $g_2$ and $w^m$ belong to $G$ (the latter because $w^m=g_1^{-1}a$). Hence, both $g_2$ and $w^m$ belong to $C_G(b)=C_{Cl(G)}(b)\cap G$. But $C_G(b)$, by Lemma \ref{CentrPouch}, is a subgroup of the Abelian pouch, $H$, of $\mathcal{T}$. Therefore, since $g_1=aw^{-m}$, we get that $g_1$ as well belongs to $H$, a contradiction.     
\end{proof}

\begin{lemma}\label{finDCosetsPouch}
Let $\mathcal{T}(G,\F)$ be a tower. Let $H$ be the Abelian pouch of $\mathcal{T}$. Suppose $g_1, g_2, g_3$ belong to $G$ and the triple $(g_1, g_2, g_3)$ does not depend on the Abelian pouch under the equivalence relation, $^kE_3^m$, of powers of double cosets, for some natural numbers $k, m$. Let $(h_n)_{n<\omega}:G\rightarrow \F$ be a test sequence for $\mathcal{T}$. Then $(g_1, g_2, g_3)$ has infinitely many images, up to the $^kE_3^m$ equivalence relation, under $(h_n)_{n<\omega}$.  
\end{lemma}
\begin{proof}
Suppose, for a contradiction, that $\{^k[(h_n(g_1), h_n(g_2), h_n(g_3))]_3^m \ | \ n<\omega\}$ is finite. Then, after refining $(h_n)_{n<\omega}$ (that, by abusing notation, is still denoted $(h_n)_{n<\omega}$), we may assume that it has cardinality $1$. In particular, there are $a,b,c\in \F$ such that, for every $n<\omega$,  
$$\F\models \exists y_1,y_2 \bigl([h_n(g_1),a]=1 \land [y_1,a]=1 \land [h_n(g_3),c]=1 \land [y_2,c]=1 \land b=y_1^kh_n(g_2)y_3^m\bigr)$$ 
Therefore, there exists a closure, $Cl(G)$, of $G$ and elements $w_1, w_2\in Cl(G)$ such that $g_1$ and $w_1$ commute with $a$, moreover $g_3, w_2$ commute with $c$ and finally $b=w_1^kg_2w_2^m$. By the same argument as in Lemma \ref{finCosetsPouch} we get that $g_1$ and $g_3$ belong to the Abelian pouch, $H$, of $\mathcal{T}$. If $w_1^k$ belongs to $H$, then the proof is identical to the proof of Lemma \ref{finCosetsPouch}. Thus, we may assume that $w_1^k$ belongs to the closure of the free Abelian group pegged on $a$, but not to the free Abelian group itself, i.e. $w_1^k\in\Z\oplus A^r\setminus\Z\oplus\Z^r$. Without loss of generality we assume that the free Abelian floor pegged on $a$ is the first floor of the tower, i.e. $\mathcal{G}(G^1, G^0)$, and the free Abelian floor pegged on $c$ (if it exists) is the second, i.e. $\mathcal{G}(G^2, G^1)$. We remark that it may also happen that $a=c$. If the height of $g_2$ is less or equal to $2$ (or less or equal to $1$ if $c$ is not a peg of an Abelian floor), then we have that $g_2$ belongs to the Abelian pouch, and we are done. Hence $i=height(g_2)>2$ (or respectively $i=height(g_2)>1$). We consider the reduced form of the element $w_1^kg_2w_2^m$ in the splitting $\mathcal{G}(Cl(G)^i, Cl(G)^{-1})$. Since $g_2\in Cl(G)^i\setminus Cl(G)^{i-1}$, and $w_1^k\in Cl(G)^{i-1}\setminus G^{i-1}$ the reduced form of $w_1^kg_2w_2^m$ has length at least $2$. Indeed, the elements that belong to $Cl(G^{i-1})$, in the reduced form of $g_2$, belong to $G^{i-1}$ and consequently they cannot fully cancel $w_1^k$. But then $w_1^kg_2w_2^m=b$ and $b$ has reduced form of length at most $1$, a contradiction.          
\end{proof}

Similarly we can prove. 

\begin{lemma}\label{finDTranslatesPouch}
Let $\mathcal{T}(G,\F)$ be a tower. Let $H$ be the Abelian pouch of $\mathcal{T}$. Suppose $g_1, g_2, g_3$ belong to $G$ and the triple $(g_1, g_2, g_3)$ does not depend on the Abelian pouch under the equivalence relation, $^kE_4^m$, of double translates, for some integers $k, m$. Let $(h_n)_{n<\omega}:G\rightarrow \F$ be a test sequence for $\mathcal{T}$. Then $(g_1, g_2, g_3)$ has infinitely many images, up to the $^kE_4^m$ equivalence relation, under $(h_n)_{n<\omega}$.  
\end{lemma}

As an easy corollary we get.

\begin{corollary}\label{CorrectNFCP}
Let $\mathcal{T}(G,\F)$ be a tower. Suppose $g_1, g_2, g_3$ belong to $G$ and $(h_n)_{n<\omega}:G\rightarrow \F$ is a test sequence for $\mathcal{T}$. If $(g_1, g_2, g_3)$ has finitely many images, up to the $^kE_4^m$ equivalence relation (for some fixed integers $k,m$), under $(h_n)_{n<\omega}$, then it has boundedly many images under the family of all morphisms that factor through $\mathcal{T}$. As a matter of fact it has at most $k\cdot m$ many images. 
\end{corollary}
\begin{proof}
By Lemma \ref{finDTranslatesPouch} we get that the triple $(g_1, g_2, g_3)$ depends on the Abelian pouch, $H$, of $\mathcal{T}$. Hence, there exist $a,b,c$ in $H$ such that $[g_1,a]=1$, $[g_3,c]=1$, $[g_1,c]=1$ and there is $\gamma$ with $[\gamma,a]=1$ and $g_2=\gamma^kb\gamma^m$. The Abelian pouch, $H$, has the structure of a tower, which is a subtower of $\mathcal{T}$. Thus, we may re-apply the argument of the previous lemmas on the triple $(a,b,c)$ and the restriction of $(h_n)_{n<\omega}$ on $H$. Therefore, we get that there exists a closure, $Cl(H)$, of $H$ and $w\in Cl(H)$ such that $a, c$ and $w$ commute with some element $\alpha\in\F$ and $b=w^k\beta w^m$ for some element $\beta\in\F$. If $\alpha$ (which we may assume to be rootless) does not generate a peg of an Abelian floor, then $w\in H$. Hence, any (restricted) morphism $h:G\rightarrow\F$, that factors through the tower $\mathcal{T}$, sends $(a,b,c)$ (and consequently $(g_1, g_2, g_3)$) to a fixed $^kE_4^m$-class, i.e. the class $^k[(\alpha, \beta, \alpha)]_4^m$. On the other hand, if $\alpha$ generates a peg, and $w$ belongs to the closure of the Abelian floor, $\Z\oplus A^r$ that is pegged on $\langle\alpha\rangle$, but not to the Abelian floor, $\Z\oplus\Z^r$, itself, then, arguing by normal forms, we must have that $w^k$ and $w^m$ belong to $\Z\oplus\Z^r$ (except when $\beta$ commutes with $\alpha$, where, in this case, we must have that $w^{k+m}$ belongs to $\Z\oplus\Z^r$). In this case, any (restricted) morphism, $h:G\rightarrow\F$, that factors through the tower $\mathcal{T}$, sends $b$ to an element that belongs to $C_{\F}(\alpha)\beta C_{\F}(\alpha)$. Therefore, it sends the triple $(a,b,c)$ to at most $k\cdot m$ many $^kE_4^m$-classes, i.e. to one of the classes in the set $\{^k[(\alpha, \alpha^i\beta\alpha^j,\alpha)]_4^m \ | \ 0\leq i<k, \ 0\leq j<m\}$.     





\end{proof}

It is not hard to prove the analogous result for the family of generalized equivalence relations. 

\begin{proposition}\label{CorrectNFCP2}
Let $\mathcal{T}(G,\F)$ be a tower. Let $H$ be the Abelian pouch of $\mathcal{T}$. Suppose $g_1, g_2, \ldots, g_n$ belong to $G$ and the $n$-tuple $(g_1, g_2, \ldots, g_n)$ does not depend on the Abelian pouch under some generalized equivalence relation $E_{\{g\}}$. Let $(h_n)_{n<\omega}:G\rightarrow \F$ be a test sequence for $\mathcal{T}$. If $(g_1, g_2, \ldots, g_n)$ has finitely many images, up to the $E_{\{g\}}$ equivalence relation, under $(h_n)_{n<\omega}$, then it has boundedly many images under the family of all morphisms that factor through $\mathcal{T}$.
\end{proposition}

The above corollary and proposition are not directly needed for the purposes of the current paper. But, in the light of the new equivalence relations, they are required for correcting the results in \cite{MR3846335}.

\subsection{Diophantine Envelopes}

We will use towers and test sequences in order to understand definable sets in nonabelian free groups. Although, any formula (in the language of groups) $\phi(\bar{x})$ is equivalent to 
a boolean combination of $\forall\exists$-formulas, modulo the first-order theory of nonabelian free groups, i.e. 

$$T_{fg}\models \phi(\bar{x})\leftrightarrow \bigwedge\limits_{i=1}^k\bigl(\psi_i^{1}(\bar{x})\lor\ldots \lor\psi_i^{m_i}(\bar{x})\bigvee \lnot \theta_i^1(\bar{x})\lor\ldots\lor\lnot \theta_i^{n_i}(\bar x)\bigr)$$
where $\psi_i^j$ and $\theta_i^l$, are $\forall\exists$ formulas for each $i\leq k$, all $j\leq m_i$ and all $l\leq n_i$,
it is hard to understand the solution sets of the basic blocks (i.e. $\forall\exists$-formulas) in nonabelian free groups. A remedy for that, 
since in many cases we do not need to know the exact cut between $\phi(\F)$ and $\lnot\phi(\F)$, is to replace a formula by a finite set of towers. Each tower $\mathcal{T}(G, \F)$ 
will be endowed with a generating set for the corresponding group $G$, which is split in three tuples, a tuple of elements identical with a basis for $\F:=\langle\bar{a}\rangle$, a tuple of elements $\bar{x}$ identified 
with the variables of the formula, and a tuple $\bar{u}$ that together with $\bar{a}$ and $\bar{x}$ generate $G$. In this fashion, $G$ admits a presentation 
$\langle \bar{u}, \bar{x}, \bar{a} \ | \ \Sigma_{G}(\bar{u}, \bar{x}, \bar{a})\rangle$ and at the same time defines a Diophantine set $\exists \bar{u} \bigl(\Sigma_G(\bar{u}, \bar{x}, \bar{a})=1\bigr)$, which is 
a subset of $\F^{|\bar{x}|}$.

The principal idea is that for a nonempty first-order formula $\phi(\bar{x})$, i.e. a formula for which $\F\models\exists \bar{x}\phi(\bar{x})$, there are  
always a tower $\mathcal{T}(G,\F)$, a distinguished tuple (denoted by the same letters $\bar{x}$) in $G$, and a test sequence $(h_n)_{n<\omega}:G\rightarrow\F$
such that $\F\models\phi(h_n(\bar{x}))$. A consequence of Sela's \cite[Theorem 1.3]{SelaIm} is that for every nonempty first-order formula,  
there exist finitely many such towers, the {\em Diophantine envelope} of the formula, that the union of 
the corresponding Diophantine sets contains the solution set of the formula. To be more precise we record the ungraded version of Sela's result (one may obtain it by applying Sela's Theorem 1.3 to the formula $\psi(\bar x, \bar y, \bar a):=\phi(\bar x, \bar a)\land \bar y=\bar a)$, i.e. by adding $\bar y$ as dummy variables and make them equal to the coefficients):

\begin{fact}[Diophantine envelope]\label{Envelope}
Let $\phi(\bar{x},\bar{a})$ be a nonempty first-order formula over $\F$. Then there exist finitely many towers, 
$\{(\mathcal{T}_i(G_i,\F))_{i\leq k}\}$, with $G_i:=\langle \bar{u}_i,\bar{x},\bar{a}\ | \ \Sigma_i(\bar{u}_i,\bar{x},\bar{a})\rangle$, that we call the Diophantine Envelope of $\phi(\bar{x},\bar{a})$, with the following properties:
\begin{itemize}
\item[(i)] for each $i\leq k$, there exists a test sequence, $(h_n)_{n<\omega}:G_i\rightarrow\F$, for $\mathcal{T}_i$  
 such that $\F\models \phi(h_n(\bar{x}),\bar{a})$.
\item[(ii)] If $\F\models \phi(\bar b, \bar a)$, then there exist $i\leq k$ and a morphism $h:G_i\rightarrow \F$, that factors through $\mathcal{T}_i$ that sends $\bar x$ to $\bar b$.   
\end{itemize}

\end{fact}

The above family of towers is called the Diophantine envelope of $\phi(\bar x, \bar a)$, because $\F\models\phi(\bar{x},\bar{a})\rightarrow \exists \bar{u}_1,\ldots,\bar{u}_k(\bigvee_{i=1}^{k}\Sigma_i(\bar{u}_i,\bar{x},\bar{a})=1)$, i.e. the corresponding union of Diophantine sets (that can be proved to be equal to a Diophantine set when working over parameters) contains $\phi(\bar x, \bar a)$.

Some comments on the towers that take part in the Diophantine envelope of a formula are in order.

\begin{remark}
We note that by the construction in Remark \ref{SimplevsExtended} we may transform  any "extended" tower, $\mathcal{T}(G,\F)$ in the Diophantine envelope of $\phi(\bar x, \bar a)$ to a simple tower, $\hat{\mathcal{T}}(\hat{G},\F)$, where $\hat{G}=G*\F_k$, for some $k<\omega$. The floor vertex group of each floor of $\hat{\mathcal{T}}$ is identical to the corresponding floor vertex group of $\mathcal{T}$ and the new retractions kill $\F_k$, hence any test sequence of $\mathcal{T}$ extends to a test sequence of $\hat{\mathcal{T}}$ and does not change the values of $\bar x$.

Finally, the other assumptions, e.g. moving pegs and thus changing the order of the floors of a tower, had been made for convenience and are not essential for the results obtained.

\end{remark}

\begin{remark}\label{NonDegenenaration}
A morphism $h$ that factors through a tower, $\mathcal{T}$, has the property that if $h$ kills a peg of an Abelian floor, then it kills the whole free Abelian vertex group. 

Indeed, if $G:=G^m>G^{m-1}>\ldots>G^1>G^0:=\F*\F_n$ is the sequence of the groups of the floors of $\mathcal{T}$, a morphism that factors through $\mathcal{T}$ has the form $$h_0\circ r_1\circ \alpha_1\circ r_2\circ\alpha_2\circ\ldots\circ r_m\circ\alpha_m$$

for some $\alpha_i\in Mod(G^i)$ and some (restricted) morphism $h_0:G^0\rightarrow\F$. 

In particular, if $(\mathcal{G}(G^{i+1}, G^i), r_{i+1})$ is an Abelian floor with a free Abelian vertex group $\Z\oplus\Z^k:=\langle z, z_1, \ldots, z_k\rangle$, and $h$ kills the peg in $G^i$, then, since $r_{i+1}$ sends $z_i$ to $z$, for $i\leq k$, and $\alpha_{i+1}$ fixes $z$, we get that $h$ must kill all $z, z_1, \ldots, z_k$. 
\end{remark}

\begin{remark}\label{MultipletEnvelope}
It is easy to construct the envelope of the conjunction $\phi(x_1)\land \phi(x_2)\land\ldots \phi(x_n)$ once we have an envelope for $\phi(x)$. We can take for each tower in the Diophantine envelope of $\phi(x)$ its $n$-multiplet. In addition, test sequences for these multiplets may give different priority to the copies of the original tower, e.g. the growth rate given to the $i$-th copy may dominate the growth rate of the $j$-th copy, under the test sequence, for an arbitrary choice of $i$ and $j$.   
\end{remark}

\begin{example}\label{ExDio}

\begin{enumerate}
\item We consider the formula $\phi(x):=\exists x_1,x_2\bigl(x=x_1^2x_2^3 \ \land \ [x_1,x_2]\neq 1\bigr)$. Its envelope is a single tower  consisting only of a  ground floor $G(x,x_1, x_2):=\langle x,x_1,x_2 \ | \ x=x_1^2x_2^3 \rangle *\F$. Indeed, any test sequence of this tower will eventually give values to $x_1, x_2$ that do not commute and, moreover, any solution of $\phi(x)$ in $\F$ extends to a morphism that factors through this tower. 

\item More abstractly, let $Sld:=Sld(\bar y, \bar x, \bar a)$ be a solid limit group, with $\bar x, \bar a$ the grading parameters (for the definition of solid limit groups see \cite[Definition 10.2]{MR1863735}). Let $Res$ be a graded resolution of a graded limit group that terminates in $Sld$. By \cite[Proposition 3.8(i)]{MR2166355}, the set $$\{ \ \bar b \in \F^{|\bar x|} \ \ |\  \textrm{there exists a strictly solid morphism}\ \ h:Sld\rightarrow\F \ \ \textrm{that extends}\ \ h(\bar x)= \bar b\}$$ is definable by an $\exists\forall$ first-order formula, say $\phi(\bar x, \bar a)$ (see \cite[Definition 1.5]{MR2166355} for the notion of a strictly solid morphism). 

We claim that the Diophantine envelope of $\phi(\bar x, \bar a)$ is a sub-family of the (finite) family of towers obtained as completions of the well-structured resolutions of the (ungraded) strict Makanin-Razborov diagram of $Sld$ (for the definition of well-structured resolutions and completions see \cite[Definition 1.11]{MR1972179} and \cite[Definition 1.12]{MR1972179} respectively). Note that by \cite[Lemma 1.13]{MR1972179} $Sld$ (or a quotient of it) embeds in any of its completed well-structured resolution and by abusing notation we will keep denoting its generating set in the towers that correspond to the completions by $\bar y, \bar x, \bar a$.  A tower $\mathcal{T}(G,\F)$, with  $G:=G(\bar u, \bar y, \bar x, \bar a)$, is in the sub-family of the Diophantine envelope of $\phi(\bar x, \bar a)$, if for at least one morphism $h:G\rightarrow\F$, its restriction to (the image of) $Sld$, is a strictly solid morphism. Indeed, since a strictly solid morphism exists for the particular solid limit group $Sld$ (otherwise the set is empty) and any morphism in $Hom_\F(Sld,\F)$ extends to a morphism from at least one completion, the chosen sub-family of towers is non-empty and defines a (union of) Diophantine sets, i.e. the projections of $Hom_\F(G,\F)$ to $\bar x$, that contains $\phi(\F, \bar a)$. To prove that there exists a test sequence for each tower in the chosen sub-family whose values when restricted to $\bar x$ realize $\phi(\bar x, \bar a)$ in $\F$ we argue as follows. By the definition of strictly solid morphisms and \cite[Theorem 11.2 \& Definition 11.4]{MR1863735}, there exists a finite disjunction of Diophantine sets that expresses that a morphism is not strictly solid. By a variant of Fact \ref{MerzTowers}, applied to all test sequences of $\mathcal{T}$ (instead of to all tuples satisfying some fixed defining relations of $G$), there exists a finite set of closures of $\mathcal{T}$ such that any test sequence of $\mathcal{T}$ that extends to satisfy one of those systems, has a subsequence that extends to one of these closures. This set of closures cannot form a covering closure since then every morphism from $G$ to $\F$ would extend to satisfy at least one system of equations and hence it would be non strictly solid. Therefore, by choosing the values of the bases of the free Abelian vertex groups in a way that no subsequence extends to any of the finitely many closures we may find a test sequence such that for any $n$ it does not satisfy any of the systems of equations and thus the values it gives to $\bar x$ belong to the definable set.       
\end{enumerate}
\end{example}

We next expand on the second example to show that if a test sequence for an arbitrary tower gives "correct" values (i.e. values that realize the first-order formula in $\F$) to a distinguished tuple of the group corresponding to the tower, then we may construct test sequences, by choosing the values of the bases of free Abelian vertex groups, so that eventually they will give as well "correct" values to the distinguished tuple. 

\begin{example}[Example 2 continued]\label{ExDio2}
Consider the definable set in Example \ref{ExDio} (2) and an arbitrary tower $\mathcal{T}(G,\F)$, where $G:=G(\bar u, \bar x, \bar a)$. Suppose for a test sequence, $(h_n)_{n<\omega}:G\rightarrow\F$, of $\mathcal{T}$, we have $\F\models\phi(h_n(\bar x),\bar a)$. By applying (a variant of) Fact \ref{MerzTowers} to all test sequences of $\mathcal{T}$ that (each of their elements) extend to  morphisms of $Sld(\bar y, \bar x, \bar a)$ we get a finite set of closures, $\mathcal{T}_1(G_1, \F), \ldots, \mathcal{T}_k(G_k,\F)$, of $\mathcal{T}$, and for each $i\leq k$ a formal solution $\bar y_i(\bar v, \bar u, \bar x, \bar a)$ with the property that the words in $\Sigma_{Sld}(y_i(\bar v, \bar u, \bar x, \bar a),\bar x,\bar a)$, where $\Sigma_{Sld}(\bar y, \bar x, \bar a)$ are the defining relations of $Sld$, are trivial in $G_i$. Therefore, any morphism $h:G\rightarrow\F$ that extends to one of the closures it also extends to a morphism from $Sld$ to $\F$. 
In addition, whenever each of the morphisms of a test sequence of $\mathcal{T}$ extends to a morphism from $Sld$ to $\F$, it, then, has a subsequence that extends to one of the (finitely many) closures. In particular, $(h_n)_{n<\omega}$ has a subsequence that extends through one of the closures. 
Moreover, since for a (graded) modular automorphism $\alpha$ of $Sld(\bar y, \bar x, \bar a)$ and a morphism $h:Sld\rightarrow\F$, $h\circ\alpha$ is strictly solid if and only if $h$ is strictly solid, we get that the extension of each $h_n$, in the subsequence, to $G_i$ extends to a strictly solid morphism for the particular values it gives to $y_i(\bar v, \bar u, \bar x, \bar a), \bar x, \bar a$. 

For each closure $\mathcal{T}_i$ we construct finitely many closures 
by applying (a variant of) Fact \ref{MerzTowers} to all test sequences of $\mathcal{T}_i$ that extend to satisfy one of the Diophantine first-order formulas expressing that a morphism is not strictly solid. The second level of closures has the property that if a morphism $h:G_i\rightarrow\F$, extends to one of them, then it is not strictly solid when restricted to $\bar y_i(\bar v,\bar u, \bar x, \bar a), \bar x, \bar a$, where $\bar y_i(\bar v,\bar u, \bar x, \bar a)$ is the formal solution obtained in the first level of closures. In addition, whenever each of the morphisms of a test sequence of $\mathcal{T}_i$ is not strictly solid (for the values it gives to $\bar y_i(\bar v,\bar u, \bar x, \bar a), \bar x, \bar a$), it, then, has a subsequence that extends to one of the closures.

These two levels of closures of $\mathcal{T}$ define, for each Abelian floor, a family of cosets as follows: for the Abelian vertex group of each floor of $\mathcal{T}$ we consider the (finite index) subgroup that is defined as the intersection of all subgroups, that each of which corresponds to some closure (we also include the second level of closures). The different cosets with respect to this finite index subgroup at each Abelian floor give us finitely many combinations. Each such combination of cosets partition test sequences of $\mathcal{T}$ into classes. If for a certain combination the test sequence factors through some closures in the first level and for each such closure it factors through it does not factor through none of the corresponding closures in the second level, then it will eventually give values to $\bar x$ that belong to $\phi(\bar x, \bar a)$. Indeed, each element of such test sequence extends to satisfy $Sld(\bar y, \bar x, \bar a)$ by the values it gives to  $\bar y_i(\bar z,\bar u, \bar x, \bar a), \bar x, \bar a$, if these values were not strictly solid then a subsequence of said test sequence would factor through a closure in the second level, but then its values would belong to the class that correspond to such closure, a contradiction. 

Finally, the fact that $(h_n)_{n<\omega}$ is such that some subsequences of it extends to a closure from the first level, but none of the elements of the subsequence extends to a closure in the second level, implies that there exists a combination of cosets with the desired property.  
\end{example}

The ideas in the above example can be generalized to any first-order formula, but, of course, a proof is much harder in the general case. Moreover, the general case requires three levels of closures. We now state the precise result.

\begin{fact}[Three levels of towers]\label{FactSela}
Let $\phi(\bar{x},\bar{a})$ be a first order formula over $\F$. Let $\mathcal{T}(G,\F)$, where $G:=G(\bar{u},\bar{x},\bar{a})$, be a tower over $\F$. 
Suppose there exists a test sequence, $(h_n)_{n<\omega}:G\rightarrow\F$ for $\mathcal{T}(G,\F)$, such that $\F\models\phi(h_n(\bar{x}),\bar{a})$. 

Then there exist: 
\begin{itemize}
 \item (Level 1) finitely many closures (at least one), $\mathcal{T}_1:=Cl_1(\mathcal{T}), \ldots, \mathcal{T}_k:=Cl_k(\mathcal{T})$, of $\mathcal{T}$;
 \item (Level 2) for each $i\leq k$, finitely many closures (possibly none), $\mathcal{R}^i_1:=Cl_1(\mathcal{T}_i),\ldots, \mathcal{R}^i_{m_i}:=Cl_{m_i}(\mathcal{T}_i)$. 
 \item (Level 3) for each $j\leq m_i$, for $i\leq k$, finitely many closures (possibly none), 
 $Cl_1(\mathcal{R}_j^i),\ldots,$ $Cl_{q_{i,j}}(\mathcal{R}_j^i)$.
\end{itemize}
   
The three levels of closures define, for each Abelian floor of $\mathcal{T}$, a finite family of cosets in a natural way. Considering combinations of these finite families of cosets, choosing one coset for each floor, we may partition test sequences according to the combination of cosets they satisfy. A test sequence $(g_n)_{n<\omega}$ in some class of the above partition satisfies eventually $\phi(\bar{x},\bar{a})$, i.e. $\F\models\phi(g_n(\bar{x}),\bar{a})$, for all $n>n_0$, if one of the following holds:
\begin{enumerate}
\item it extends to one of the closures in the first level (possibly more than one) and for each such closure that it extends to, it does not extend to any of its closures in the second level.
\item it extends to one of the closures in the first level (possibly more than one), and for some of the closures it extends to, it also extends to at least one of its closures in the second level, and for each such closure in the second level, it again extends to some closure in the third level.   
\end{enumerate}
Finally, such a test sequence exists. 
\end{fact}

The above fact is a consequence of the (relative) quantifier elimination \cite{MR2238944}, \cite{MR2249582} and the form the basic first-order definable sets have \cite[Propositions 3.7, 3.8 \& 3.9]{MR2166355}. The final claim of the existence of a test sequence in the above fact straightforwardly gives. 

\begin{corollary}\label{FixCosetTest}
Let $\phi(\bar{x},\bar{a})$ be a first order formula over $\F$. Let $\mathcal{T}(G,\F)$, where $G:=G(\bar{u},\bar{x},\bar{a})$, be a tower over $\F$. 
Suppose there exists a test sequence, $(h_n)_{n<\omega}:G\rightarrow\F$ for $\mathcal{T}(G,\F)$, such that $\F\models\phi(h_n(\bar{x}),\bar{a})$.

Then, there exists a closure of $\mathcal{T}(G,\F)$ such that for any test sequence $(g_n)_{n<\omega}:G\rightarrow\F$ that extends to a test sequence 
of the closure we have:
$$\F\models\phi(g_n(\bar{x}),\bar{a}),\ \ \ \textrm{for all but finitely many}\ \ n.$$
In particular, if $\mathcal{T}(G,\F)$ is hyperbolic any of its test sequences has the above property.  
\end{corollary}

\subsection{Implicit function theorems}

Implicit function theorems (see \cite{MR2154989}) or extended Merzlyakov theorems (see \cite{MR1972179}) form the basis for the validation procedure that a $\forall\exists$ sentence is true in a nonabelian free group, but also for the (relative) quantifier elimination. An example of such a type of theorem was given in Fact \ref{TestMerzTowers}.

%
%
%
%
Implicit function theorems can be generalized allowing at the place of a system of equations $\Sigma(\bar{x},\bar{y},\bar a)=1$ a first-order formula $\phi(\bar{x},\bar{y})$ as long as $\F\models \forall \bar{x} \exists^{<\infty} \bar{y} \phi(\bar{x}, \bar{y})$ (see \cite[Theorem 6.34]{MR4030181}).  

Although Theorem 6.34 in \cite{MR4030181} has only been proved for real elements, an analogue for imaginary sorts also holds. We will sketch the idea of the proof and how one can modify it to prove the following statement. 

\begin{theorem}\label{Implicit}
Let $\phi(x_1, \ldots, x_m, y_1,\ldots, y_k, \bar a)$ be a first-order formula over $\F$ in the multi-sorted language $\mathcal{L}^{eq}$ that corresponds to the theory of the free group. Let $(x_1, \ldots, x_m)$ be a tuple of variables, where $x_j$ is a variable of sort $S_{E_{i_j}}$, and $\bar{y}=(y_1, y_2, \ldots, y_k)$ be a tuple of variables in the basic sorts. Let $\mathcal{T}(G, \F)$ be a tower with $G:=G(\bar u, \bar z, \bar a)$, such that $\bar z=(\bar z_1, \bar z_2, \ldots, \bar z_m)$, where the length of $\bar z_j$ is equal to the length of tuples in the sort $E_{i_j}$, for $j\leq m$. 

Let $\F^{eq}\models \forall \bar{x}\exists^{<\infty}\bar{y}\phi(\bar{x},\bar{y},\bar a)$ and 
suppose there exists a test sequence $(h_n)_{n<\omega}:G\rightarrow \F$, for the tower $\mathcal{T}:=\mathcal{T}(G,\F)$, and a sequence of tuples in the basic sorts $(\bar{c}_n)_{n<\omega}$ 
such that  $$\F^{eq}\models\phi([h_n(\bar z_1)]_{E_{i_1}}, [h_n(\bar z_2)]_{E_{i_2}},\ldots,[h_n(\bar z_m)]_{E_{i_m}} ,\bar{c}_n,\bar a) \ \ \textrm{ for all}\ \ n<\omega.$$

Then there exists a closure, $Cl(\mathcal{T})(Cl(G), \F)$, of $\mathcal{T}$, with $Cl(G):=Cl(G)(\bar v, \bar u, \bar z, \bar a)$ and a tuple of elements $\bar{w}=(\bar{w}_1, \bar{w}_2, \ldots, \bar{w}_k)$, where the length of $\bar w_i$ is equal to the length of tuples in the sort $S_{E_{j_i}}$ that corresponds to the variable $y_i$, for $i\leq k$, such that:
\begin{itemize}
\item a subsequence of $(h_n)_{n<\omega}$, extends to a test sequence $(H_n)_{n<\omega}:Cl(G)\rightarrow\F$ for $Cl(\mathcal{T})$. 
\item $\F^{eq}\models \phi([H_n(\bar z_1)]_{E_{i_1}}, [H_n(\bar z_2)]_{E_{i_2}},\ldots,[H_n(\bar z_m)]_{E_{i_m}},[H_n(\bar{w}_1)]_{E_{j_1}},\ldots, [H_n(\bar{w}_k)]_{E_{j_k}})$ for all $n<\omega$.
\end{itemize}
\end{theorem}
\begin{proof}[Sketch]
The idea of the proof when $\phi(\bar x, \bar y, \bar a)$ is a first-order formula in the language $\mathcal{L}$ is as follows. 
\begin{itemize}
    \item[Step I] We apply Theorem \cite[Theorem 1.3]{SelaIm} to $\phi(\bar x, \bar y,\bar a)$, considering $\bar x, \bar a$ to be the grading parameters, in order to obtain a finite family of graded towers based over solid or rigid limit groups with respect to the fixed grading. 
    \item[Step II] We prove (\cite[Theorem 6.25]{MR4030181}) that for every graded tower in the family obtained in step I, the tuple $\bar y$ belongs to a subgroup $\Gamma$ of the solid group in the base of the tower such that for any two morphisms, $h_1, h_2:Sld\rightarrow\F$, in the same strictly solid family the restrictions $h_1\upharpoonright_\Gamma, h_2\upharpoonright_\Gamma $ are identical.    
    \item[Step III] For any test sequence $(h_n)_{n<\omega}:G\rightarrow\F$ as in the hypothesis of the statement, there exists, by a pigeonhole argument, a subsequence that extends to satisfy the defining relations of one of the solid limit groups that lies in the base of a (graded) tower (of the finitely many obtained in Step I). After possibly changing the values of the extension, but within the same striclty solid family, this sequence of morphisms gives values to $\bar x$ and $\bar y$ that satisfy $\phi(\bar x, \bar y, \bar a)$ in $\F$. Note that in the case of a rigid group there is no change in the values of the extension. 
    \item[Step IV] Since $\bar y$ belongs to $\Gamma$, its values by the extension of (the subsequence of) $(h_n)_{n<\omega}$ do not change and in particular by applying (extended) Merzlyakov's theorem to $\mathcal{T}$ and the system of equations given by the defining relations of the solid group, we obtain a tuple of formal solutions in some closure of $\mathcal{T}$ as in the conclusion of the statement (\cite[Theorem 6.32]{MR4030181}).  
\end{itemize}
In order to generalize the above proof to formulas in the language $\mathcal{L}^{eq}$ we do the following. We first replace $\phi(\bar x, \bar y, \bar a)$ by the real formula $\psi(\bar X, \bar Y, \bar a)$ corresponding to $\phi(\bar x, \bar y, \bar a)$. We, then, apply Theorem \cite[Theorem 1.3]{SelaIm} to the formula $\psi(\bar X, \bar Y, \bar a)$ considering $\bar X, \bar a$ as the grading parameters. The main point is to prove step II, i.e. that for every graded tower in the (graded) Diophantine envelope of $\psi(\bar X, \bar Y, \bar a)$, the tuple $\bar Y$ depends on a subgroup $\Gamma$ of the solid group in the base of the tower that is invariant under morphisms in the same strictly solid family. This is indeed true and claimed in the proofs of \cite[Theorem 2.1]{SelaIm} for conjugacy classes (and with a small modification it works for commuting conjugation classes), in \cite[Theorem 2.2]{SelaIm} for cosets and \cite[Theorem 2.3]{SelaIm} for double cosets (again with a small modification the proofs work for cosets with conjugation and double cosets with conjugation). The claim in these proofs works under the assumption that $\bar Y$ is unique, but what is really used is boundedness. The rest of the proof is identical to the proof in the real case. 
\end{proof}

Since a closure of a hyperbolic tower coincides with the tower itself we get the following corollary.

\begin{corollary}\label{HyperImplicit}
Suppose the hypotheses of Theorem \ref{Implicit} hold. Suppose, moreover, that the tower $\mathcal{T}:=\mathcal{T}(G,\F)$ is hyperbolic.

Then there exists a tuple of elements $\bar{w}=(\bar{w}_1, \bar{w}_2, \ldots, \bar{w}_k)$ in $G$ such that for any test sequence for $\mathcal{T}$, $(g_n)_{n<\omega}:G\rightarrow \F$, for all but finitely many $n<\omega$, the following holds : 

$$\F^{eq}\models\phi([g_n(\bar z_1)]_{E_{i_1}}, [g_n(\bar z_2)]_{E_{i_2}},\ldots,[g_n(\bar z_m)]_{E_{i_m}} ,[g_n(\bar{w}_1)]_{E_{j_1}},\ldots, [g_n(\bar{w}_k)]_{E_{j_k}},\bar a)$$ 

\end{corollary}


\section{Main Proof}\label{MainProof}
This section contains the proof of the main result, namely no infinite field is interpretable in the first-order theory of the free group. We tackle some special cases before moving to the full proof. In particular, 
we consider the Abelian case and the hyperbolic case.   

In the Abelian case we prove that when an infinite interpretable set is contained in the Abelian pouch of an envelope, it cannot be the domain of a field. This case is actually an essential part of the main proof, as it deals with the situation where the noncommutativity argument (see Proposition \ref{NonCommut}) cannot be applied. 

The hyperbolic case is when a Diophantine envelope of an interpretable set contains a hyperbolic tower and a test sequence of this tower gives infinitely many tuples in the set. This case is not essential, but added for clarity. The proof in this case uses the same argument as the proof of the general case, but it is free of certain technicalities. 

\subsection{Abelian case}
We first tackle the special case where the interpretable set is contained in the images of the Abelian pouches of the  Diophantine Envelope.

In order to conclude that such a set cannot be the domain of an infinite field, we need to know that centralizers of non-trivial elements in any model of the theory of nonabelian free groups are one-based. 

\subsubsection{Centralizers of nontrivial elements}

We recall that centralizers of nontrivial elements in nonabelian free groups are pure groups \cite[Corollary 6.28]{MR4030181}. In particular:

\begin{fact}\label{CentraBased}
Let $b$ be a nontrivial element of $\F$. Then, its centralizer, $C_{\F}(b)$, is one-based.
\end{fact}

We will prove that the same holds for any group $G$ elementarily equivalent to a free group. Recall that in a stable group "genericity" is first-order expressible (see \cite[Lemma 5.4]{MR1827833}).

\begin{fact}
Let $G$ be a stable group and $\phi(x,\bar{y})$ a first-order formula over $\emptyset$. Then there exists a natural number $n$, such that for any $\bar{a}$ in $G$, either $n$ translates of $\phi(G,\bar{a})$ cover $G$, or no finite number of translates of $\phi(G,\bar{a})$ cover $G$, where by $\phi(G,\bar{a})$ we denote the solution set of $\phi(x,\bar{a})$ in $G$.
\end{fact}

\begin{theorem}\label{MonsterOneBased}
Let $G$ be elementarily equivalent to a nonabelian free group and $b$ a nontrivial element of $G$. Then its centralizer, $C_{G}(b)$, is one-based.
\end{theorem}
\begin{proof}
By \cite[Chapter 2, Proposition 5.8]{MR1429864} it will be enough to prove that $C_G(b)$ is superstable of finite $U$-rank and all minimal types (i.e. stationary types of $U$-rank $1$) are locally modular. By Buechler's dichotomy lemma (see \cite[Chapter 2, Proposition 3.2]{MR1429864} or \cite[Theorem 1]{MR820131}) it will be enough to prove that $C_{G}(b)$ has $U$-rank $1$. In particular, since the first-order theory of the free group does not have the finite cover property (nfcp), it is enough to prove that any infinite definable subset $\phi(G,\bar{c}):=X$ of $C_{G}(b)$ in $G$ is generic, i.e. finitely many translates, $a.X$, of $X$ cover $C_{G}(b)$. If, for a contradiction, this is not true, and $\phi(G,\bar{c})$ witnesses it, we can express it with a first-order sentence as follows:
$$\exists x,\bar y \bigl(x\neq 1 \ \land \ ``\phi(z,\bar y)\land [x,z]=1 \ \textrm{is infinite}" \ \land \ ``\phi(z,\bar{y}) \ \textrm{is not generic in} \ [x,z]=1"\bigr)$$ 
Infinitude and genericity can be expressed by nfcp (actually genericity by stability alone). Hence, the above sentence is true in a nonabelian free group. This is a contradiction, since any centralizer of a nontrivial element in a nonabelian free group is one-based, hence it has $U$-rank $1$. 
\end{proof}

\begin{corollary}\label{InvariantOneBased}
Let $\mathbb{M}$ be the monster model of the theory of the free group. Then the family of centralizers of nontrivial elements $\Sigma:=\{C_{\mathbb{M}}(b) \ | \ b\in\mathbb{M}\setminus\{1\}\}$ is a $\emptyset$-invariant family of one-based definable sets.
\end{corollary}

\subsubsection{Internality to the family of centralizers}

%

\begin{definition}\label{Dependance}
Let $G$ be a group, $E$ an equivalence relation in $G^m$ and $H$ a subgroup of $G$. We say that the equivalence class of a tuple $\bar g\in G^m$ depends on $H$, if $\bar g\sim_E\bar h$, for some tuple $\bar h\in H^m$. 

Similarly, if $(E_i)_{i\leq k}$ are equivalent relations (possibly repeating), in $G^{m_i}$, for $i\leq k$, and 
$\bar{g}:=\bar{g}_1, \ldots,\bar{g}_k$ is a tuple in $G^{m_1}\times\ldots\times G^{m_k}$, then we say that $\bar g$ depends on $H$ if $\bar{g}_i\sim_{E_i}\bar{h}_i$ for a tuple $\bar{h}:=\bar{h}_1,\ldots,\bar{h}_k$ in $H^{m_1}\times\ldots\times H^{m_k}$. 
\end{definition}

Before proving the main result of this subsection we will be needing a technical lemma.

\begin{lemma}\label{AbelianFloorDep}
Let $\mathcal{T}(G,\F)$ be a tower and $b$ an element that commutes with a nontrivial element in the Abelian pouch, $H$, of $\mathcal{T}$. Then either $b$ belongs to the Abelian pouch or it belongs to the free Abelian group of an  Abelian floor that is pegged on the Abelian pouch.
\end{lemma}
\begin{proof}
Let $G:=G^m>G^{m-1}>\ldots>G^1>G^0:=\F*\F_n$ be the sequence of the groups of the floors of $\mathcal{T}$. By the definition of the Abelian pouch it is the subgroup which is constructed by adding the first $i$ (Abelian) floors to $\F$ (i.e. disregarding $\F_n$), for some $i\leq m$. For convenience of notation, we will abuse language and we 
define by $\tilde{\mathcal{T}}(G, \F)$ the "tower" whose first floor consists of the $i$ first floors of $\mathcal{T}$, i.e. $G:=\tilde{G}^{k}>\ldots>\tilde{G}^1>G_0:=\F*\F_n$, where $k=m-i$ and $\tilde{G}_1$ is $G_i$. This is not formally a tower, but it will be convenient to use induction on the height of $b$ in $\tilde{\mathcal{T}}$ as well as the nested analysis of $b$ in $\tilde{\mathcal{T}}$ instead of the original tower.  

We will prove the lemma by induction on the height of $b$ in $\tilde{\mathcal{T}}$. For $b$ of height $1$, we 
see the normal form of $b$ with respect to the free product $H*\F_n$. If $b$ does not belong to $H$, then it does not commute with any nontrivial element of $H$ and we are done. 

Suppose we have proved the result for any $b$ of any height less than $r$. We will prove it for elements of height $r+1$. We split the proof in cases, according to whether the $r+1$ floor is surface or Abelian.  
\begin{itemize}
    \item Suppose $\tilde{\mathcal{G}}(\tilde{G}^{r+1}, \tilde{G}^r)$ is a surface floor. Since $b$ commutes with a nontrivial element in $\tilde{G}^r$ and this element has a unique fixed vertex (the vertex stabilized by $\tilde{G}^r$) with respect to the action that corresponds to the graph of groups of this floor, we must have that $b$ fixes the same vertex and thus it belongs to $G^r$. The result now follows from the induction hypothesis.
    \item Suppose $\tilde{\mathcal{G}}(\tilde{G}^{r+1}, \tilde{G}^r)$ is an Abelian floor. Suppose $b$ commutes with $a$ for some nontrivial element $a$ in $H\subset G^r$. Consider the action that corresponds to the graph of groups of this floor. If $a$ fixes only the vertex that is stabilized by $G^r$, then the result follows by the induction hypothesis. If not, $a$ must fix an edge and since it belongs to $H$, the peg of the floor can be taken to belong to $H$. In particular, since the free Abelian group that corresponds to the Abelian floor is maximal Abelian in $\tilde{G}^{r+1}$, $b$ belongs to this group.    
\end{itemize}
\end{proof}

\begin{proposition}\label{AbCase}
Let $x$ be a variable in sort $S_E$ and $\phi(x)$ be a first-order formula in $\mathcal{L}^{eq}$. We consider the formula 
$$\theta(\bar{z}):=\exists x \bigl(\phi(x)\land R_E(x, \bar{z})\bigr)$$ 
where $R_E(x,\bar z)$ is the elimination relation for $E$ and $\bar{z}$ is a tuple in the basic sorts. Let  $\psi(\bar{y})$ be the real formula corresponding to $\theta(\bar z)$
$$\F^{eq}\models \forall\bar{y}\bigl(\theta(f_{E_1}(\bar{y}_1), f_{E_2}(\bar{y}_2), \ldots, f_{E_\ell}(\bar{y}_{\ell}))\leftrightarrow\psi(\bar{y})\bigr)$$
where the tuple of variables, $\bar{y}:=\bar{y}_1,\bar{y}_2,\ldots,\bar{y}_{\ell}$, in the real sort  corresponds to the tuple  $z_1,\ldots,z_{\ell}$ of variables in the basic imaginary sorts.

Let $\mathcal{T}_1(G_1, \F), \ldots, \mathcal{T}_n(G_n,\F)$ be the towers in a Diophantine envelope of $\psi(\bar{y})$. Let, for $i\leq n$, $G_i:=\langle \bar{u}_i, \bar y, \bar a \ | \ \Sigma_i(\bar{u}_i, \bar y, \bar a)\rangle$ and $H_i\leq G_i$ be the Abelian pouch. 

Suppose, for each $i\leq n$, the tuple of elements $\bar{y}_1, \ldots, \bar y_\ell$ in $G_i^{m_{i1}}\times\ldots\times G_i^{m_i\ell}$ depends on the Abelian pouch $H_i$ of $\mathcal{T}_i$.

Then $\phi(\F^{eq})$ cannot be given definably the structure of an infinite field in $\F^{eq}$. 

\end{proposition}

\begin{proof}
We will show that $\phi(x)$ is internal to the family of centralizers of nontrivial elements, which we call $\Sigma$. Since, by Corollary \ref{InvariantOneBased}, this is a $\emptyset$-invariant family of one-based sets, we get, by \cite[Corollary 12]{MR2039343}, that $\phi(x)$ is one-based  and consequently it cannot be the domain of an infinite field. 

Since, by Lemma \ref{GeometricElimination}, any solution $b$ of $\phi(x)$ in $\mathbb{M}^{eq}$ is in the definable closure of the solution set of the elimination relation $R_E(b,\bar{z})$, say $\bar b_1,\ldots,\bar b_k$, it will be enough to prove that $\theta(\bar{z})$ is internal to $\Sigma$. In addition, we may assume that no tuple in the solution set of $\theta(\bar{z})$ contains a trivial element (in any sort). Indeed, as there are only finitely many choices for trivial elements in a fixed length tuple, we may treat each such case  separately, e.g. remove $\theta(1,z_2,\ldots, z_\ell), \theta(z_1, 1, z_3, \ldots, z_\ell), \ldots, \theta(1,1, z_3, \ldots, z_\ell)$ etc from $\theta(z_1, \ldots, z_\ell)$ and do the same for each such set.  Equivalently we may see $\theta(\bar{z})$ as a disjoint union of first-order formulas. Note that if each formula in a finite family of formulas is internal to $\Sigma$, then their disjunction is. In addition, we may choose the Diophantine envelope for each such formula identical to the Diophantine envelope for $\psi(\bar{y})$ with the exception of the distinguished tuple, i.e. we remove the tuple of variables that corresponds to an identity element. In particular, the hypothesis of dependence (for the tuples in the envelope of the corresponding real formula) still holds for each formula in the disjunction.  Now, each solution $\bar b_j:=b_{j1}\ldots b_{j\ell}$, for $j\le k$, is in the definable closure  of the corresponding real tuple $\bar{c}_1,\ldots,\bar{c}_\ell$, i.e. the tuple such that $f_{E_i}(\bar{c}_i)=b_{ji}$, for $i\leq \ell$, we only need to show that $\psi(\bar{y})$ is internal to $\Sigma$.

We fix a tower $\mathcal{T}_i(G_i,\F)=\mathcal{T}(G, \F)$, with $G:=\langle \bar{u}, \bar{y}, \bar{a} \ | \ \Sigma(\bar{u},\bar{y},\bar{a})\rangle$ and its Abelian pouch $H$. We will show that, there is a "parameter set" $B\subset \mathbb{M}$, that contains the parameters of the formula $\psi(\bar{y})$, such that for each solution $\bar{c}$ of $D(\bar{y}):=\exists \bar{u} \bigl(\Sigma(\bar{u},\bar{y},\bar{a})=1\bigr)$ in $\mathbb{M}$, we get that $\bar c$ is in the definable closure of $B$ together with solutions of centralizers in $\Sigma$ based on $B$, i.e. definable over $B$. 

The tuple $\bar{y}=\bar{y}_1,\ldots,\bar{y}_\ell$, by the hypothesis, depends on the Abelian pouch $H$ of $\mathcal{T}$. We claim that we may change the distinguished tuple, $\bar y$, with the tuple, $\bar w\in H$, it depends on. Indeed, we argue that if for a tuple, $\bar b$, we have $\F^{eq}\models\theta(\bar b)$ and this is witnessed by a morphism $h:G\rightarrow \F$ (that factors through the tower), i.e. $f_{E_1}(h(\bar y_1))=b_1, \ldots, f_{E_\ell}(h(\bar y_\ell))=b_\ell$, then $h$ gives to $\bar w$ the same values modulo the corresponding equivalence relations. We take cases according to the equivalence relation $\bar y_i$ corresponds to. For clarity we first tackle the case of the commuting conjugation relation. We then move to the elementary relations and argue that they are enough. 
\begin{itemize}
    \item Suppose $\bar y_i= y_i$ is a variable of sort $E_1$ - the commuting conjugation relation. Let $w$ be an element in the pouch, $H$, such that $[y_i^g,w]=1$ for some $g$ in $G$. By our hypothesis for $\theta(\bar z)$ we have that $h(y_i)\neq 1$. Since $w$ commutes with $y_i^g$ we have, by Lemma \ref{AbelianFloorDep}, that $y_i^g$ either belongs to $H$ or to an Abelian floor whose peg is generated by a root of $w$. If $y_i^g$ belongs to $H$, then we take $w=y_i^g$ and we are done. If not, then, by Remark \ref{NonDegenenaration}, the morphism $h$ cannot kill $w$ (otherwise it must kill its root and consequently $y_i^g$ and thus $y_i$) and in particular $E_1(h(\gamma_i), h(w))$ as we wanted.
    \item Suppose $(y_j, y_r)$ is a tuple of sort $_cE_2$ - the cosets with conjugation relation. Let $(w_1, w_2)$ be a tuple in the pouch, $H$, such that $[y_r^g,w_2]=1$ and $y_j^gC_G(y_r^g)=w_1C_G(w_2)$, for some $g$ in $G$. By our hypothesis for $\theta(\bar z)$ we have that $h(y_r)\neq 1$. Since $w_2$ commutes with $y_r^g$ we have, by Lemma \ref{AbelianFloorDep}, that either $y_r^g$ belongs to $H$ (and in this case we are done) or $y_r^g$ belongs to an Abelian floor whose peg is generated by a root of $w_2$. In the latter case, by Remark \ref{NonDegenenaration}, the morphism $h$ cannot kill $w_2$ (otherwise it must kill its root and consequently $y_r$). Therefore, $_cE_2((h(y_j),h(y_r)), (h(w_1), h(w_2)))$ as we wanted. 
    \item Suppose $(y_k, y_j, y_r)$ is a tuple of sort $_cE_3$ - the double cosets with conjugation relation. Let $(w_1, w_2, w_3)$ be a tuple in the pouch, $H$, such that $[y_k^g,w_1]=1$ and $[y_r^g,w_3]=1$ and $C_G(y_k^g)y_j^gC_G(y_r^g)=C_G(w_1)w_2C_G(w_3)$, for some $g$ in $G$. By our hypothesis for $\theta(\bar z)$ we have that $h(y_r)\neq 1$ and $h(y_k)\neq 1$. Since $w_1$ commutes with $y_k^g$ and $w_3$ commutes with $y_r^g$ we have, by Lemma \ref{AbelianFloorDep}, that either $y_k$ belongs to $H$ or $y_k$ belongs to an Abelian floor whose peg is generated by a root of $w_1$. Similarly for $y_r$ and $w_3$. In any combination of cases, by Remark \ref{NonDegenenaration}, the morphism $h$ cannot kill neither $w_1$ nor $w_3$. Therefore, $_cE_3((h(y_k),h(y_j),h(y_r)), (h(w_1), h(w_2), h(w_3)))$ as we wanted.
\end{itemize}

We now move to any basic relation. Then main point as seen in the above cases is that, for each elementary relation, with the exception of the trivial class, the relation is positively existentially definable. But any morphism respects a positive existential formula and in addition, by Remark \ref{NonDegenenaration} and Lemma \ref{AbelianFloorDep}, the corresponding class of the image of the tuple $w$ is not the trivial class. The above still holds for any refinement that is still positively existentially definable when we exclude the trivial class.      

Since $H$ admits a graph of groups decomposition, where the underlying graph is a star, the central vertex group is $\F$ and the rest of the vertex groups are free Abelian $\Z\oplus\Z^{m_l}$, for some $l<\omega$, we can give each element in the tuple $\bar{y}$ a normal form: 
$$c_1d_1c_2 \ldots d_mc_{m+1}$$
where each $c_i$, for $i\leq m+1$ is in $\F$ and each $d_i$, for $i\leq m$, is in some free Abelian group. We fix an element in the tuple $\bar y$, which we call $y$, and we consider its normal form as a word $w=w(c_1,d_1,c_2,\ldots, d_m,c_{m+1})$, where each $d_i$ commutes with some element $b_i$ in $\F$.  Hence, we have:

$$\F\models \forall \bar y\Bigl(D(\bar{y})\rightarrow \bigl(\exists d_1,\ldots,d_m (\bigwedge\limits_{i=1}^m [d_i,b_i]=1\land \bar{y}=\bar{w}(c_1,d_1,c_2,\ldots,d_m,c_{m+1})\bigr)\Bigr) $$

And since $\mathbb{M}$ is elementarily equivalent to $\F$, the same is true in $\mathbb{M}$. 

We choose $B:=\F$ and by the above formula we can easily check that each solution of $D(\bar{y})$ in $\mathbb{M}$ is in the definable closure of $\F$ together with elements in the centralizers $C_{\mathbb{M}}(b_i)$, for $i\leq m$, which are based on $\F$ (since they are definable over $\F$).  

Finally, since this holds for every tower in the envelope, the proposition is proved.
\end{proof}


\subsection{Hyperbolic case}

In this subsection we tackle the case of a definable set that its Diophantine envelope contains a hyperbolic tower that witnesses an infinitude of solutions for the definable set. 

\begin{proposition}\label{HypCase}
Let $x$ be a variable in sort $S_E$ and $\phi(x)$ be a first-order formula in $\mathcal{L}^{eq}$. Let  $\psi(\bar{y})$ be the real formula corresponding to $\phi(x)$, i.e.  
$$\F^{eq}\models \forall\bar{y}\bigl(\phi(f_{E}(\bar{y}))\leftrightarrow\psi(\bar{y})\bigr)$$
Let $\mathcal{T}_1(G_1, \F), \ldots, \mathcal{T}_n(G_n,\F)$ be the towers in the Diophantine envelope of $\psi(\bar{y})$. Suppose $\mathcal{T}(G,\F):=\mathcal{T}_i(G_i,\F)$, for some $i\leq n$, is hyperbolic and let $G:=\langle \bar{u}, \bar y, \bar a \ | \ \Sigma(\bar{u}, \bar y, \bar a)\rangle$. Suppose for some test sequence, $(h_n)_{n<\omega}:G\rightarrow \F$, of $\mathcal{T}$, the set of $E$-classes $\{ ([h_n(\bar{y})]_E)_{n<\omega}\}$, is infinite.    

Then $\phi(\F^{eq})$ cannot be given definably the structure of an Abelian group in $\F^{eq}$. 
\end{proposition}
\begin{proof}
Suppose, for the sake of contradiction, that $\phi(x)$ can be given definably an Abelian group structure and $\alpha(x,y,z)$ is the formula that defines the graph of the Abelian group operation on $\phi(\F^{eq})$. 
By Theorem \ref{Elim}, the elimination relation, $R_E$, for $E$, gives a bounded number, say $m$, of $\ell$-tuples in the basic sorts for any element in the sort $S_E$.  

We take an $N$-multiplet (with the natural order of floors), $\mathcal{T}^N(G_N, \F)$, of the hyperbolic tower $\mathcal{T}$, where $N$ is a number such that $\lceil\frac{\lfloor (N-1)/2\rfloor}{2}\rceil$ is larger than $m\cdot \ell+1$. The (hyperbolic) tower $\mathcal{T}^N$ can be seen as a star of isomorphic groups, where all factor subgroups are isomorphic to $G$ and the common subgroup is the coefficient group $\F$. Then, $G_N:=G_N(\bar U_1, \bar Y_1, \ldots, \bar U_N, \bar Y_N, \bar a)$, admits a generating set  where each $\bar Y_i$, for $i\leq N$, is a copy of the $\bar y$ tuple in the corresponding factor subgroup. 

We consider the formula $\alpha_N(x_1, x_2, \ldots, x_N, z)$ 
which defines the $N$-summation of elements in $\phi(\F^{eq})$ and the formula $\beta(x_1, x_2, \ldots, x_N, z_1, z_2, \ldots, z_\ell):=\exists z \bigl(\alpha_N(x_1,$  $x_2, \ldots, x_N, z)\land R_E(z, z_1, \ldots, z_\ell)\bigr)$, that assigns to the sum a tuple in the basic sorts. Note that
$$ \F^{eq}\models\forall \bar x \exists^{<\infty}\bar z \bigl(\beta(x_1, x_2, \ldots, x_N, z_1, z_2, \ldots, z_\ell)\bigr)$$ 


We note that, by our hypothesis, we may assume that $[h_n(\bar y)]_E$ are pairwise not equal. We will construct a test sequence, $(g_n)_{n<\omega}$, for $\mathcal{T}^N$ as follows: the restriction of $(g_n)_{n<\omega}$ to the first copy of $\mathcal{T}$ is $(h_n)_{n<\omega}$, the restriction of $(g_n)_{n<\omega}$ to the second copy of $\mathcal{T}$ in $\mathcal{T}^N$ is a refinement of $(h_n)_{n<\omega}$ whose growth dominates the growth of $(h_n)_{n<\omega}$, and likewise for the rest of the copies. In this way we make sure that each $\bar Y_i$ is given infinitely many values up to $E$-equivalence. Note that $\mathcal{T}^N$ is a hyperbolic tower. We apply Corollary \ref{HyperImplicit} to $\mathcal{T}^N$, the test sequence $(g_n)_{n<\omega}$, and the formula $\beta(x_1, x_2, \ldots, x_N, z_1, z_2, \ldots, z_\ell)$. 
Hence we obtain a formal solution, which is a tuple of elements $\bar{w}=(\bar{w}_1, \bar{w}_2, \ldots, \bar{w}_\ell)$ in $\mathcal{T}^N$, such that $\bar w_i$ corresponds to the basic sort of $z_i$, for $i\leq \ell$, with the property that for any test sequence $(f_n)_{n<\omega}:G_N\rightarrow \F$, 
$$\F^{eq}\models \beta([f_n(\bar{Y}_1)]_E, \ldots, [f_n(\bar{Y}_N)]_E, [f_n(\bar{w}_1)]_{E_{i_1}}, \ldots, [f_n(\bar{w}_\ell)]_{E_{i_\ell}})$$
for all but finitely many $n$. 

We first prove that we can find a formal solution $\bar w$, such that for at least one tuple, say $\bar w_j$, for $j\leq \ell$, in $\bar{w}=(\bar{w}_1, \bar{w}_2, \ldots, \bar{w}_\ell)$, its corresponding  equivalence class does not depend on $\F$. We assume, for a contradiction, not. Hence, there exists a fixed $\ell$-tuple, say $\bar b_1$, in $\F^{eq}$ such that for any test sequence, $(f_n)_{n<\omega}$, of 
$\mathcal{T}^N$, and for all $n$ large enough (the $n$ depends on the sequence), we have $R_E(\alpha_N([f_n(\bar{Y}_1)]_E, \ldots, [f_n(\bar{Y}_N)]_E, z),\bar b_1)$. We now define the following formula.

$$\beta_1(x_1, x_2, \ldots, x_N, z_1, z_2, \ldots, z_\ell):=\beta(x_1, x_2, \ldots, x_N, z_1, z_2, \ldots, z_\ell)\land \bar z\neq \bar b_1$$

We re-apply, Corollary \ref{HyperImplicit} for the formula $\beta_1$ and the same test sequence. Either we will get some formal solution $\bar w'$ such that one of its tuples does not depend on $\F$ or there exists a fixed $\ell$-tuple $\bar b_2\neq \bar b_1$, in $\F^{eq}$ such that for any test sequence, $(f_n)_{n<\omega}$, of 
$\mathcal{T}^N$, and for all $n$ large enough (the $n$ depends on the sequence), we have $R_E(\alpha_N([f_n(\bar{Y}_1)]_E, \ldots, [f_n(\bar{Y}_N)]_E, z),\bar b_2)$. We likewise define, for $j<m$, formulas: 

$$\beta_j(x_1, x_2, \ldots, x_N, z_1, z_2, \ldots, z_\ell):=\beta(x_1, x_2, \ldots, x_N, z_1, z_2, \ldots, z_\ell)\land \bar z\neq \bar b_1\land\ldots\land\bar z\neq \bar b_j $$ 

Suppose now, for the sake of contradiction, that at no step we can find a formal solution that does not depend on $\F$ and so we have produced distinct $\bar b_1, \bar b_2, \ldots, \bar b_m$ such that for any test sequence, $(f_n)_{n<\omega}$, of 
$\mathcal{T}^N$, and for all $n$ large enough (the $n$ depends on the sequence), we have $R_E(\alpha_N([f_n(\bar{Y}_1)]_E, \ldots, [f_n(\bar{Y}_N)]_E, z),\bar b_1)\land \ldots \land R_E(\alpha_N([f_n(\bar{Y}_1)]_E, \ldots, [f_n(\bar{Y}_N)]_E, z),\bar b_m)$. Since $R_E(a, \bar z)$ contains at most $m$ $\ell$-tuples, for any $a\in S_E(\F)$, we get that for any test sequence, $(f_n)_{n<\omega}$, of 
$\mathcal{T}^N$, and for all $n$ large enough (the $n$ depends on the sequence) the sum $\alpha_N([f_n(\bar{Y}_1)]_E, \ldots, [f_n(\bar{Y}_N)]_E, z)$ is fixed and equal to the unique element that satisfies the formula $R_E(z,\bar b_1)\land \ldots \land R_E(z,\bar b_m)$. In order to obtain a contradiction we refine the restriction of $(g_n)_{n<\omega}$ to the last copy of $\mathcal{T}$ in $\mathcal{T}^N$. This refined sequence, $(g'_n)_{n<\omega}$, will give different values than $(g_n)_{n<\omega}$, up to $E$-equivalence, to $\bar Y_N$, but the same values to $\bar Y_1, \ldots, \bar Y_{N-1}$, for all $n<\omega$. In particular, for some large enough $n$, we will get that the sum of $[g_n(\bar{Y}_1)]_E, \ldots, [g_n(\bar{Y}_N)]_E$ is equal to the sum of $[g'_n(\bar{Y}_1)]_E, \ldots, [g'_n(\bar{Y}_N)]_E$ but $[g_n(\bar{Y}_N)]_E\neq [g'_n(\bar{Y}_N)]_E$, a contradiction.


Finally, we see the $N$-multiplet tower, $\mathcal{T}^N$, as a star of isomorphic groups with $N$ rays. Let $(g_n)_{n<\omega}:G_N\rightarrow\F$ be the test sequence we constructed in the previous paragraph and $\bar{w}$ a formal solution for which at least one tuple in it does not depend on $\F$. Without loss of generality the tuple $\bar w_1$ that corresponds to the basic equivalence relation $E_b$ is such. Let $G_N^i$ be the subgroup of $G_N$ generated by $\langle \bar U_i, \bar Y_i, \bar a\rangle$, for $i\leq N$. Each such subgroup has the structure of a tower and, as matter of fact, we may give it the same structure as $\mathcal{T}$. We may also assume that under $(g_n)_{n<\omega}$, the growth rate of $G_N^i$ dominates the growth rate of $G_N^{i-1}$. Since, each such subtower is independent from the rest, i.e. the order in which we construct them while constructing $\mathcal{T}^N$ can be chosen arbitrarily, we may permute the values of $(g_n)_{n<\omega}$ so that each time the subtowers are revealed in different order. For example if we exchange the values $g_n\upharpoonright_{G_N^1}$ with the corresponding values $g_n\upharpoonright_{G_N^2}$ and leave the rest identical, then this new sequence will be a test sequence for $\mathcal{T}^N$ that will reveal last the tower that corresponds to $G^2_N$, second last the tower that corresponds to $G^1_N$ etc. In particular, if $\sigma$ is a permutation in $S_N$, then we denote by $(g^{\sigma}_n)_{n<\omega}$ the test sequence that permutes the values of $(g_n)_{n<\omega}$ according to the permutation $\sigma$. We next consider the values of the tuple $\bar w_1:=\bar w_1(\bar U_1, \bar Y_1, \ldots, \bar U_N, \bar Y_N, \bar a)$ under the distinct test sequences $(g^{\sigma}_n)_{n<\omega}$ for $\sigma$ in $S_N$. It is actually the same as the image of $\sigma(\bar w_1)$, i.e. the tuple whose reduced form with respect to the star of isomorphic groups is a reduced form of $\bar w_1$ with the factor subgroups permuted according to $\sigma$,  under the test sequence $(g_n)_{n<\omega}$. 

Hence, since, by Abelianity, the $N$ summation of $[g^{\sigma}_n(\bar Y_1)]_E, \ldots, [g_n^{\sigma}(\bar Y_N)]_E$, for any $\sigma$, gives the same value as the $N$ summation of $[g_n(\bar Y_1)]_E, \ldots, [g_n(\bar Y_N)]_E$ we must have that $g_n(\sigma(\bar w_1))$ can have at most $m\cdot \ell$ different values up to the $E_b$-equivalence relation. We take cases according to the basic equivalence class $E_b$:
\begin{itemize}
    \item if $E_b$ is a refinement of $_cE_2$, i.e. the elementary relation of cosets with conjugation, then, by Lemma \ref{OrbitsCosets} and the choice of $N$, we have more than $m\cdot \ell$ distinct $E_2$-classes (and thus $E_b$-classes)  in the orbit of $\bar w_1$ under $S_N$, and by Proposition \ref{CosetDiff} taking also into consideration that the tower is hyperbolic, for $n$ large enough, they are all mapped by $g_n$ to distinct $_cE_2$-classes (and thus $E_b$-classes) in $\F$.
    \item if $E_b$ is a refinement of $_cE_3$, i.e. the elementary relation double cosets with conjugation, then using Lemma \ref{OrbitsCosets} and Proposition \ref{DCosetDiff} gives the same result as the previous point for $_cE_3$-classes and thus for $E_b$-classes.
\end{itemize}

\end{proof}

\subsection{General case}
In this final subsection we will prove the main result of the paper, namely no infinite field is interpretable in the first-order theory of the free group. 

\begin{theorem}\label{MainTheorem}
The first-order theory of the free group does not interpret an infinite field.
\end{theorem}

\begin{proof}
Suppose, for the sake of contradiction, that $\phi(x)$ can be given definably an Abelian group structure and $\alpha(x,y,z)$ is the formula that defines the graph of the Abelian group operation on $\phi(\F^{eq})$. 
By Theorem \ref{Elim}, the elimination relation, $R_E$, for $E$ gives a bounded number, say $m$, of $\ell$-tuples in the basic sorts for any element in the sort $S_E$.

 
Now, since every tuple of sort $E$, has at most $m$ images under $R_E$, we may decompose $\phi(x)$ in a disjoint union of $m$ formulas using the formulas $\theta_1(\bar z_1), \theta_2(\bar z_1, \bar z_2), \ldots, \theta_m(\bar z_1, \bar z_2, \ldots, \bar z_m)$ that we now define: The formula $\theta_i$, $i\leq m$, is satisfied by a tuple consisting of the $i$-many distinct solutions of $R_E(a,\bar z)$ in $\F^{eq}$, such that $\F\models \phi(a)$ and $|R_E(a, \F^{eq})|=i$. It is obvious that for each $a$ that satisfies $\phi(x)$ in $\F^{eq}$, there exists a (unique) $i\leq m$ such that $\theta_i$ is realized by $(\bar \gamma_1, \ldots, \bar \gamma_i)$ with 

$$\F^{eq}\models R_E(a, \bar \gamma_1)\land\ldots\land R_E(a, \bar \gamma_i)\bigwedge_{k<r\leq i} \bar \gamma_r\neq \bar \gamma_k\bigwedge \forall \bar z (R_E(a,\bar z)\rightarrow \bigvee_{j\leq i}\bar z=\bar \gamma_j)$$

For each $i\leq m$, consider the real formula $\psi_i(\bar y_1, \ldots \bar y_i)$, where $\bar y_j:=(\bar y_j^1, \ldots, \bar y_j^{\ell})$ that corresponds to $\theta_i(\bar z_1, \ldots, \bar z_i)$:

$$\F^{eq}\models\forall\bar{y}_1,\ldots\bar y_i\bigl(\theta_i(f_{E_1}(\bar{y}^1_1), \ldots, f_{E_{\ell}}(\bar y^\ell_1), \ldots,f_{E_1}(\bar{y}^1_i), \ldots, f_{E_{\ell}}(\bar y^\ell_i) )\leftrightarrow\psi_i(\bar{y}_1, \ldots, \bar y_i)\bigr)$$

By Fact \ref{Envelope}, for each $i\leq m$, there exist finitely many towers $\mathcal{T}^i_1(G^i_1,\F), \ldots, \mathcal{T}^i_{k_i}(G^i_{k_i},\F)$, which form the Diophantine Envelope of $\psi_i(\bar{y}_1, \ldots,\bar y_i)$, 
and by Proposition \ref{AbCase} we may assume that for at least one tower, say $\mathcal{T}(G, \F):=\mathcal{T}_r^p(G_r^p,\F)$ for some $p\leq m$ and $r\leq k_p$, with $G:=G(\bar u, \bar y, \bar a)$, the tuple of elements $\bar{y}$ does not depend on the Abelian pouch of $\mathcal{T}$. We choose $\bar y_j=(\bar y_j^1, \ldots, \bar y_j^\ell)$, a sub-tuple of $\bar y$, that "represents" some tuple $\bar z_j$ in   $\theta_p(\bar{z}_1, \ldots, \bar {z}_p)$ and does not depend on the Abelian pouch. In particular, by Proposition \ref{finGenERPouch} and Lemmas \ref{finiteConj}, \ref{finCosetsPouch}, \ref{finDCosetsPouch}, and \ref{finDTranslatesPouch} for any test sequence, $(h_n)_{n<\omega}$, of $\mathcal{T}$, the images of $(\bar y^1_j, \ldots, \bar y^\ell_j)$ under $(h_n)_{n<\omega}$ take infinitely many values up to the corresponding equivalence relations.   
We fix $N$ such that $\lceil\frac{\lfloor (N-1)/2\rfloor}{2}\rceil$ is larger than $m\cdot \ell+1$,  
and we consider the formula $\alpha_N(x_1, x_2, \ldots, x_N, z)$ in $\mathcal{L}^{eq}$ that defines the summation of $N$ elements of $\phi_p(x)$, where $\phi_p(x)$ defines the subset of $\phi(x)$ that consists of the elements that have exactly $p$ images under $R_E$. By our previous observation $\phi_p(x)$ has infinitely many solutions in $\F^{eq}$. For convenience of notation, we will denote by $R_E^p(x, \bar z_1, \ldots, \bar z_p)$ the formula 
$$R_E(x,\bar z_1)\land R_E(x, \bar z_2)\land \ldots\land R_E(x, \bar z_p)\bigwedge_{i<j\leq p}\bar z_i\neq \bar z_j$$ 
and by $R(x, \bar z_1, \ldots, \bar z_m)$ the formula 
$$\bigwedge_{i\leq m} R_E(x,\bar z_i)\bigwedge \Bigl[(R_E(x,\bar z_1)\land |R_E(x,\bar z)|=1)\lor (R_E^2(x,\bar z_1, \bar z_2)\land |R_E(x,\bar z)|=2) \lor \ldots \lor $$
$$\lor (R_E^m(x,\bar z_1, \bar z_2,\ldots, \bar z_m)\land |R_E(x,\bar z)|=m)\Bigr]$$ 
Since, every $a$ in the sort $S_E$ has at most $m$ images under $R_E$, we get that $\F^{eq}\models\forall x\exists \bar z_1, \ldots, \bar z_m$ $\bigl( R(x, \bar z_1, \ldots, \bar z_m)\bigr)$, but even more importantly $\F^{eq}\models \forall x_1, x_2, \bar z_1, \ldots, \bar z_m \bigl(R(x_1,\bar z_1, \ldots, \bar z_m) \land R(x_2,\bar z_1, \ldots, \bar z_m) \rightarrow x_1=x_2\bigr)$. We change into basic sorts with the following formula: 
$$\zeta(\bar{z}^1, \ldots, \bar{z}^N, \bar{s}):=\exists x_1, x_2, \ldots, x_N, s\bigl(\alpha_N(x_1, x_2, \ldots, x_N, s)\bigwedge_{i\leq N} R_E^p(x_i, \bar{z}^i_1, \ldots, \bar z^i_p)\land R(s, \bar{s}_1, \ldots, \bar s_m)\bigr)$$

We note that $\forall \bar z^1, \bar z^2, \ldots, \bar z^N\exists^{<\infty}\bar s\bigl(\zeta(\bar z^1, \bar z^2, \ldots, \bar z^N, \bar s)\bigr)$. We now take an $N$-multiplet, $\mathcal{T}^N(G_N,\F)$, with $G_N:=G_N(\bar U^1, \bar Y^1, \ldots, \bar U^N, \bar Y^N, \bar a)$, of the tower $\mathcal{T}$. The $\bar Y^i$'s are copies of $\bar y$ of the tower $\mathcal T$ and belong to the corresponding "copy" of $\mathcal{T}$ in $\mathcal{T}^N$.  
Recall that there exists a test sequence, $(h_n)_{n<\omega}:G\rightarrow\F$, of $\mathcal{T}$, such that $h_n(\bar y)$ satisfies $\psi_p(\bar y)$ in $\F$, for all $n<\omega$. Moreover, we may refine this test sequence so that the images of $\bar y$ under $(h_n)_{n<\omega}$ are pairwise distinct up to the corresponding equivalence relations. We further refine $(h_n)_{n<\omega}$ as follows. Let $Cl(\mathcal{T})$ be the closure of $\mathcal{T}$ provided by Corollary \ref{FixCosetTest}. For each Abelian floor $\mathcal{G}(G^{i+1}, G^i)$, with free Abelian group $\Z\oplus\Z^{m_i}$, of $\mathcal{T}$, there is a closure, $f:\Z\oplus\Z^{m_i}\rightarrow\Z\oplus A^{m_i}$, in $Cl(\mathcal{T})$. For each such floor we may choose the powers $(k_{1}(n), \ldots, k_{m_i}(n))$, of specializations via $(h_n)_{n<\omega}$, $(a(n), a(n)^{k_{1}(n)}, \ldots, a(n)^{k_{m_i}(n)})$, of the free Abelian group in a way that, for every $n<\omega$, they all belong to the same coset $(\mu_1, \ldots, \mu_{m_i})+\Delta_i$, for $\Delta_i$ a finite index subgroup of $Z^{m_i}$.     

In addition, we may extend $(h_n)_{n<\omega}$ to a test sequence, $(g_n)_{n<\omega}:G_N\rightarrow\F$, such that $g_n(\bar Y_i)$ satisfies $\psi_p(\bar y)$ in $\F$, for any $i\leq N$ and all $n<\omega$. Indeed, $(g_n)_{n<\omega}$ can be defined as follows: Its restriction to the first copy of $\mathcal{T}$ is  $(h_n)_{n<\omega}$, its restriction to the second copy is a refinement of $(h_n)_{n<\omega}$ in a way that the growth rate conditions are met, and likewise for the rest of the copies. 

We, thus, may apply Theorem \ref{Implicit} to $\mathcal{T}^N$, the formula $\zeta(\bar z_1, \bar z_2, \ldots, \bar z_N, \bar s)$ and the test sequence $(g_n)_{n<\omega}$. We get a closure, $Cl(\mathcal{T}^N)$, of $\mathcal{T}^N$, with $Cl(G_N):=Cl(G_N)(\bar V^1, \bar U^1, \bar Y^1, $ $\ldots, \bar V^N, \bar U^N, \bar Y^N, \bar V^0$ $\bar a)$ for which a subsequence of $(g_n)_{n<\omega}$ extends to a test sequence, still denoted $(g_n)_{n<\omega}$,  and a formal solution, $\bar w (\bar V, \bar U_1, \bar Y_1, \ldots, \bar U_N, \bar Y_N,$ $\bar a)$, for $\bar s$,  such that: 
$$\F^{eq}\models\zeta\bigl([g_n(\bar Y^1_1)],\ldots,[g_n(\bar Y^1_p)],\ldots, [g_n(\bar Y^N_1)], \ldots, [g_n(\bar Y^N_p)], [g_n(\bar w_1)], \ldots, [g_n(\bar w_m)]\bigr)$$

Where by $[\bar b]$ we denote the tuple of basic equivalent classes $[\bar b_1]_{E_1}, \ldots, [\bar b_\ell]_{E_{\ell}}$ that correspond to the sorts of the last $\ell$ arguments of the relation $R_E(x, z_1, \ldots, z_{\ell})$. 
\\ \\
{\bf Claim:} The formal solution $\bar w$ does not depend on the Abelian pouch of $Cl(\mathcal{T}^N)$. \\ \\
{\bf Proof of claim:}
Indeed, suppose for the sake of contradiction, that it does. Recall that $\{[g_n(\bar Y^N)] \ | \ n<\omega\}$ is infinite. 
We modify $(g_n)_{n<\omega}$ in order to obtain a new test sequence as follows. The sequence $(g'_n)_{n<\omega}$ agrees with $(g_n)_{n<\omega}$ on the Abelian pouch of $Cl(\mathcal{T}^N)$ and all closures of copies of $\mathcal{T}$ but the last. On the closure of the last copy we use Facts \ref{FactSela} and \ref{BasisFlex} to change the values of $(g_n)_{n<\omega}$ so that the equivalence classes of the images of $\bar Y^N$ still belong to the definable set.  We may assume that for infinitely many $n$, $g_n(\bar Y^N)\neq g'_n(\bar Y^N)$ (if not refine the restriction of $(g'_n)_{n<\omega}$ to the appropriate part of the tower). The sum of the $N$ elements that correspond to $g_n(\bar Y_1), \ldots, g_n(\bar Y_N)$ and the sum of the $N$ elements that correspond to $g'_n(\bar Y_1), \ldots, g'_n(\bar Y_N)$ will be the same for all but finitely many $n<\omega$. Indeed, the two sequences agree in the Abelian pouch of $Cl(\mathcal{T}^N)$. On the other hand, for infinitely many $n<\omega$, $[g_n(\bar Y^N)]\neq [g'_n(\bar Y^N)]$, a contradiction. \qed  
\\

For each Abelian floor $\mathcal{G}(G^{i+1}, G^i)$, with free Abelian group $\Z\oplus\Z^{m_i}$, of $\mathcal{T}$, which is higher than the Abelian pouch, we have $N$ copies of it in the tower $\mathcal{T}^N$. To each such copy corresponds a closure, $f_j:\Z\oplus\Z^{m_i}\rightarrow\Z\oplus A^{m_i}_j$, for $j\leq N$, in $Cl(\mathcal{T}^N)$. Recall that we extended $(h_n)_{n<\omega}$, a test sequence for $\mathcal{T}$, to a test sequence, $(g_n)_{n<\omega}$, for $\mathcal{T}^N$, in a way that for each Abelian floor the powers, $(k_{1,j}(n), \ldots, k_{m_i,j}(n))$, of specializations via $(g_n)_{n<\omega}$, $(a_j(n), a_j(n)^{k_{1,j}(n)}, \ldots, a_j(n)^{k_{m_i,j}(n)})$, of the free Abelian groups, for $j\leq N$, all belong to the same coset $(\mu_1, \ldots, \mu_{m_i})+\Delta_i$, for $\Delta_i$ a finite index subgroup of $Z^{m_i}$. Hence, since a subsequence of $(g_n)_{n<\omega}$ extends to a test sequence of the closure $Cl(\mathcal{T}^N)$, the intersection of $(\mu_{1,j}, \ldots, \mu_{m_i,j})+\Delta_{i,j}$, for $j\leq N$, of cosets that correspond to the closures $\Z\oplus A^{m_i}_j$ defines a non-empty coset $(\nu_1, \ldots, \nu_{m_i})+\bigcap_{j\leq N}\Delta_{i,j}$, for $\bigcap_{j\leq N}\Delta_{i,j}$ a finite index subgroup of $Z^{m_i}$. Moreover, we may assume that $(\nu_1, \ldots, \nu_{m_i})+\bigcap_{j\leq N}\Delta_{i,j}\subseteq (\mu_1, \ldots, \mu_{m_i})+\Delta_i$. 

We now may choose, by Fact \ref{BasisFlex}, a test sequence, $(f_n)_{n<\omega}$, of $\mathcal{T}^N$ such that for each $i$ that corresponds to an Abelian floor of $\mathcal{T}$, which is higher than the Abelian pouch, and all $N$ copies of it in $\mathcal{T}^N$, the powers, $(k_{1,j}(n), \ldots, k_{m_i, j}(n))$, of specializations via $(f_n)_{n<\omega}$, $(b_j(n), b_j(n)^{k_{1,j}(n)}, \ldots, b_j(n)^{k_{m_i,j}(n)})$, for $j\leq N$, all belong to the coset $(\nu_1, \ldots, \nu_{m_i})+\bigcap_{j\leq N}\Delta_{i,j}$. In particular, all values extend to the corresponding closures in $Cl(\mathcal{T}^N)$ and in addition all images of $\bar Y_i$, for $i\leq N$, under $(f_n)_{n<\omega}$, correspond to $E$-classes that belong to $\phi_p(x)$. The reason for using a test sequence with these properties, instead of $(g_n)_{n<\omega}$, is that when we permute the values of $(f_n)_{n<\omega}$, "exchanging" the values it gives to different copies of $\mathcal{T}$, they still extend to a test sequence of $Cl(\mathcal{T}^N)$.  

The tower $Cl(\mathcal{T}^N)$ may be seen as a star of groups where the Abelian pouch of $Cl(\mathcal{T}^N)$ is the common subgroup. We note that in this case, since the corresponding Abelian floors in different copies of $\mathcal{T}$ might be assigned different closures, the factor subgroups are not necessarily isomorphic. We use Proposition \ref{SymmetricTowers} to "symmetrise" the closure $Cl(\mathcal{T}^N)$. We denote the symmetrised tower by $Sm(Cl(\mathcal{T}^N))$, with $Sm(Cl(G_N)):=Sm(Cl(G_N))(\bar C^1, \bar V^1, \bar U^1, \bar Y^1, $ $\ldots,\bar C^N, \bar V^N, \bar U^N, \bar Y^N, \bar V^0$ $\bar a)$ and will refer to it as the symmetric closure. We  express the tuple $\bar w$ in the generators of the symmetric closure, i.e. $\bar w:=\bar w(\bar C^1, \bar U^1, \bar Y^1, \ldots, \bar C^N, \bar U^N, \bar Y^N, \bar V^0$ $\bar a) $. We note that for this $\bar w$ and the extension of $(f_n)_{n<\omega}$ to $Sm(Cl(\mathcal{T}^N))$, still denoted by $(f_n)_{n<\omega}$, we have:   

$$\F^{eq}\models\zeta\bigl([f_n(\bar Y^1_1)],\ldots,[f_n(\bar Y^1_p)],\ldots, [f_n(\bar Y^N_1)], \ldots, [f_n(\bar Y^N_p)], [f_n(\bar w_1)], \ldots, [f_n(\bar w_m)]\bigr)$$

We finally consider the reduced form of $\bar w$ with respect to the star of groups decomposition of $Sm(Cl(\mathcal{T}^N))$. Recall that there exists a (sub-)tuple $w$ in $\bar w$ that does not depend on the Abelian pouch of the original $Cl(\mathcal{T}^N)$. We may assume that for infinitely many $n$, $f_n(w)$, belongs to the image of the sum of the elements that correspond to $f_n(\bar Y^1), \ldots, f_n(\bar Y^N)$. When we move to the symmetric closure we might make $w$ depend on the Abelian pouch. This is not a problem, because $w$ still does not belong to the Abelian pouch of the symmetric closure and, thus, we can still use Lemma \ref{OrbitsCosets} in order to give different values to the coarser equivalence class it belongs to.  

Suppose the words in $w$ have the following form (for simplicity of notation we only exhibit one of them) 

$$w=a_1(\bar C^{i_1}, \bar U^{i_1}, \bar Y^{i_1}, \bar V^0, \bar a)a_2(\bar C^{i_2}, \bar U^{i_2}, \bar Y^{i_2}, \bar V^0, \bar a)\ldots a_q(\bar C^{i_q}, \bar U^{i_q}, \bar Y^{i_q}, \bar V^0, \bar a)$$ 

where $i_j\in\{1,\ldots, N\}$ and $i_j\neq i_{j+1}$, for $j<q$. For any $\sigma$ in the symmetric group $S_N$, we consider the test sequence $(f_n^\sigma)_{n<\omega}$ whose values are identical with those of $(f_n)_{n<\omega}$ when restricted to the Abelian pouch and are a permutation of the values of $(f_n)_{n<\omega}$ according to $\sigma$ for the floors above the Abelian pouch. By the choice of $(f_n)_{n<\omega}$, every $(f_n^\sigma)_{n<\omega}$ extends to a test sequence of the symmetric closure $Sm(Cl(\mathcal{T}^N))$, and it permutes the values of $\bar Y^i$'s according to $\sigma$. Moreover, by the choice of the symmetric closure the value $f_n^\sigma(w)$ is equal to the value $f_n(w^\sigma)$, where $w^\sigma$ is: 

$$a_1(\bar C^{\sigma(i_1)}, \bar U^{\sigma(i_1)}, \bar Y^{\sigma(i_1)}, \sigma (\bar V^0), \bar a)a_2(\bar C^{\sigma(i_2)}, \bar U^{\sigma(i_2)}, \bar Y^{\sigma(i_2)}, \sigma(\bar V^0), \bar a)\ldots$$ $$a_q(\bar C^{\sigma(i_q)}, \bar U^{\sigma(i_q)}, \bar Y^{\sigma(i_q)}, \sigma(\bar V^0), \bar a)$$

We take cases according to which basic equivalence relation $w$ corresponds to. Recall that $w$ is either a couple or a triple. 

\begin{itemize}
    \item If $(w^1, w^2)$ represents a refinement of an $_cE_2$-class , then, by Lemma \ref{OrbitsCosets} and the choice of $N$, we have more than $m\cdot\ell$ distinct $_cE_2$-classes in the orbit of $(w^1,w^2)$ under $S_N$. 
    By Proposition \ref{CosetDiff}, for all large enough $n$, all distinct $(w^1, w^2)^\sigma$ are mapped by $f_n$ to distinct $_cE_2$-classes in $\F$. 
    \item If $(w^1, w^2, w^3)$ represents a refinement of an $_cE_3$-class, then using Lemma \ref{OrbitsCosets} and Proposition \ref{DCosetDiff} yields the same result as the previous point.
\end{itemize}

\end{proof}


\bibliography{biblio}
\ \\
Rizos Sklinos\\ \\
Academy of Mathematics \& Systems Science\\
Chinese Academy of Science\\
Room 1016 SiYuan Building\\
No.55 of Zhongguancun East Road\\
Haidian District, Beijing 100190\\
P. R. China

\end{document}